\documentclass[11pt]{article}
\usepackage{graphicx} 
\usepackage{quiver}
\usepackage{verbatim}
\usepackage{todonotes}
\usepackage{xcolor}

\usepackage{amsmath, amsfonts, amssymb, amsthm, mathtools}
\usepackage{thmtools}
\usepackage{hyperref}
\usepackage[margin=1in]{geometry}
\usepackage[shortlabels]{enumitem}
\usepackage[ruled]{algorithm2e}

\providecommand{\keywords}[1]
{
  \small	
  \textbf{\textit{Keywords---}} #1
}

\newcommand{\matt}[1]{\todo[inline,color=green!30]{MK: #1}}

\newcommand{\1}{\{*\}}

\newcommand{\Opt}{\mathsf{Opt}}

\newcommand{\Saddle}{\mathsf{Saddle}}
\newcommand{\DiffSaddle}{\mathsf{DiffSaddle}}

\newcommand{\DiffConv}{\mathsf{DiffConv}}
\newcommand{\Conv}{\mathsf{Conv}}
\newcommand{\Conc}{\mathsf{Conc}}
\newcommand{\DiffConc}{\mathsf{DiffConc}}

\newcommand{\Dynam}{\mathsf{Dynam}}

\newcommand{\NDD}{\mathsf{NDD}}
\newcommand{\FG}{\mathsf{FlowNet}}

\newcommand{\Set}{\mathrm{\textnormal{\textbf{Set}}}}
\newcommand{\Sets}{\mathrm{\textnormal{\textbf{Sets}}}}
\newcommand{\FinSet}{\mathrm{\textnormal{\textbf{FinSet}}}}

\newcommand{\Vect}{\mathrm{\textnormal{\textbf{Vect}}}}
\newcommand{\Cospan}{\mathrm{\textnormal{\textbf{Cospan}}}}
\newcommand{\UWD}{\mathrm{\textnormal{\textbf{UWD}}}}
\newcommand{\maps}{\colon}
\newcommand{\R}{\mathbb{R}}

\newcommand{\flow}{\mathsf{flow}}
\newcommand{\netflow}{\mathsf{netflow}}
\newcommand{\gad}{\mathsf{ga}\text{-}\mathsf{d}}
\newcommand{\gd}{\mathsf{gd}}
\newcommand{\Euler}{\mathsf{Euler}}
\newcommand{\pdsubg}{\mathsf{pd}\text{-}\mathsf{subg}}
\newcommand{\superg}{\mathsf{superg}}
\newcommand{\subg}{\mathsf{subg}}

\newcommand{\id}{\mathrm{id}}

\newcommand{\define}[1]{\textbf{#1}}

\newcommand{\CC}{\mathcal{C}}
\newcommand{\DD}{\mathcal{D}}
\newcommand{\OO}{\mathcal{O}}

\DeclareMathOperator*{\argmax}{arg\,max}
\DeclareMathOperator*{\argmin}{arg\,min}

\newtheorem{theorem}{Theorem}[section]
\newtheorem{corollary}{Corollary}[theorem]
\newtheorem{lemma}[theorem]{Lemma}

\theoremstyle{definition}
\newtheorem{definition}[theorem]{Definition}

\newtheorem{example}[theorem]{Example}
\newtheorem{remark}[theorem]{Remark}

\title{A Compositional Framework for First-Order Optimization}
\author{Tyler Hanks\and Matthew Klawonn\and Matthew Hale\and Evan Patterson\and James Fairbanks}
\date{ }

\begin{document}

\maketitle

\begin{abstract}

Optimization decomposition methods are a fundamental tool to develop distributed solution algorithms for large scale optimization problems arising in fields such as machine learning, optimal control, and operations research. In this paper, we present an algebraic framework for hierarchically composing optimization problems defined on hypergraphs and automatically generating distributed solution algorithms that respect the given hierarchical structure. 
The central abstractions of our framework are operads, operad algebras, and algebra morphisms, which formalize notions of syntax, semantics, and structure preserving semantic transformations respectively. These abstractions allow us to formally relate composite optimization problems to the distributed algorithms that solve them. Specifically, we show that certain classes of optimization problems form operad algebras, and a collection of first-order solution methods, namely gradient descent, Uzawa's algorithm  (also called gradient ascent-descent), and their subgradient variants, yield algebra morphisms from these problem algebras to algebras of dynamical systems. Primal and dual decomposition methods are then recovered by applying these morphisms to certain classes of composite problems.
Using this framework, we also derive a novel sufficient condition for when a problem defined by compositional data is solvable by a decomposition method. We show that the minimum cost network flow problem satisfies this condition, thereby allowing us to automatically derive a hierarchical dual decomposition algorithm for finding minimum cost flows on composite flow networks. We implement our operads, algebras, and algebra morphisms in a Julia package called AlgebraicOptimization.jl and use our implementation to empirically demonstrate that hierarchical dual decomposition outperforms standard dual decomposition on classes of flow networks with hierarchical structure.

\end{abstract}\hspace{10pt}


\section{Introduction}

Decomposition methods such as primal and dual decomposition are fundamental tools to develop distributed solution algorithms for large scale optimization problems arising in machine learning \cite{boyd_distributed_2010,parikh_block_2014}, optimal control \cite{giselsson_accelerated_2013,primal_decomp_mpc}, and operations research \cite{benders1962decomp, rahmaniani_benders_2017}. These methods generally work by splitting a large problem into several several simpler subproblems and repeatedly solving these to arive at a solution to the original problem. As such, \emph{decomposition} methods are most naturally applicable to problems which themselves are \emph{composites} of subproblems, for some appropriate notion of composition. We say that such problems have \emph{compositional structure}. Often, compositional structure at the problem level is left implicit, and expertise is required to determine whether a decomposition method can be applied. Additionally, it can be non-trivial to transform a given problem into an equivalent one which can be decomposed. 

On the other hand, when compositional problem structure is made explicit, decomposition methods can often be applied more directly. Common examples of problems with explicit compositional structure are those defined on graphs and hypergraphs, where nodes represent subproblems and (hyper)edges represent coupling between subproblems. The benefit of making compositional structure explicit is that by fixing the types of problems that can inhabit nodes and the types of coupling allowed between subproblems, one can derive a decomposition method which solves problems defined over arbitrary graphs or hypergraphs \cite{boyd_decomp_2015, hallac_snapvx_nodate, Jalving2022plasmo}. In this sense, problems over graphs and hypergraphs can be viewed as standard forms for problems to which decomposition is applicable. However, it can still be challenging to translate a given problem or the data defining a problem instance into such a standard form.

We are thus motivated to develop a general framework for decomposition methods that both explicitly represents compositional structure and offers a principled way to transform problem data into such explicit structures. This paper presents a foundational step towards achieving this goal by introducing a novel algebraic/structural perspective on first-order optimization decomposition methods. 

The central abstractions of this framework revolve around \emph{undirected wiring diagrams} (UWDs) which are a special class of hypergraphs with a distinguished boundary node. Undirected wiring diagrams form a syntactic structure called an \emph{operad} which allows the hierarchical construction of UWDs by recursive substitution \cite{spivak_operad_2013}. Algebras on the operad of undirected wiring diagrams, or UWD-algebras for short, provide formal semantic interpretations to UWDs that respect their hierarchical structure. Intuitively, a UWD-algebra specifies how to build interconnected systems where the pattern of interconnection can be represented by a UWD. Consequently, UWD-algebras have been used to define network-style composition operations for an array of combinatorial objects such as labelled graphs \cite{baez_structured_2022}, electrical circuits \cite{baez2018circuits}, and Petri networks \cite{baez_compositional_2017}, as well as quantitative objects such as linear and convex relations \cite{baez2015control, stein2023compositional}, Markov processes \cite{Baez_2016_markov}, and dynamical systems \cite{libkind_operadic_2022}. UWD-algebras also provide a natural abstraction for modeling optimization problems defined on hypergraphs, which is the approach we take in this paper.

UWD-algebra morphisms then provide structure preserving transformations between different semantic interpretations of UWDs. Using this machinery, the central idea of our framework can be summarized as follows. If a first-order optimization algorithm defines an algebra morphism from a UWD-algebra of optimization problems to a UWD-algebra of dynamical systems, that algorithm decomposes problems defined on arbitrary UWDs. Furthermore, applying such a morphism to a problem generates a dynamical system to solve the problem in a distributed fashion using message passing semantics.

Turning to specific examples of this general pattern, we first define a UWD-algebra of once-differentiable (denoted $C^1$), unconstrained minimization problems called $\Opt$ and show that gradient descent is an algebra morphism from $\Opt$ to the UWD-algebra $\Dynam_D$ of discrete dynamical systems, namely it is the morphsim
\begin{equation}
    \gd^\gamma\maps \Opt\to \Dynam_D,
\end{equation}
where $\gamma>0$ is a fixed step size. 

\begin{figure}
    \centering
    \includegraphics[width=\textwidth]{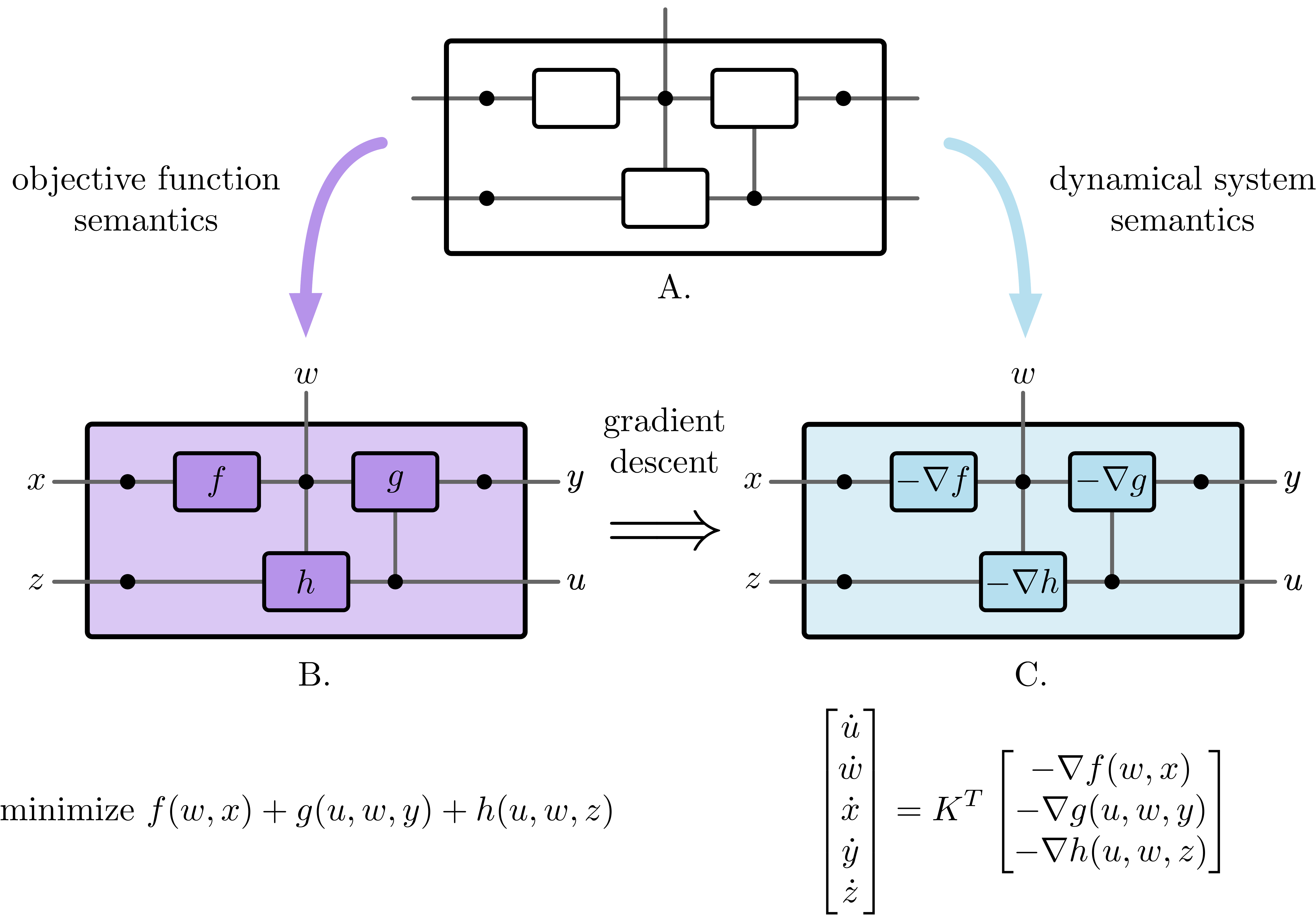}
    
    
    \caption{A. An example undirected wiring diagram (UWD), which is a special type of hypergraph. Each box has a finite set of connection points which we call \emph{ports}. We refer to the small circles as \emph{junctions} and the edges connecting ports to junctions as \emph{wires}. 
    Importantly, every UWD has a \emph{boundary}, visualized by the large outer box. We refer to ports on inner boxes and outer boxes as \emph{inner ports} and \emph{outer ports}, respectively. B. The UWD in (A) interpreted as a composite optimization problem. Subproblems inhabit boxes and their optimization variables inhabit wires. Subproblems connected by the same junction share the variables on wires incident to that junction. C. The UWD in (A) interpreted as a composite dynamical system. Subsystems inhabit boxes and their state variables inhabit wires with junctions indicating which state variables are shared. The change in a shared state variable is computed by summing changes from contributing subsystems (encoded by the matrix $K^T$). Gradient descent gives a structure-preserving map from the objective function semantics to the dynamical systems semantics.}
    \label{fig:uwds}
\end{figure}

With the gradient descent morphism, we then show that distributed message-passing versions of gradient descent and primal decomposition can be recovered by applying $\gd^\gamma$ to different types of composite problems. Naturality of $\gd^\gamma$ implies that the decomposed solution algorithm for solving a composite problem is equivalent to solving the entire composite problem directly with gradient descent. Moreover, the fact that the gradient algebra morphism preserves the hypergraph structure implies that problems defined on arbitrary UWDs can be decomposed. This is illustrated in Figure \ref{fig:uwds}. Informally, we can summarize these results by saying that the following diagram \emph{commutes} for any choice of subproblems and composition pattern:
\[\begin{tikzcd}
	{\text{Subproblems}} && {\text{Problem}} \\
	\\
	{\text{Subsystems}} && {\text{System}}
	\arrow["{\text{compose}}", from=1-1, to=1-3]
	\arrow["{\text{compose}}"', from=3-1, to=3-3]
	\arrow["{\text{gradient descent}}"{description}, from=1-1, to=3-1]
	\arrow["{\text{gradient descent}}"{description}, from=1-3, to=3-3]
\end{tikzcd}\]

Another drawback of traditional decomposition methods defined on graphs and hypergraphs which our framework addresses is that the resulting decompositions are \emph{flat}, i.e., there is a single master problem coordinating all the subproblems. However, many practical problems exhibit \emph{hierarchical} compositional structure, in which subproblems can themselves be master problems for a collection of lower level subproblems \cite{scattolini2009architectures}\cite{mesarovic2000theory} \cite{findeisen1980control}. Exploiting hierarchical structure when available can often lead to even faster solution times. Thus, a further benefit of our framework is the ability to capture and exploit such hierarchical structure. Specifically, stating our results in terms of operad algebras makes it clear that we can compose optimization problems hierarchically by recursive substitution of UWDs and that gradient descent must respect this hierarchical composition.

So far we have only spoken of gradient descent applied to non-convex differentiable minimization problems. 
In addition, this paper presents a family of closely related results for different types of problems including unconstrained and equality constrained convex programs, and various first-order solution methods such as Uzawa's algorithm \cite{uzawa_1960}
(also called gradient ascent-descent) and a primal-dual subgradient method \cite{nesterov_2009_pdsubg}. Specifically, we show that each of these types of problems defines a UWD-algebra and each solution method provides an algebra morphism to dynamical systems. The full hierarchy of results presented in this paper is summarized in Figure \ref{fig:hierarchy}. These extended results allow us to apply the same reasoning as above to derive message-passing versions of Uzawa's algorithm and dual decomposition within the same general framework. 

This novel perspective on decomposition algorithms allows us to derive a new sufficient condition for when a problem is decomposable: namely, if there exists an algebra morphism from a UWD-algebra of data on which the problem is defined to one of our algebras of optimization problems which translates the problem data into an optimization problem.
We refer to this as the \emph{compositional data condition}.
Satisfying the compositional data condition is also a constructive proof in that composing a problem translation morphism with the appropriate (sub)gradient morphism yields a distributed solution algorithm. This condition thus gives a principled way to convert problem data into explicit compositional problem structure which can be exploited via decomposition. To demonstrate the use of this sufficient condition, we show that the minimum cost network flow (MCNF) problem defines an algebra morphism from a UWD-algebra of flow networks to the UWD-algebra of unconstrained concave optimization problems, which when composed with the gradient ascent morphsim recovers a generalization of the standard dual decomposition algorithm for solving MCNF. 

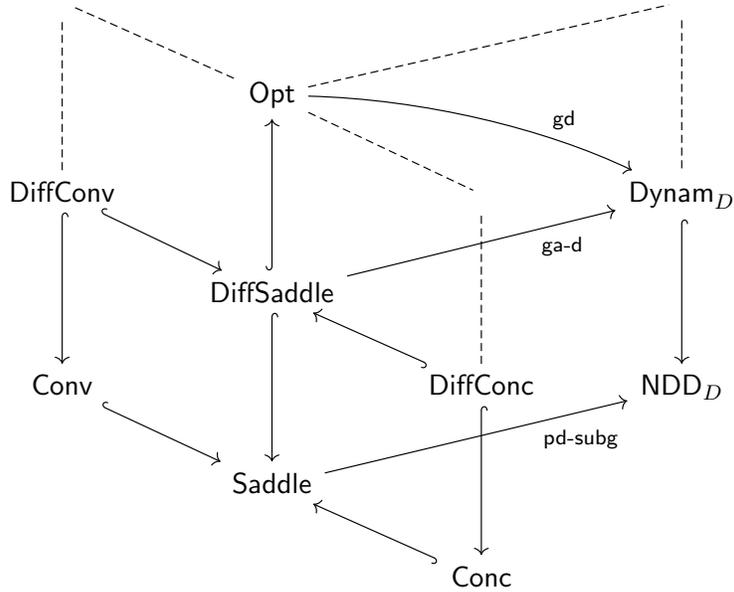
\begin{figure}
\[\begin{tikzcd}
	{} &&& {} \\
	& \Opt \\
	\DiffConv && {} & {\Dynam_D} \\
	& \DiffSaddle \\
	\Conv && \DiffConc & {\NDD_D} \\
	& \Saddle \\
	&& \Conc
	\arrow[hook, from=3-1, to=4-2]
	\arrow[hook, from=5-3, to=4-2]
	\arrow[hook, from=7-3, to=6-2]
	\arrow[hook, from=5-1, to=6-2]
	\arrow[hook, from=3-1, to=5-1]
	\arrow[hook, from=4-2, to=6-2]
	\arrow[hook, from=5-3, to=7-3]
	\arrow["\pdsubg"'{pos=0.7}, from=6-2, to=5-4]
	\arrow["\gad"'{pos=0.7}, from=4-2, to=3-4]
	\arrow[hook, from=3-4, to=5-4]
	\arrow[hook, from=4-2, to=2-2]
	\arrow["\gd"{pos=0.7}, curve={height=-12pt}, from=2-2, to=3-4]
	\arrow[dashed, no head, from=3-1, to=1-1]
	\arrow[dashed, no head, from=1-1, to=2-2]
	\arrow[dashed, no head, from=2-2, to=3-3]
	\arrow[dashed, no head, from=5-3, to=3-3]
	\arrow[dashed, no head, from=2-2, to=1-4]
	\arrow[dashed, no head, from=1-4, to=3-4]
\end{tikzcd}\]
\caption{The full hierarchy of results presented in this paper. Nodes represent the various UWD-algebras developed including those for composing saddle problems, convex problems, concave problems, all with and without differentiability assumptions, as well as composing deterministic and non-deterministic dynamical systems. Hooked arrows indicate that there is an inclusion of one algebra into another. Non-hooked arrows are the algebra morphisms including gradient descent ($\gd$), gradient ascent-descent ($\gad$) and the primal-dual subgradient method ($\pdsubg$). Composing the inclusions with the gradient algebra morphisms yields (sub)gradient descent for convex problems and (super)gradient ascent for concave problems.}
\label{fig:hierarchy}
\end{figure}

In summary, our contributions are as follows.
\begin{enumerate}
    \item We construct a family of UWD-algebras for hierarchically composing different types of problems including non-convex $C^1$ problems as well as unconstrained and equality constrained convex programs.
    \item We prove that gradient descent and Uzawa's algorithm are algebra morphisms from their respective problem algebras to the algebra of dynamical systems. We further prove that subgradient variants of these algorithms also provide algebra morphisms into a novel UWD-algebra of \emph{non-deterministic} dynamical systems.
    \item Based on these results, we provide a novel sufficient condition for when problems are decomposable in terms of algebra morphisms from data algebras into problem algebras. 
    \item As a demonstration of this sufficient condition, we construct a UWD-algebra of flow networks and prove that the translation of a flow network into its associated minimum cost network flow optimization problem yields an algebra morphism. Composing this morphism with the gradient ascent morphism yields a hierarchical dual decomposition solution algorithm.
    \item We present a prototype implementation of this framework in the Julia programming language and empirically demonstrate how exploiting higher-level compositional structure can improve solution times by utilizing hierarchical decompositions of flow networks.
\end{enumerate}

We hope that the presentation of this framework will benefit optimization researchers and practitioners in the following ways:
\begin{itemize}
    \item \textbf{Implementation}. In the same way that disciplined convex programming \cite{liberti_disciplined_2006} lowered the barrier of entry for specifying and solving convex optimization problems, our framework aims to lower the barrier of entry for \emph{composing} optimization problems and automatically generating distributed solution algorithms.
    \item \textbf{Performance}. It is well known that exploiting problem structure can improve performance, and our framework can accommodate a broad array of structures to do so. For example, we demonstrate that the algorithm produced by our framework to solve minimum cost network flow for composite networks outperforms both a centralized solution method and a conventional distributed algorithm, and this performance improvement is attained by exploiting large scale compositional structure in flow networks.  
    \item \textbf{Formal Visualization}. Undirected wiring diagrams provide an intuitive graphical syntax for composing problems. Further, UWD-algebras give a formal and unambiguous semantics to this graphical syntax, enabling diagrams to be compiled automatically to executable code.
    \item \textbf{Generalization}. In our framework, primal and dual decomposition are both recovered as special cases of applying the same UWD-algebra morphism. Consequently, our framework also enables the combination of primal and dual decomposition in a hierarchical fashion.
    We hope that by presenting a general tool for building distributed optimization algorithms, new and unique algorithms can be discovered and implemented more easily. 
    \item \textbf{Relation to other fields}. Framing our results in the language of category theory situates optimization decomposition methods in a broader mathematical landscape. For example, supervised learning \cite{fong2019backprop, cruttwell2021categorical}, optimal control~\cite{she2023characterizing}, and game theory \cite{ghani2018compositional, atkey2021compositional} have all been studied using related categorical abstractions. We hope that by framing optimization decomposition methods in the same language, we can further elucidate the connections between these fields.
\end{itemize}

The rest of the paper is organized as follows. Section \ref{sec:prelims} gives the relevant mathematical background in optimization and category theory. Specifically, Section \ref{sec:operads} reviews the definitions of operads, operad algebras, and operad algebra morphisms, which are the core abstractions we use to develop this framework. Section \ref{sec:gd} proves the fundamental result that gradient descent is an algebra morphism from the UWD-algebra of differentiable minimization problems to the UWD-algebra of dynamical systems. At this point, the core of the framework has been developed. Subsequent sections focus on extensions, generalizations, and examples.

Specifically, Section \ref{sec:uzawa} expands our framework to handle equality constrained convex programs by proving that Uzawa's algorithm yields an algebra morphism from a UWD-algebra of differentiable saddle functions to dynamical systems. Section \ref{sec:subg} generalizes this to primal-dual subgradient methods for handling equality constrained non-differentiable convex problems. To achieve this generalization we introduce new UWD-algebras for continuous and discrete non-deterministic dynamical systems, and prove that Euler's method is an algebra morphism in the non-deterministic setting. 

Turning to applications, Section \ref{sec:netflow} shows an example use of the compositional data condition by constructing a UWD-algebra of flow networks and proving that the minimum cost network flow problem defines an algebra morphism into the algebra of concave problems. Section \ref{sec:impl} presents empirical results obtained using our Julia implementation of this framework. Finally, Section \ref{sec:conclusion} offers concluding remarks and directions for future work.

\section{Preliminaries}\label{sec:prelims}

This paper brings together concepts from first-order optimization and category theory. We assume knowledge of fundamental concepts of category theory, such as categories, functors, and natural transformations. Some prior exposure to symmetric monoidal categories and lax monoidal functors is also helpful. For convenience, we have compiled a list of important categorical definitions in Appendix \ref{sec:app}. For a more complete explanation of basic category theory, a number of excellent introductory texts are available \cite{fong2018seven, riehl_context_2017, awodey2010category}. For an introduction to monoidal category theory, we refer the reader to \cite{baez2011physics}. Standard references for convex analysis and optimization are \cite{theRock} and \cite{boyd_convex}.

\subsection{Notation}

\begin{itemize}
    \item Given an $n\in\mathbb{N}$, we use $[n]$ to denote the set $\{1,\dots,n\}$.
    \item Given a linear map $T\maps V\to W$, we use $\mathcal{M}(T)$ to denote the matrix representation of $T$ with respect to a chosen basis.
    \item We use $C^1$ to denote the set of once-continuously-differentiable functions.
    \item We use $+$ to denote the coproduct of two objects and $\coprod$ to denote the coproduct of an arbitary collection of objects. Likewise, we use $\times$ to denote the Cartesian product of two objects and $\prod$ to denote the Cartesian product of an arbitrary collection of objects.
\end{itemize}

\subsection{First-order Optimization}\label{sec:foopt}

In this section, we recall the basics of first-order optimization such as gradient descent for solving unconstrained minimization problems and Uzawa's algorithm for solving constrained minimization problems. We also briefly cover the notions of sub- and super-gradients and how these generalize first-order methods to non-differentiable problems.

First-order optimization methods find extrema of real valued objective functions using only information about the objective's value (zeroth-order information) and the objective's first derivatives (first-order information). Recall that the gradient of a differentiable \define {objective function} $f\maps \R^n\to \R$ at a point $x\in\R^n$ defines the direction to move from $x$ which would most quickly increase the value of $f$. Thus, if our task is to minimize $f$, the most natural method is gradient descent, which iterates
\begin{equation}
x_{k+1}\coloneqq x_k - \gamma \nabla f(x_k)
\end{equation}
from a given starting value $x_0$, where $\gamma>0$ is a parameter called the \define{step size} or \define{learning rate} in the context of machine learning. We call $x$ the \define{decision variable} or \define{optimization variable}.
In general, the gradient descent optimization scheme can be thought of as a family of maps
\begin{equation}
    \gd_n\maps (\R^n\to\R)\to (\R^n\to\R^n)
\end{equation}
indexed by $n\in \mathbb{N}$ which takes a differentiable objective function $f\maps \R^n\to \R$ to the smooth map $-\nabla f\maps \R^n\to\R^n$ \cite{shiebler_generalized_2022}. Such a smooth map implicitly defines both a continuous and a discrete dynamical system. The continuous system is given by the system of ordinary differential equations (ODEs)
\begin{equation}
    \dot{x}\coloneqq -\nabla f(x)
\end{equation}
known as \emph{gradient flow}, while the discrete system can be obtained by applying Euler's method to the continuous system. Specifically, given a smooth map $d\maps \R^n\to\R^n$ defining the ODE $\dot{x}\coloneqq d(x)$, Euler's method for a given step size $\gamma>0$ gives the discrete system $x_{k+1}\coloneqq x_k+\gamma d(x_k)$. The discrete gradient descent algorithm is thus the result of applying Euler's method to gradient flow.

With a suitable choice of step size and mild assumptions on $f$, gradient descent is guaranteed to converge to a local minimum. Under stricter convexity assumptions on $f$, gradient descent can be shown to converge to a global minimum. Recall that a function $f\maps \R^n\to \R$ is \define{convex} if it satisfies Jensen's inequality, i.e.,
\begin{equation}\label{eq:jensen}
    f(tx + (1-ty)) \leq tf(x) + (1-t)f(y)
\end{equation}
for all $x,y\in\R^n$ and $t\in[0,1]$. When the inequality in \eqref{eq:jensen} is strict, $f$ is called \define{strictly convex}.

The basic idea behind gradient descent can be extended to also solve constrained minimization problems. A generic equality-constrained minimization problem $P$ has the standard form
\begin{equation*}
    P: \left\{\begin{array}{ll@{}ll}
    \text{minimize} & f & (x) & \\
    \text{subject to} & h & (x) = 0.
    \end{array}\right.
\end{equation*}
The function $h\maps \R^n\to\R^m$ defines $m$ \define{equality constraints}. If $f$ is convex and $h(x)=Ax-b$ is affine, then $P$ is a convex optimization problem. We consider this case for the remainder of Section \ref{sec:prelims}. A standard method for solving $P$ is via \emph{Lagrangian relaxation}, which relaxes the constrained problem in $\R^n$ to an unconstrained problem in $\R^{n+m}$. Specifically, the \define{Lagrangian} of $P$ is the function
\begin{equation}
    L(x,\lambda)\coloneqq f(x) + \lambda^T(Ax-b).
\end{equation}
The components of $\lambda\in\R^m$ are called the \define{Lagrange multipliers} and there is one for each constraint. The Lagrangian is a \define{saddle function} as it is convex in $x$ for any fixed $\lambda$ and concave in $\lambda$ for any fixed $x$. The key insight of Lagrange was that saddle points of $L$, i.e., points $(\hat{x}, \hat{\lambda}) \in \argmin_{x \in \mathbb{R}^n} \argmax_{\lambda \in \mathbb{R}^m} L(x, \lambda)$, correspond to constrained minima of $P$. As such, finding a saddle point of $L$ is sufficient for solving $P$ (\cite{boyd_convex}, \S 5.4.2).

Uzawa's algorithm is a straightforward generalization of gradient descent for finding saddle points of saddle functions. Specifically, it performs gradient \emph{descent} on the convex part of the saddle function and gradient \emph{ascent} on the concave part. Applying this to the Lagrangian $L$ gives the algorithm
\begin{equation}
    \begin{array}{ll@{}ll}
        x_{k+1} & \coloneqq x_k - \gamma \nabla_x L(x_k,\lambda_k)  \\
        \lambda_{k+1} & \coloneqq \lambda_k +\gamma\nabla_\lambda L(x_k,\lambda_k).
    \end{array}
\end{equation}
This can again be seen as the Euler discretization applied to a continuous system, which we call \emph{saddle flow}:
\begin{equation}
\begin{array}{ll@{}ll}
    \dot{x} &\coloneqq -\nabla_x L(x,\lambda) \\
    \dot{\lambda} & \coloneqq \nabla_\lambda L(x,\lambda).
    \end{array}
\end{equation}
Under similar conditions to gradient descent, Uzawa's algorithm converges to a saddle point of $L$.

Finally, most of these ideas can be generalized to the case when $f$ and $h$ are still convex and affine, respectively, but not necessarily differentiable. For example, the absolute value function is a common objective which is convex but not differentiable at zero. Minimizing such a function is achieved through a generalization of the gradient to convex functions that need not be differentiable.
\begin{definition}
    Given a convex function $f\maps\R^n\to\R$, a vector $v\in\R^n$ is called a \define{subgradient} at a point $x\in \R^n$ if for every $y\in\R^n$,
    \begin{equation}
        f(y) - f(x) \geq v^T(y - x).
    \end{equation}
    The set of all such points is called the \define{subdifferential} of $f$ at $x$.

    Similarly, when $f$ is concave, a vector $w\in\R^n$ is called a \define{supergradient} at a point $x\in\R^n$ if for every $y\in\R^n$,
    \begin{equation}
        f(y) - f(x) \leq w^T(y-x).
    \end{equation}
    The set of all such points is called the \define{superdifferential} of $f$ at $x$. The superdifferentials and subdifferentials of $f$ at $x$ are both denoted $\partial f(x)$; the distinction between super and sub will be clear from context.
\end{definition}
The subdifferential of a convex function $f$ at a point $x$ is the set of affine \emph{under}-approximations of $f$ at $x$. When $f$ is differentiable at $x$, $\partial f(x)$ is the singleton set $\{\nabla f(x)\}$, implying that the subgradient is a proper generalization of the gradient for convex functions.

The subdifferential is the basis for an iterative algorithm to find minima of non-differentiable convex functions, known as subgradient descent. For a convex objective function $f\maps\R^n\to \R$, this simply iterates
\begin{equation}
    x_{k+1} = x_k - \gamma \rho_k,
\end{equation}
where $\rho_k$ is any subgradient of $f$ at $x_k$\footnote{This subgradient method is not guaranteed to descend at each iteration. However, by keeping track of the best iterate found so far, this can be made into a descent method.}. The continuous-time analog of subgradient descent is subgradient differential inclusion, given by the system
\begin{equation}
    \dot{x}\in -\partial f(x).
\end{equation}
It is important to note that because the subdifferential is a point-to-set mapping, these dynamical systems are \emph{non-deterministic}, meaning that at a given state, they can evolve in different directions. In spite of this non-determinism, subgradient descent also converges to within an $\epsilon$-neighborhood of a minimum, where $\epsilon$ scales with step size. By using a combination of subgradients and supergradients, such methods can be extended to find extrema of non-differentiable saddle functions.

\subsection{Algebras of Undirected Wiring Diagrams}\label{sec:operads}


An \define{undirected wiring diagram (UWD)} is a combinatorial object representing a pattern of interconnection between boxes and junctions. Specifically, a UWD consists of finite sets of boxes and junctions. Each box has a finite set of ports associated to it, and each port is wired to a junction. One of the boxes is distinguished as an outer box, and the ports on the inner boxes and outer box are referred to as inner ports and outer ports respectively. An example UWD is shown in Figure \ref{fig:explicit_uwd}.A.

The goal of this section is to define algebras on the operad of undirected wiring diagrams and show how they can be specified by lax symmetric monoidal functors $(\FinSet,+)\to (\Set,\times)$, which we call \define{finset algebras}.
Similarly, we discuss how monoidal natural transformations between finset algebras give rise to UWD-algebra morphisms. These will be the core abstractions repeatedly utilized throughout the paper. We ground this section with the examples from \cite{libkind_operadic_2022} of continuous and discrete dynamical systems as UWD-algebras and Euler's method as an algebra morphism from continuous to discrete.

Intuitively, undirected wiring diagrams give a \emph{syntax} for composition, and can be imbued with various \emph{semantics}. Formally, UWDs form an \emph{operad} and a specific choice of semantics is given by an \emph{operad algebra}.


\begin{definition}[Definition 2.1.2 in \cite{spivak_operad_2013}]
    An \define{symmetric coloured operad} $\mathcal{O}$ consists of the following:
    \begin{itemize}[noitemsep]
    \item A collection of objects,
    \item For each $n\in \mathbb{N}^+$ and collection of objects $s_1,\dots,s_n,t$ in $\mathcal{O}$, a collection of morphisms $\mathcal{O}(s_1,\dots,s_n;t)$, 
    \item An identity morphism $\id_t\in\mathcal{O}(t;t)$ for each object $t$,
    \item Composition maps of the form
    \begin{equation}
        \mathcal{O}(s_1,\dots,s_n;t)\times \prod_{i=1}^n\mathcal{O}(r_{i,1},\dots,r_{i,m_i};s_i) \to \mathcal{O}(r_1,\dots,r_n;t),
    \end{equation}
    where $r_i = r_{i,1},\dots,r_{i,m_i}$.
    \item Permutation maps of the form
    \begin{equation}
        \OO(s_1,\dots,s_n;t)\to \OO(s_{\sigma(1)},\dots, s_{\sigma(n)};t)
    \end{equation}
    for every permutation $\sigma$ in the symmetric group $\Sigma_{[n]}$.
    \end{itemize}
    This data is subject to suitable identity, associativity, and symmetry laws.
\end{definition}
Henceforth, we simply use the term operad to refer to symmetric coloured operads. A basic but essential operad is that of sets. 
\begin{definition}[Example 2.1.4 in \cite{spivak_operad_2013}]
    The operad of sets, denoted $\Sets$, has sets as objects. Given sets $X_1,\dots,X_n,Y$, the morphisms in $\Sets(X_1,\dots,X_n;Y)$ are the functions $X_1\times\dots\times X_n\to Y$. Given composable morphisms $g\maps Y_1\times\dots\times Y_n\to Z$ and $f_i\maps X_{i,1}\times\dots\times X_{i,m_i}\to Y_i$ for all $i\in [n]$, the composite is the function
    \begin{equation}
        (X_{1,1}\times\dots\times X_{1,m_1})\times\dots\times (X_{n,1}\times\dots \times X_{n,m_n})\xrightarrow{f_1\times\dots\times f_n} Y_1\times\dots\times Y_n\xrightarrow{g}Z.
    \end{equation}
\end{definition}

\begin{figure}
    \centering
    \includegraphics[scale=.175]{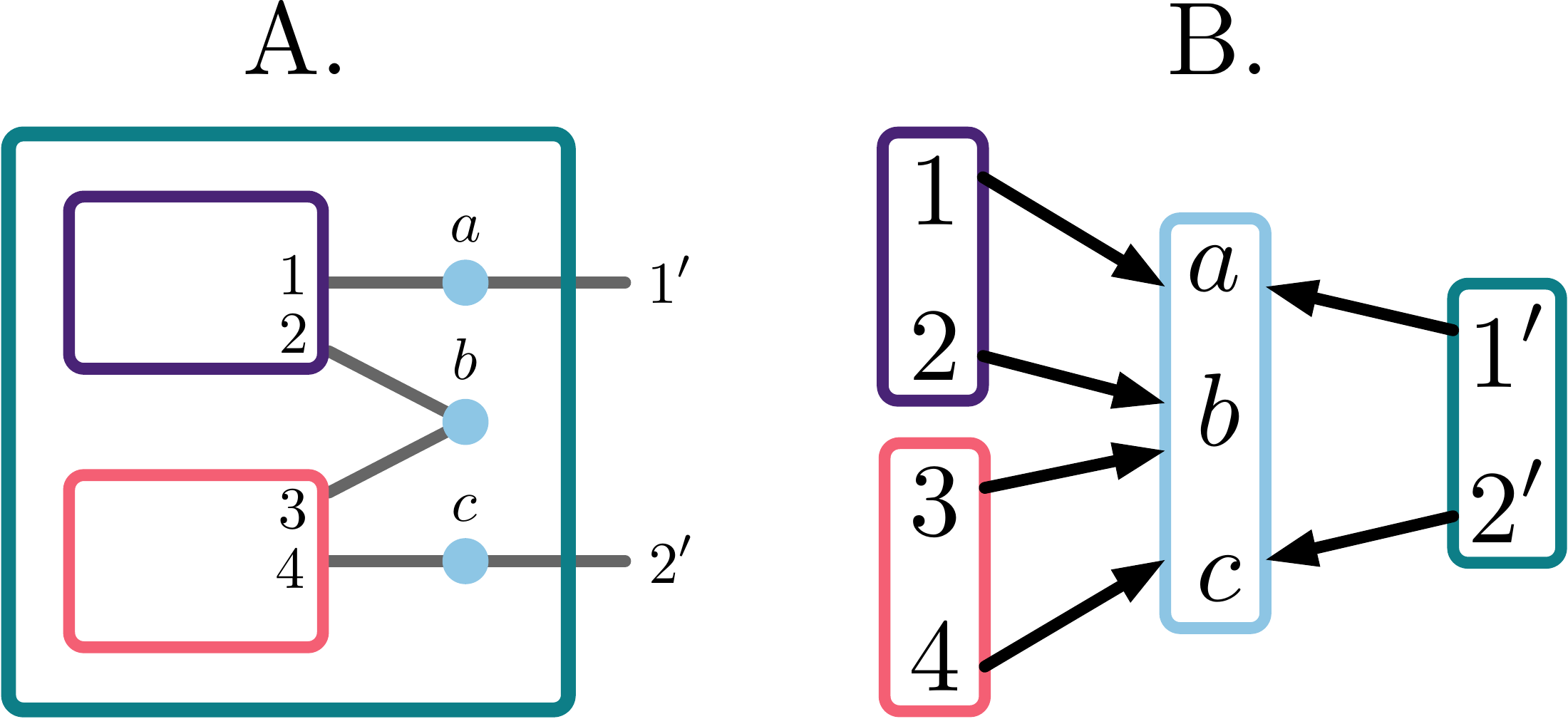}
    \caption{A. An example UWD with two inner boxes. B. The cospan representation of the UWD in (A). The signature of this UWD is $\{1,2\}+\{3,4\}\to \{a,b,c\}\leftarrow\{1',2'\}$. }
    \label{fig:explicit_uwd}
\end{figure}

To see how UWDs form an operad, first note that any UWD with $n$ inner boxes can be encoded by a pair of functions $P_1+\dots+P_n\xrightarrow{l} J\xleftarrow{r}P'$, where $P_i$ is the set of inner ports of box $i$, $J$ is the set of junctions, and $P'$ is the set of outer ports. The function $l$ then encodes the connection pattern of inner ports to junctions, while $r$ encodes the connection pattern of outer ports to junctions. Such a pair of morphisms with common codomain is called a \define{cospan}. Figure \ref{fig:explicit_uwd} shows an example of the correspondance between a UWD and its cospan representation.

\begin{definition}[Example 2.1.7 in \cite{spivak_operad_2013}]
    The operad of UWDs, denoted \textnormal{\textbf{UWD}}, has finite sets as objects. Given finite sets $P_1\dots,P_n,P'$, a morphism in $\UWD(P_1,\dots,P_n;P')$ is a UWD with those ports, i.e., a choice of finite set $J$ and a cospan $P_1+\dots+P_n\xrightarrow{l} J\xleftarrow{r}P'$. Given composable UWDs 
    \begin{equation}
        \coprod_{i\in[n]}Y_i\xrightarrow{l'} J'\xleftarrow{r'} Z,
    \end{equation}
    and
    \begin{equation}
        \coprod_{j\in [m_i]} X_{i,j}\xrightarrow{l_i}J_i\xleftarrow{r_i} Y_i,
        \qquad i \in [n],
    \end{equation}
    their composite is defined by taking the following pushout:
\[\begin{tikzcd}
	&&&& Z \\
	\\
	&& {\coprod_{i\in [n]}Y_i} && {J'} \\
	\\
	{\coprod_{i\in[n]}\coprod_{j\in[m_i]}X_{i,j}} && {\coprod_{i\in[n]}J_i} && {J''}
	\arrow["{r'}", from=1-5, to=3-5]
	\arrow["{l'}", from=3-3, to=3-5]
	\arrow["{\coprod_ir_i}"', from=3-3, to=5-3]
	\arrow["{\coprod_il_i}"', from=5-1, to=5-3]
	\arrow["p"', from=5-3, to=5-5]
	\arrow["q", from=3-5, to=5-5]
	\arrow["\lrcorner"{anchor=center, pos=0.125, rotate=180}, draw=none, from=5-5, to=3-3]
\end{tikzcd}\]
    The composite UWD is then given by the cospan with left leg $p\circ \coprod_il_i$ and right leg $q\circ r'$.
    The identity UWD on $X$ is the identity cospan $X=X=X$.
\end{definition}

Although composition of UWDs is somewhat complicated to define formally, it has an intuitive explanation. It takes a ``target" UWD with $n$ inner boxes along with $n$ ``filler" UWDs and substitutes one of the filler UWDs into each of the inner boxes of the target UWD. This operation is only defined when each filler UWD has the same number of outer ports as the number of ports on the inner box it is filling. 

Having defined the operad of UWDs and the operad of sets, we can now construct UWD-algebras, which give semantic interpretations to UWDs. These are special cases of \emph{operad functors}, which are the operadic analogue of functors between categories.

\begin{definition}[Definition 2.2.1 in \cite{spivak_operad_2013}]
    A \define{UWD-algebra} is an operad functor $F\maps \UWD\to \Sets$. Unpacking this definition, $F$ consists of an object map taking each finite set $P$ to a set $F(P)$ and a morphism map taking each UWD of the form $\Phi\maps P_1+\dots + P_n\to J\leftarrow P'$ to a function $F(\Phi)\maps F(P_1)\times\dots\times F(P_n)\to F(P')$. These maps must respect the following axioms:
    \begin{enumerate}
        \item $F$ preserves identities, i.e., $F(\id_X)=\id_{F(X)}$ for all objects $X$.
        \item $F$ respects composition, i.e., the following diagram commutes for all composable UWDs:
\[\begin{tikzcd}
	{\UWD(Y_1,\dots,Y_n;Z)\times\prod_{i\in[n]}\UWD(X_i;Y_i)} & {\UWD(X_1,\dots,X_n;Z)} \\
	\\
	{\Sets(FY_1,\dots,FY_n;FZ)\times\prod_{i\in[n]}\Sets(FX_i;FY_i)} & {\Sets(FX_1,\dots,FX_n;FZ)}
	\arrow["\circ", from=1-1, to=1-2]
	\arrow["\circ"', from=3-1, to=3-2]
	\arrow["F"', from=1-1, to=3-1]
	\arrow["F", from=1-2, to=3-2]
\end{tikzcd}\]
    where $X_i\coloneqq X_{i,1},\dots,X_{i,m_i}$ and $FX_i$ denotes objectwise application of $F$ to each $X_{i,j}$. Intuitively, this diagram says that the result of substituting UWDs and then applying the algebra must be the same as applying the algebra to each component UWD and then composing those as functions.
    \end{enumerate}
    Given a UWD $\Phi\in \UWD(P_1,\dots,P_n;P')$ and fillers $o_i\in F(P_i)$ for all $i\in[n]$, we refer to the resultant $o'\in F(P')$ obtained by $F(\Phi)(o_1,\dots,o_n)$ as an $F$-\textbf{UWD}.
\end{definition}
We can think of a UWD-algebra $F$ applied to a finite set $X$ as defining the set of possible objects which can fill a box with $X$ connection points. Then, we can think of $F$ applied to a UWD $\Phi$ as a \emph{composition function} which specifies how to compose objects filling each inner box of $\Phi$ into an object which fills the outer box of $\Phi$.

\begin{definition}[Definition 2.2.5 in \cite{spivak_operad_2013}]
    Given UWD-algebras $F,G\maps \UWD\to \Sets$, an \define{algebra morphism} $\alpha\maps F\Rightarrow G$ consists of, for every finite set $X$, a function $\alpha_X\maps F(X)\to G(X)$, known as the \define{component} of $\alpha$ at $X$. These components must obey the following law: 
    given a UWD $\Phi\maps P_1+\dots +P_n\to J\leftarrow P'$, the following \define{naturality square} must commute:
\[\begin{tikzcd}
	{\prod_{i\in [n]}F(P_i)} && {\prod_{i\in[n]}G(P_i)} \\
	{F(P')} && {G(P')}
	\arrow["{F(\Phi)}"', from=1-1, to=2-1]
	\arrow["{G(\Phi)}", from=1-3, to=2-3]
	\arrow["{\prod_i\alpha_{P_i}}", from=1-1, to=1-3]
	\arrow["{\alpha_{P'}}"', from=2-1, to=2-3]
\end{tikzcd}\]
\end{definition}
Intuitively, naturality of $\alpha$ just says that composing objects with $F$ then transforming the composite to an object of $G$ is the same as transforming each component object and composing with $G$.

There is an important connection between the operadic definitions presented above and (i) symmetric monoidal categories (SMCs), 
(ii) lax symmetric monoidal functors, and (iii) monoidal natural transformations (consult Appendix \ref{sec:smcs} for these definitions). Specifically, every SMC $(\CC,\otimes)$ has an underlying operad $\OO(\CC)$ with objects the same as those of $\CC$ and whose morphisms $\OO(\CC)(s_1,\dots,s_n;t)$ are the morphisms $s_1\otimes\dots\otimes s_n\to t$ in $\CC$. For example, $\Sets$ is the operad underlying the SMC $(\Set,\times)$. Similarly, $\UWD$ is the operad underlying the SMC $(\Cospan,+)$ of cospans between finite sets, where composition is given by pushout and the monoidal product is given by disjoint union. Furthermore, every lax symmetric monoidal functor $(F,\varphi)\maps (\CC,\otimes)\to (\DD,\boxtimes)$ gives rise to an operad functor $\OO(F)\maps \OO(\CC)\to\OO(\DD)$ which acts on objects the same as $F$ and on morphisms $f\maps X_1\otimes\dots\otimes X_n\to Y$ is defined as
\begin{equation}
    \OO(F)(f)\coloneqq F(X_1)\boxtimes\dots\boxtimes F(X_n)\xrightarrow{\varphi_{X_1,\dots,X_n}}F(X_1\otimes\dots\otimes X_n)\xrightarrow{F(f)}F(Y).
\end{equation}
Thus, we can specify UWD-algebras by specifying lax symmetric monoidal functors $(\Cospan,+)\to (\Set,\times)$, also known as \define{cospan algebras} \cite{fong_hypergraph_2019}. Finally, the components of monoidal natural transformations $\alpha\maps F\Rightarrow G$ between cospan algebras will also be natural as components of a UWD-algebra morphism $\alpha \maps \OO(F)\Rightarrow \OO(G)$ \cite{hermida_representable_2000}. Because of this tight connection, we will typically use the concepts of SMCs, lax symmetric monoidal functors, and monoidal natural transformations interchangably with their operadic counterparts moving forward.


Throughout this paper, we utilize a particularly well behaved class of cospan algebras, namely those generated by lax symmetric monoidal functors from $(\FinSet,+)\to (\Set,\times)$, which we call \define{finset algebras}. The following simple, but fundamental, example illustrates the use of finset algebras to define linear maps which merge or copy components of their inputs.

\begin{example}[Pushforwards and Pullbacks]\label{ex:push-pull}
    \begin{figure}
        \centering
        \includegraphics[width=\textwidth]{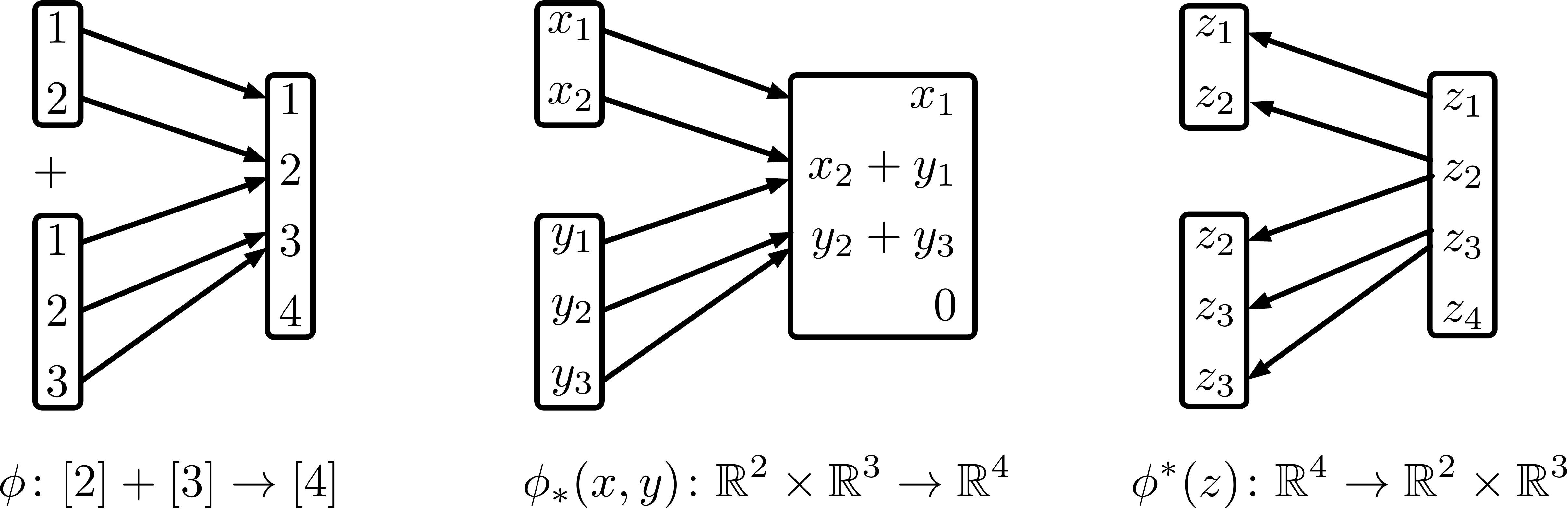}
        \caption{A visualization of the relationship between a function $\phi\maps[2]+[3]\to [4]$, the action of the pushforward $\phi_*$ on an input pair $(x,y)$, and the action of the pullback $\phi^*$ on an input $z$. The pushforward sums components of $(x,y)$ according to $\phi$ while the pullback duplicates components of $z$ according to $\phi$.}
        \label{fig:push-pull}
    \end{figure}

    Given a finite set $N$, the set of functions $\{N\to\R\}=\R^N$ is a vector space, which is called the \define{free vector space} generated by $N$. If we let $n\coloneqq |N|$ be the cardinality of $N$, then choosing an ordering $N\xrightarrow{\cong}[n]$ on $N$ induces an isomorphism of vector spaces $\R^N\cong \R^n$. 
    We can thus think of the free vector space $\R^N$ as standard $n$-dimensional Euclidean space, but where each basis vector is labelled by an element of $N$. This free vector space construction will be essential for creating UWD-algebras whose semantics are given by functions defined on $\R^n$ such as objective functions and dynamical systems. First, we consider two ways of constructing linear maps between free vector spaces that give rise to finset algebras.

    Given a function $\phi\maps N\to M$ between finite sets, the \define{pushforward} along $\phi$ is the linear map $\phi_*\maps\R^N\to\R^M$ defined by
    \begin{equation}
        \phi_*(x)_j \coloneqq \sum_{i\in \phi^{-1}(j)} x_i
    \end{equation}
    for all $x\in\R^N$ and $j\in M$. This lets us define the \emph{strong}\footnote{A strong symmetric monoidal functor is one whose comparison maps are all isomorphisms (see Definition \ref{smf}).} finset algebra $(-)_*\maps (\FinSet,+)\to (\Set,\times)$, which takes finite sets to their free vector spaces and functions between finite sets to their pushforwards. The product comparison is given by the isomorphisms $\varphi_{N,M}\maps \R^N\times\R^M\cong \R^{N+M}$, while the unit comparison $\varphi_0\maps 1\to \R^\emptyset$ is unique because $\R^\emptyset$ contains only the zero vector.

    Given a function $\phi\maps X+Y\to Z$, we can apply this finset algebra to produce a linear map which takes a pair of vectors $x\in\R^X$ and $y\in\R^Y$ as input, applies the product comparison isomorphism to concatenate $x$ and $y$ into a single vector, then applies the pushforward $\phi_*$ to the resultant vector. 
    

    A dual operation to the pushforward along $\phi$ is the \define{pullback} along $\phi$, which is the linear map $\phi^*\maps \R^M\to \R^N$ defined by
    \begin{equation}
        \phi^*(y)_i\coloneqq y_{\phi(i)}
    \end{equation}
    for all $y\in\R^M$ and $i\in N$. It is straightforward to show that the pullback along $\phi$ is dual as a linear map to the pushforward along $\phi$. The pullback operation defines a \emph{contravariant} finset algebra $(-)^*\maps (\FinSet^{op},+)\to (\Set,\times)$ 
    taking finite sets to their free vector spaces and functions to their pullbacks. The product comparisons and unit comparison are the same as those of $(-)_*$, but in the opposite direction. This algebra takes a function $\phi\maps X+Y\to Z$ and a vector $z\in \R^Z$ and produces a pair of vectors in $\R^X\times\R^Y$ by copying components of $z$ into the respective components of the results dictated by $\phi$. This is illustrated in Figure \ref{fig:push-pull}.
    

    
\end{example}

Intuitively, the action of a finset algebra on a morphism $\phi\maps X_1+\dots + X_n\to Y$ specifies how to compose an $n$-tuple of objects defined on $X_1,\dots,X_n$ into a single object defined on $Y$. Such an action is sufficient to specify a UWD-algebra, as described in the following lemma.

\begin{lemma}\label{lem:cspLift}
    Given a lax symmetric monoidal functor $(F,\varphi)\maps (\FinSet,+)\to(\Set,\times)$, there is a lax symmetric monoidal functor $(F\textnormal{Csp},\varphi')\maps (\Cospan,+)\to(\Set,\times)$ defined by the following maps:
    \begin{itemize}
        \item On objects, $F\textnormal{Csp}$ takes a finite set $X$ to the set of triples of the form $(S\in\FinSet, o\in F(S), m\maps X\to S)$.
        \item Given a cospan $\Phi\coloneqq (X\xrightarrow{l}J\xleftarrow{r} Y)$, the function $F\textnormal{Csp}(\Phi)\maps F\textnormal{Csp}(X)\to F\textnormal{Csp}(Y)$ is defined by
        \begin{equation}
            (S,o,m)\mapsto (S+_X J, F(p_S)(o), p_J\circ r),
        \end{equation}
        where $S+_X J, p_S$, and $p_J$ are defined by the pushout:
\[\begin{tikzcd}
	& X && Y \\
	S && J \\
	& {S+_XJ}
	\arrow["m", from=1-2, to=2-1]
	\arrow["l"', from=1-2, to=2-3]
	\arrow["r", from=1-4, to=2-3]
	\arrow["{p_S}"', from=2-1, to=3-2]
	\arrow["{p_J}", from=2-3, to=3-2]
	\arrow["\lrcorner"{anchor=center, pos=0.125, rotate=135}, draw=none, from=3-2, to=1-2]
\end{tikzcd}\]

        \item Given objects $N,M\in\FinSet$, the product comparison
        \begin{equation}
            \varphi'_{N,M}\maps F\textnormal{Csp}(N)\times F\textnormal{Csp}(M)\to F\textnormal{Csp}(N+M)
        \end{equation}
        is given by
        \begin{equation}
            \varphi'_{N,M}((S_1,o_1,m_1),(S_2,o_2,m_2))\coloneqq (S_1+S_2, \varphi_{S_1,S_2}(o_1,o_2),m_1+m_2),
        \end{equation}
        where $\varphi_{S_1,S_2}$ is the product comparison of $F$.
        \item The unit comparison $\varphi'_0\maps \1\to F\textnormal{Csp}(\emptyset)$ picks the triple $(\emptyset,\varphi_0,\emptyset\xrightarrow{!}\emptyset)$,
        where $\varphi_0$ is the unit comparison of $F$.
    \end{itemize}
\end{lemma}
\begin{proof}
    This is an application of the equivalence between hypergraph categories and cospan algebras proven in \cite{fong_hypergraph_2019} to the hypergraph category of $F$-decorated cospans introduced in \cite{fong_algebra_2016}. This is also straightforward to prove directly.
\end{proof}

We can develop intuition for this lemma by again thinking in terms of UWDs. The set of triples $(S, o, m\maps X\to S)$ obtained by applying $F\textnormal{Csp}$ to $X$ corresponds to the set of fillers for a UWD box with $X$ ports. However, the object filling a box is allowed to have a domain $S$ other than $X$, and $m\maps X\to S$ determines the relationship between the ports of the box and the domain of the filling object. Then, treating the cospan $\Phi\coloneqq X\xrightarrow{l}J\xleftarrow{r} Y$ as a UWD, the map $l\maps X\to J$ specifies the gluing of inner ports to junctions. By taking the pushout of $m$ and $l$, the map $p_S\maps S\to S+_XJ$ specifies how to glue together the domains of the objects filling each box: if two ports are mapped to the same junction by $l$, then their image under $m$ will be mapped to the same element of $S+_XJ$ by $p_S$.

With this perspective, the fundamental operation is applying $F$ to the function $p_S$, which should specify how to glue together parts of an object with domain $S$ to produce an object with domain $S+_XJ$ according to $p_S$. The rest of the operations are essentially book-keeping operations to produce a valid UWD with this new object as its filler.

\begin{remark}\label{rem:simple}
    Consider a UWD-algebra $F\mathrm{Csp}$ generated by a finset algebra $F$ by applying Lemma \ref{lem:cspLift}. If we restrict our attention to UWDs of the form $\Phi\coloneqq (X\xrightarrow{\phi}J\cong Y)$, so that every junction maps to a unique outer port, and only consider algebra objects $F\mathrm{Csp}(X)$ of the form $(S, o\in F(S), X\xrightarrow{\cong}S)$, we get a subalgebra that corresponds directly to the finset algebra $F$. Specifically, the pushout needed to define the function $F\mathrm{Csp}(\Phi)\maps F\mathrm{Csp}(X)\to F\mathrm{Csp}(Y)$ is
\[\begin{tikzcd}
	& X && Y \\
	S && J \\
	& J
	\arrow["\phi", from=1-2, to=2-3]
	\arrow["\cong"', from=1-4, to=2-3]
	\arrow["{\cong}"', from=1-2, to=2-1]
	\arrow["\phi"', from=2-1, to=3-2]
	\arrow["{\cong}", from=2-3, to=3-2]
	\arrow["\lrcorner"{anchor=center, pos=0.125, rotate=135}, draw=none, from=3-2, to=1-2]
\end{tikzcd}\]
    and the result of $F\mathrm{Csp}(\Phi)(S,o,X\xrightarrow{\cong}S)$ is $(J, F(\phi)(o), Y\xrightarrow{\cong}J)$.

    Therefore, in this restricted case that the set of outer ports is isomorphic to the set of junctions and each UWD filler has a domain which is isomorphic to the set of ports on the box it is filling, we obtain the resulting composite object by simply applying our finset algebra $F$ to the function $\phi$ which maps inner ports to junctions.


    We will make use of this simplification in some of our examples, but we emphasize that all this machinery still works for general UWDs and algebra objects.
\end{remark}

In a similar vein, we can use monoidal natural transformations between finset algebras to define monoidal natural transformations between cospan algebras.
\begin{lemma}\label{lem:cspNatLift}
    Suppose we are given two lax symmetric monoidal functors $F,G\maps(\FinSet,+)\to(\Set,\times)$ and a monoidal natural transformation $\alpha\maps F\Rightarrow G$. Let $F\textnormal{Csp}$ and $G\textnormal{Csp}$ be the cospan algebras which result from applying Lemma \ref{lem:cspLift} to $F$ and $G$. Then there is a monoidal natural transformation $\alpha'\maps F\textnormal{Csp}\Rightarrow G\textnormal{Csp}$ with components $\alpha'_N\maps F\textnormal{Csp}(N)\to G\textnormal{Csp}(N)$ defined by
    \begin{equation}
        (S,v\in F(S), m\maps N\to S)\mapsto (S,\alpha_S(v),m).
    \end{equation}
\end{lemma}
\begin{proof}
    This is an application of the equivalence between cospan algebra morphisms and hypergraph functors proven in \cite{fong_hypergraph_2019} applied to the hypergraph functor induced by applying Theorem 2.8 in \cite{fong_algebra_2016} to $\alpha$.
\end{proof}

\textbf{In summary}, we can build a UWD-algebra from a finset algebra $F$ by applying Lemma~\ref{lem:cspLift} to produce a cospan algebra $F\mathrm{Csp}$ and then taking its underlying operad functor $\OO(F\mathrm{Csp})$. Similarly, given finset algebras $F$ and $G$ and a monoidal natural transformation $\alpha\maps F\Rightarrow G$, we can produce a UWD-algebra morphism by applying Lemma \ref{lem:cspNatLift} to obtain a monoidal natural transformation $\alpha'\maps F\mathrm{Csp}\Rightarrow G\mathrm{Csp}$. The components of $\alpha'$ then also define the algebra morphism $\alpha'\maps\OO(F\mathrm{Csp})\Rightarrow \OO(G\mathrm{Csp})$.


All the UWD-algebras we define in this paper will be generated by a finset algebra in the way described above. Similarly, all algebra morphisms we define will come from monoidal natural transformations between finset algebras. Consequently, for the remainder of this paper, we elide the difference between finset algebras and their associated UWD-algebras. Specifically, whenever we define a finset algebra $F$, uses of $F$ beyond the definition implicitly refer to $\OO(F\textnormal{Csp})$.




\begin{example}[Building Linear Maps with UWDs]\label{ex:linearmaps}

    \begin{figure}
        \centering
        \includegraphics[width=\textwidth]{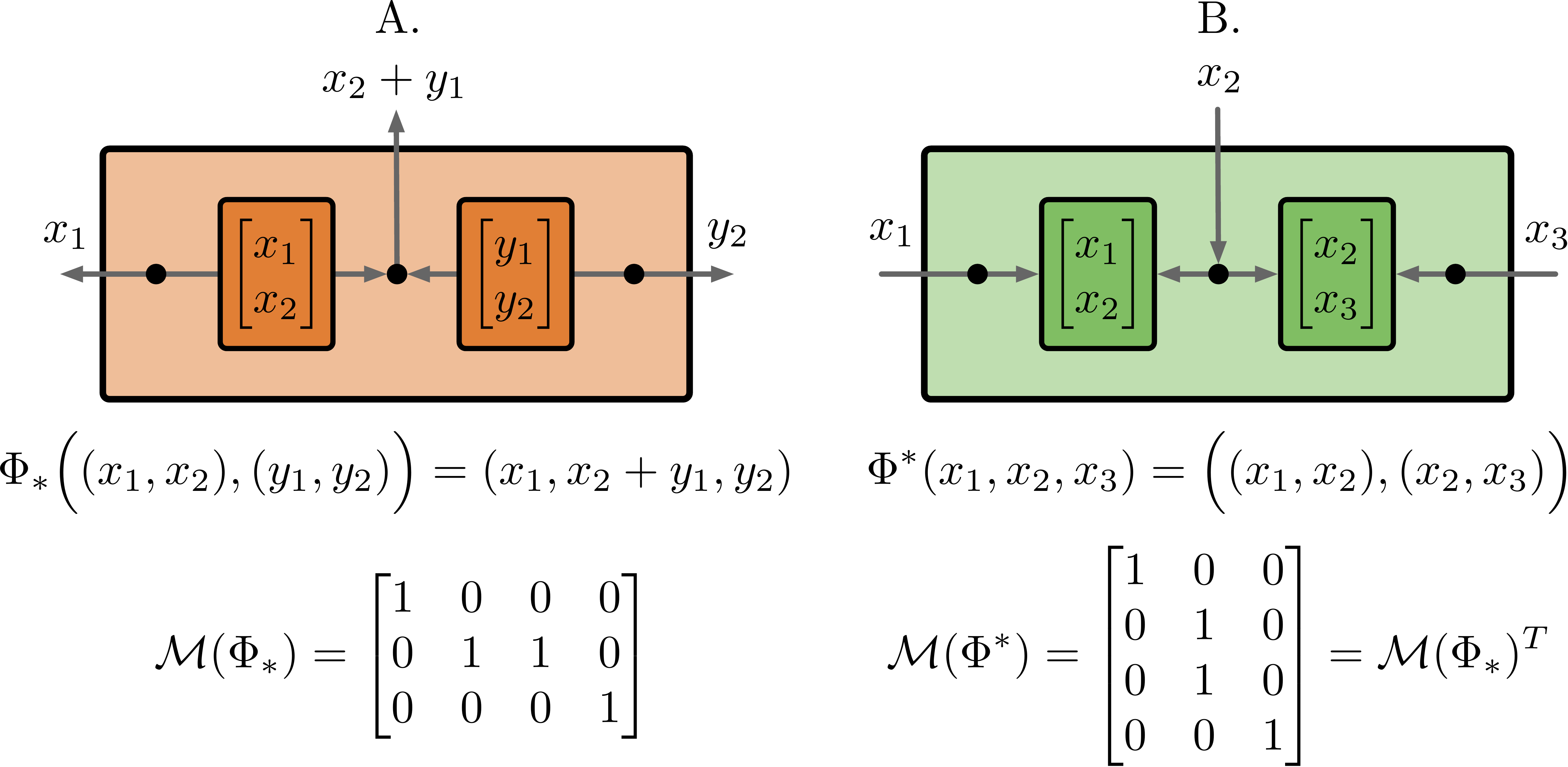}
        \caption{A. An example of the \emph{collect algebra} acting on a UWD $\Phi$ with two inner boxes. The resulting linear map $\Phi_*$ takes a pair of vectors in $\R^2\times \R^2$ as input and produces a vector in $\R^3$ by summing the components of the inputs which share the same junction in $\Phi$. Arrows are added to emphasize that the flow of information is directed from the inner boxes to the boundary box. B. An example of the \emph{distribute algebra} acting on $\Phi$. The resulting linear map $\Phi^*$ takes a vector in $\R^3$ and produces a pair of vectors in $\R^2\times \R^2$ by copying components of the input which share the same junction in $\Phi$. Arrows are added to emphasize that the flow of information is directed from the boundary box to the inner boxes. The matrix representations are with respect to the standard bases. 
        Note that these examples satisfy the simplifying assumptions of Remark \ref{rem:simple}.}
        \label{fig:linear_maps}
    \end{figure}
    
    We can apply Lemma \ref{lem:cspLift} to the pushforward and pullback algebras defined in Example \ref{ex:push-pull} to lift them to UWD-algebras. This lets us use the graphical syntax of UWDs to specify linear maps. 
    
    We call the UWD-algebra obtained by applying Lemma \ref{lem:cspLift} to the pushforward algebra the \define{collect algebra} because it takes vectors for each inner box of a UWD and collects them into a vector for the outer box. Likewise, we call the UWD-algebra obtained by applying Lemma \ref{lem:cspLift} to the pullback algebra the \define{distribute algebra} because it takes a vector for the outer box and distributes it to produce vectors for each inner box. Both these operations are shown in Figure \ref{fig:linear_maps}.

    Simple though they are, these algebras will be indispensable to defining the subsequent algebras of optimization problems and dynamical systems.
\end{example}

\begin{example}[Composing Dynamical Systems]\label{ex:dynam}
    There is a finset algebra $\Dynam\maps (\FinSet,+)\to (\Set,\times)$ given by the following maps:
    \begin{itemize}
        \item On objects, $\Dynam$ takes a finite set $N$ to the set of smooth maps $\{\upsilon\maps \R^N\to \R^N\}$.
        \item Given a morphism $\phi\maps N\to M$ in $\FinSet$, $\Dynam(\phi)\maps \Dynam(N)\to\Dynam(M)$ is defined by the function $\upsilon\mapsto \phi_*\circ \upsilon \circ\phi^*$.
        \item Given finite sets $N$ and $M$, the product comparison $\varphi_{N,M}\maps \Dynam(N)\times\Dynam(M)\to\Dynam(N+M)$ is given by the function
        \begin{equation}
            (\upsilon,\rho)\mapsto \iota_{N*}\circ\upsilon \circ\iota_N^* + \iota_{M*}\circ\rho\circ\iota_M^*,
        \end{equation}
        where $\iota_N\maps N\to N+M$ and $\iota_M\maps M\to N+M$ are the natural inclusions.
        \item The unit comparison $\varphi_0\maps \1\to\Dynam(\emptyset)$ is uniquely determined because the set of maps $\{\R^0\to \R^0\}$ is a singleton.
    \end{itemize}
    This example is the content of Lemmas 15 and 16 in \cite{baez_compositional_2017}.

    We can now apply Lemma \ref{lem:cspLift} to obtain a UWD-algebra for composing dynamical systems. 
    A filler for a box in a $\Dynam$-UWD with $X$ ports is a triple $(S, \upsilon\maps \R^S\to\R^S, m\maps X\to S)$. We refer to such an object as an \emph{open} dynamical system, because it designates only a subset of the system's state variables as being open to sharing with other systems. This is determined by the map $m\maps X\to S$. The connection pattern of ports to junctions in a UWD then specifies which open state variables of one system are shared with which open state variables of other systems. The overall intuitive picture of composite systems obtained by $\Dynam$ is that the change in a shared state variable is the sum of changes from contributing subsystems, where a UWD indicates which state variables are shared.

    Now, we explicitly unpack the definition of composition given by the $\Dynam$ UWD-algebra. First, given a UWD $\Phi\maps X_1+\dots + X_n\xrightarrow{l} J \xleftarrow{r} Y$ and component open dynamical systems, $(S_i, \upsilon_i\maps \R^{S_i}\to\R^{S_i}, m_i\maps X_i\to S_i)$ for all $i\in [n]$, the product comparison is applied to the component systems to produce a single system which can be thought of as the disjoint union of the input systems. Explicitly, this is the open system $(S\coloneqq S_1+\dots+S_n, \upsilon\maps \R^S\to \R^S, m\coloneqq m_1 + \dots + m_n)$, where the dynamics $\upsilon$ are given by
    \begin{equation}\label{eq:lax-d}
        \sum_{i=1}^n \iota_{S_i *}\circ\upsilon_i\circ\iota_{S_i}^*.
    \end{equation}
    Here, the $\iota_{S_i}$'s are the natural inclusions of $S_i$ into $S$, so \eqref{eq:lax-d} can be thought of as simply ``stacking" the dynamics of $\upsilon_1$ through $\upsilon_n$, with no interactions occurring between subsystems. Then, to complete the composition, let $X\coloneqq X_1+\dots+X_n$ and form the following pushout.
\[\begin{tikzcd}
	& X && Y \\
	S && J \\
	& {S+_XJ}
	\arrow["m", from=1-2, to=2-1]
	\arrow["l"', from=1-2, to=2-3]
	\arrow["r", from=1-4, to=2-3]
	\arrow["{p_S}"', from=2-1, to=3-2]
	\arrow["{p_J}", from=2-3, to=3-2]
	\arrow["\lrcorner"{anchor=center, pos=0.125, rotate=135}, draw=none, from=3-2, to=1-2]
\end{tikzcd}\]
    The composite open dynamical system is then
    \begin{equation}
        (S+_X J, p_{S*}\circ\upsilon \circ p_S^*, p_J\circ r).
    \end{equation}
    The dynamics of this composite system uses the \emph{distribute} algebra to distribute the current state vector to the subsystems and the \emph{collect} algebra to collect the results of each subsystem applied to its part of the state vector. Crucially, components of the state vector which inhabit shared junctions of $\Phi$ are copied to the subsystems incident to those junctions by the distribute algebra, and the results from each incident subsystem are summed by the collect algebra to obtain a state update for the junction.

    At this point, we introduce some notation to make it easier to write down composite systems. Given a collection of $m$ systems $\{\upsilon_j\maps\R^{N_j}\to\R^{N_j}\mid j\in[m]\}$, we denote the action of $\Dynam$'s product comparison $\varphi_{N_1,\dots,N_m}(\upsilon_1,\dots,\upsilon_m)$ as
    \begin{equation}
        \begin{bmatrix}
            \upsilon_1 \\
            \vdots \\
            \upsilon_m
        \end{bmatrix}.
    \end{equation}
    We use this notation to emphasize that the system obtained by applying the product comparison just ``stacks up" the individual input dynamical systems.

    We can slightly modify $\Dynam$ to produce a UWD-algebra for composing discrete dynamical systems, as in Proposition 3.1 of \cite{libkind_operadic_2022}. There is a finset algebra $\Dynam_D\maps(\FinSet,+)\to(\Set,\times)$ given by the following maps:
    \begin{itemize}
        \item On objects, $\Dynam_D$ takes a finite set $N$ to the set of discrete dynamical systems $\{\upsilon \maps \R^N\to\R^N\}$.
        \item Given a morphism $\phi\maps N\to M$ in $\FinSet$, $\Dynam_D(\phi)\maps\Dynam_D(N)\to\Dynam_D(M)$ is defined by the function $v\mapsto \id_{\R^M} + \phi_*\circ (v - \id_{\R^N})\circ \phi^*$, where the $\id$'s are in $\Vect$. 
        \item The comparison maps are the same as $\Dynam$.
    \end{itemize}    

    Proposition 3.3 of \cite{libkind_operadic_2022} shows that Euler's method for a given stepsize $\gamma>0$ defines the components of an algebra morphism from continuous to discrete systems. Specifically, there is a monoidal natural transformation $\Euler^\gamma\maps \Dynam\Rightarrow\Dynam_D$ with components
    \begin{equation}
        \Euler^\gamma_N\maps\Dynam(N)\to\Dynam_D(N),\;
        \upsilon\mapsto \id_{\R^N} + \gamma\upsilon.
    \end{equation}
\end{example}

\begin{remark}\label{rem:message_passing}
Systems inhabiting $\Dynam$-UWDs have a natural \emph{message passing semantics}. Specifically, at each time point, the ``distribute'' step is given by applying the distribute algebra to the current state vector to project the relevant components to their associated subsystems. The parallel computation step is given by applying each of the subsystem dynamics, and the ``collect'' step is given by applying the collect algebra to the results. These steps will enable our formalization of decomposition methods in the subsequent sections.
\end{remark}

\section{Gradient Descent is an Algebra Morphism}\label{sec:gd}
This section develops the fundamental result of the paper: that gradient descent is an algebra morphism from a UWD-algebra of unconstrained, differentiable minimization problems to the UWD-algebra of discrete dynamical systems (Example \ref{ex:dynam}). 

\subsection{Composing Optimization Problems}

To show that gradient descent is an algebra morphism, we must first specify our UWD-algebra of optimization problems. In this algebra, subproblems will compose by summing their objective functions subject to the pattern of decision variable sharing dictated by a UWD.

\begin{restatable}[]{lemma}{optFunctorial}
    There is a finset algebra $(\Opt,\varphi)\maps (\FinSet,+)\to (\Set,\times)$ defined as follows.
    \begin{itemize}
        \item On objects, $\Opt$ takes a finite set $N$ to the set of functions $\{f\maps\R^N\to \R \mid f\in C^1\}$.
        \item On morphisms, $\Opt$ takes a function $\phi\maps N\to M$ to the function $\Opt(\phi)\maps \Opt(N)\to\Opt(M)$ defined by
        \begin{equation}
            \Opt(\phi)\maps f\mapsto f\circ\phi^*.
        \end{equation}
        \item Given objects $N,M\in\FinSet$, the product comparison $\varphi_{N,M}\maps \Opt(N)\times \Opt(M)\to \Opt(N+M)$ is given by pointwise addition in $\R$, i.e.,
        \begin{equation}
            \varphi_{N,M}(f,g)(z)\coloneqq f(\iota_N^*(z)) + g(\iota_M^*(z))
        \end{equation}
        for all $z\in \R^{N+M}$. Here, $\iota_N\maps N\to N+M$ and $\iota_M\maps M\to N+M$ are the inclusions of $N$ and $M$ into their disjoint union.
        \item The unit comparison $\varphi_0\maps \1\to \Opt(\emptyset)$ picks out the zero function $0 \mapsto 0$.
    \end{itemize}
\end{restatable}

\begin{proof}
    We construct the covariant functor $\Opt$ as the composite
    \[
      \Opt \coloneqq \FinSet\xrightarrow{(-)^*}\Vect^\text{op}\xrightarrow{O}\Set
    \]
    of two contravariant functors, where $(-)^*$ is the distribute algebra from Example \ref{ex:linearmaps}.
    The contravariant functor $O$ takes a vector space $V$ to the set of $C^1$ functions $V\to \R$. Similarly to $(-)^*$, the functor $O$ acts on linear maps $T\maps V\to W$ by precomposition: $O(T)(f) \coloneqq f\circ T$. This is plainly functorial, and thus $\Opt \coloneqq O \circ (-)^*$ is functorial. It is also clear that precomposition of a $C^1$ function with a linear map is again $C^1$.

    
    We now need to verify that the product comparison is natural, i.e., that the following diagram commutes: 
\[\begin{tikzcd}
	{\Opt(X)\times\Opt(Y)} &&& {\Opt(N)\times\Opt(M)} \\
	\\
	{\Opt(X+Y)} &&& {\Opt(N+M)}
	\arrow["{\Opt(\phi)\times\Opt(\psi)}", from=1-1, to=1-4]
	\arrow["{\varphi_{X,Y}}"', from=1-1, to=3-1]
	\arrow["{\varphi_{N,M}}", from=1-4, to=3-4]
	\arrow["{\Opt(\phi+\psi)}"', from=3-1, to=3-4]
\end{tikzcd}\]
    for all finite sets $X,Y,N,M$ and functions $\phi\maps X\to N$ and $\psi\maps Y\to M$. To do so, let $(f,g)\in\Opt(X)\times\Opt(Y)$. Following the top path yields the function $h_1\coloneqq f\circ\phi^*\circ \iota_N^* + g\circ\psi^*\circ\iota_M^*$ while following the bottom path yields the function $h_2\coloneqq (f\circ\iota_X^* + g\circ\iota_Y^*)\circ (\phi+\psi)^*$. To show that these are equivalent, first note that since $(-)^*$ is a strong monoidal functor, the map $(\phi+\psi)^*$ is isomorphic to the map $\phi^*\oplus\psi^*$. Similarly, since $(-)^*$ takes coproducts in $\FinSet$ to products in $\Vect$, we know that $\iota_X^*\maps \R^{X+Y}\to \R^X$ is equivalent to the projection $\pi_{\R^X}\maps\R^X\oplus\R^Y\to\R^X$, and similarly for all the above $\iota^*$'s. Thus we can rewrite $h_1$ as $f\circ\phi^*\circ \pi_{\R^N} + g\circ\psi^*\circ\pi_{\R^M}$ and $h_2$ as $(f\circ \pi_{\R^X} + g\circ \pi_{\R^Y})\circ (\phi^*\oplus \psi^*)$.

    Finally, noting that the following pair of diagrams commutes
\begin{equation}\label{eq:oplus_comm}\begin{tikzcd}
	{\R^N\oplus\R^M} && {\R^X\oplus \R^Y} && {\R^N\oplus\R^M} && {\R^X\oplus\R^Y} \\
	{\R^N} && {\R^X} && {\R^M} && {\R^Y}
	\arrow["{\phi^*\oplus \psi^*}", from=1-1, to=1-3]
	\arrow["{\pi_{\R^X}}", from=1-3, to=2-3]
	\arrow["{\pi_{\R^N}}"', from=1-1, to=2-1]
	\arrow["{\phi^*}"', from=2-1, to=2-3]
	\arrow["{\phi^*\oplus\psi^*}", from=1-5, to=1-7]
	\arrow["{\psi^*}"', from=2-5, to=2-7]
	\arrow["{\pi_{\R^M}}"', from=1-5, to=2-5]
	\arrow["{\pi_{\R^Y}}", from=1-7, to=2-7]
\end{tikzcd}\end{equation}
    shows that $h_1$ is equivalent to $h_2$.

    The associativity and symmetry of the product comparison follow from the associativity and symmetry of pointwise addition of functions, while the unitality follows from the unitality of the zero function with respect to pointwise addition. 
\end{proof}

\begin{example}[$\Opt$-UWDs]\label{ex:opt}

Recall the UWD in Figure \ref{fig:uwds}.A, which we call $\Phi$. Note that the junctions of $\Phi$ are in one-to-one correspondence with outer ports, so we will use the simplification of Remark \ref{rem:simple}. Applying $\Opt$ to this UWD gives a function $\Opt(\Phi)\maps \Opt([2])\times \Opt([3])\times\Opt([3])\to\Opt([5])$ specifying how to compose three subproblem objectives of appropriate dimensions to yield a composite objective on $\R^5$. 

Given $C^1$ subproblems $f\maps \R^2\to \R,g\maps \R^3\to\R$, and $h\maps \R^3\to \R$, the result $\Opt(\Phi)(f,g,h)$ is the composite problem
\begin{equation}\label{eq:simple_ex}
    \textnormal{minimize } f(w,x) + g(u,w,y) + h(u,w,z),
\end{equation}
corresponding to the $\Opt$-UWD in Figure \ref{fig:uwds}.B. To see why, first note that the distribute algebra copies components of the decision variable to their respective subproblems dictated by the mapping of inner ports to junctions. Then, the product comparison of $\Opt$ simply sums the results of applying each objective to the components of the decision variable they received.
\end{example}

\subsection{Solving Composite Problems with Gradient Descent}

We can now prove that gradient descent defines an algebra morphism from $\Opt$ to $\Dynam_D$. Our approach is to first show that gradient flow gives a morphism to continuous systems and then recover gradient descent as a corollary by composing with the Euler's method algebra morphism.

\begin{restatable}[]{theorem}{flowCsp}\label{thm:flow}
    Recalling the symmetric monoidal functor $\Dynam$ from Example \ref{ex:dynam}, there is a monoidal natural transformation $\flow\maps \Opt\Rightarrow \Dynam$ with components 
    \begin{equation*}
    \flow_N\maps \Opt(N)\to \Dynam(N)
    \end{equation*}
    defined by
    \begin{equation}
        \flow_N(f)\coloneqq x\mapsto -\nabla f(x).
    \end{equation}
\end{restatable}
\begin{proof}
    We first need to verify naturality. For this, we need to ensure that
\[\begin{tikzcd}
	{\Opt(N)} && {\Opt(M)} \\
	\\
	{\Dynam(N)} && {\Dynam(M)}
	\arrow["{\Opt(\phi)}", from=1-1, to=1-3]
	\arrow["{\flow_N}"', from=1-1, to=3-1]
	\arrow["{\flow_M}", from=1-3, to=3-3]
	\arrow["{\Dynam(\phi)}"', from=3-1, to=3-3]
\end{tikzcd}\]
commutes for any choice of $N,M\in \FinSet$ and $\phi\maps N\to M$. Fix any optimization problem $f\in \Opt(N)$. The top path of the naturality square yields the vector field
\begin{equation}\label{eq:top_path}
    (y \in \R^M) \mapsto -\nabla (f(\phi^*(y))) = -\nabla(f(Ky)),
\end{equation}
where we use $K$ to denote the matrix representation of $\phi^*$ with respect to the standard basis. Then, by applying the chain rule, we have that \eqref{eq:top_path} is equal to 
\begin{equation}
    y\mapsto -K^T\nabla f(Ky).
\end{equation}
Likewise, following the bottom path yields
\begin{equation}\label{eq:botpath}
    y\mapsto \phi_*(-\nabla f(\phi^*(y))) = -\phi_*(\nabla f(\phi^*(y))).
\end{equation}
Noting that the pushforward is dual to the pullback shows that \eqref{eq:top_path} is equivalent to \eqref{eq:botpath}, as desired.

To verify that the natural transformation is monoidal, we need to show that the diagrams
\[\begin{tikzcd}
	{\Opt(N)\times\Opt(M)} &&& {\Dynam(N)\times\Dynam(M)} \\
	\\
	{\Opt(N+M)} &&& {\Dynam(N+M)}
	\arrow["{\flow_N\times \flow_M}", from=1-1, to=1-4]
	\arrow["{\flow_{N+M}}"', from=3-1, to=3-4]
	\arrow["{\varphi^{\Opt}_{N,M}}"', from=1-1, to=3-1]
	\arrow["{\varphi^{\Dynam}_{N,M}}", from=1-4, to=3-4]
\end{tikzcd}\]
and
\[\begin{tikzcd}
	\1 && {\text{Opt}(\emptyset)} \\
	&& {\text{Dynam}(\emptyset)}
	\arrow["{\varphi_0^\Opt}", from=1-1, to=1-3]
	\arrow["{\varphi_0^\Dynam}"', from=1-1, to=2-3]
	\arrow["{\flow(\emptyset)}", from=1-3, to=2-3]
\end{tikzcd}\]
commute for all $N,M\in\FinSet$. The unit comparison diagram commutes trivially (noting that the constant zero function $\R^0\to \R$ is the same as the unique vector field $\R^0\to\R^0$). For the product comparison diagram, let $f\in \Opt(N)$ and $g\in \Opt(M)$ be optimization problems. Following the top path yields the vector field
\begin{equation}
    z\mapsto \iota_{N*}(-\nabla f(\iota_N^*(z))) + \iota_{M*}(-\nabla g(\iota_M^*(z))), 
\end{equation}
while following the bottom path yields
\begin{equation}
    z\mapsto -\nabla(f(\iota_N^*(z))) -\nabla(g(\iota_M^*(z))).
\end{equation}
These vector fields are seen to be equivalent by the same reasoning used to verify naturality.
\end{proof}

\begin{corollary}
    Given a positive real number $\gamma$, there is a monoidal natural transformation $\gd^\gamma\maps \Opt\Rightarrow \Dynam_D$ with components $\gd^\gamma_N\maps \Opt(N)\to\Dynam_D(N)$ given by the function
    \begin{equation}
        f\mapsto \id_{\R^N} - \gamma\nabla f.
    \end{equation}
\end{corollary}
\begin{proof}
    The claimed monoidal transformation is the composite of monoidal transformations
    \[
    \Opt \xRightarrow{\flow} \Dynam \xRightarrow{\Euler^\gamma} \Dynam_D,
    \]
    where the Euler transformation was described at the end of Section \ref{sec:operads}.
\end{proof}

In words, the transformation $\gd^\gamma$ takes a differentiable minimization problem to the dynamical system which implements gradient descent on it. Naturality of gradient descent means that composing the gradient descent dynamical systems of subproblems is equivalent to composing subproblems and then taking the gradient descent dynamical system of the composite problem. This equivalence is what enables composite problems to be solved using the distributed message passing semantics of composite dynamical systems, illustrated by the following example.

\begin{example}[Distributed Gradient Descent]\label{ex:dgd}
Recall Problem \eqref{eq:simple_ex}:
\begin{equation*}
    \textnormal{minimize } f(w,x) + g(u,w,y) + h(u,w,z).
\end{equation*}
Here, we assume that $f\maps\R^2\to\R, g\maps\R^3\to\R$, and $h\maps\R^3\to\R$ are all $C^1$ functions. We showed in Example \ref{ex:opt} that this problem is the composite produced by filling the UWD $\Phi$ in Figure \ref{fig:uwds}.A with the component functions $f,g,$ and $h$. This composite has the form
\begin{equation}
    P(u,w,x,y,z) \coloneqq \boldsymbol{1}^T\begin{bmatrix} 
        f(w,x) \\ g(u,w,y) \\ h(u,w,z)
    \end{bmatrix}.
\end{equation}
 In other words, the composite problem takes an input vector, distributes it to the subproblems, and sums the results. If we apply $\gd^\gamma$ to the composite problem $P$ and let $s\coloneqq (u,w,x,y,z)^T$, we get the discrete system
\begin{equation}\label{eq:comp_system}
    s_{k+1} = s_k -\gamma\nabla P(s_k),
\end{equation}
for a chosen step-size $\gamma>0$.
This is a centralized gradient descent algorithm for minimizing $P$. However, by Theorem \ref{thm:flow} and the fact that $P$ is \emph{decomposed} by the UWD $\Phi$, we know that this system is equivalent to applying $\gd^\gamma$ to each individual subproblem and composing these discrete dynamical systems according to $\Phi$, i.e.,
\begin{equation}\label{eq:dynam_d_comp}
    \Dynam_D(\Phi)\Bigl((\id_{\R^2}-\gamma\nabla f), (\id_{\R^3}-\gamma\nabla g), (\id_{\R^3}-\gamma\nabla h)\Bigr).    
\end{equation}
Simplifying \eqref{eq:dynam_d_comp} results in the equivalent dynamical system
\begin{equation}\label{eq:distributed-system}
    s_{k+1} = s_k-\gamma K^T\begin{bmatrix}
        \nabla f \\ \nabla g \\ \nabla h
    \end{bmatrix}Ks,
\end{equation}
where $K$ is the pullback of the function mapping inner ports to junctions of $\Phi$.

Note that naturality of $\gd^\gamma$ says that \eqref{eq:comp_system} and \eqref{eq:distributed-system} are \emph{extensionally equal}, meaning they produce the same output given the same input. This also implies that the trajectories of these systems from a given initial condition will be the same. However, these systems have very different \emph{computational} properties.
Specifically, unlike the system in \eqref{eq:comp_system}, the system in \eqref{eq:distributed-system} can be run using the distributed message passing semantics discussed in Remark \ref{rem:message_passing}. Explicitly, the system \eqref{eq:distributed-system} yields Algorithm \ref{alg:dgd}, which is a simple distributed gradient method for minimizing \eqref{eq:simple_ex}.
\SetKwComment{Comment}{/* }{ */}

\begin{algorithm}\label{alg:dgd}
    \caption{Message Passing Gradient Descent}
    \KwIn{An initial state vector $s_0\in\R^5$.}
    $s\gets s_0$\;
    \While{a stopping criterion is not reached}{
        $t \gets K*s$ \Comment*[r]{Duplicate shared states}
        $t_1,t_2,t_3\gets \pi_1(t),\pi_2(t),\pi_3(t)$ \Comment*[r]{Project subvectors to subsystems}
        $t_1, t_2, t_3\gets \nabla f(t_1),\nabla g(t_2),\nabla h(t_3)$ \Comment*[r]{This can be done in parallel}
        $t\gets \texttt{concatenate}(t_1,t_2, t_3)$\;
        $\Delta s\gets \gamma*K^T*t$ \Comment*[r]{Sum changes in shared states}
        $s\gets s - \Delta s$ \Comment*[r]{Update state}
    }
    \Return{s}\;
\end{algorithm}

\end{example}

The process described in Example \ref{ex:dgd} of taking a UWD and component subproblems and generating a distributed solution algorithm for solving them can be automated by the meta-algorithm shown in Algorithm \ref{alg:meta}. 

\begin{algorithm}\label{alg:meta}
\caption{Minimization Algorithm Generator}
\KwIn{A UWD $\Phi$ with $n$ inner boxes}
\KwIn{Open $C^1$ objectives $f_1,\dots,f_n$ to fill each inner box of $\Phi$.}
\KwIn{A positive real step-size $\gamma$.}
\KwOut{A message passing dynamical system for minimizing $\Opt(\Phi)(f_1,\dots,f_n)$.}

\For{$i\gets1$ \KwTo $n$}{
    $v_i\gets \gd^\gamma(f_i)$\;
    }

\Return{$\Dynam_D(\Phi)(v_1,\dots,v_n)$}\;

\end{algorithm}

\begin{example}[Recovering Primal Decomposition]\label{ex:primal_decomp}
    
\begin{figure}
    \centering
    \includegraphics[width=\textwidth]{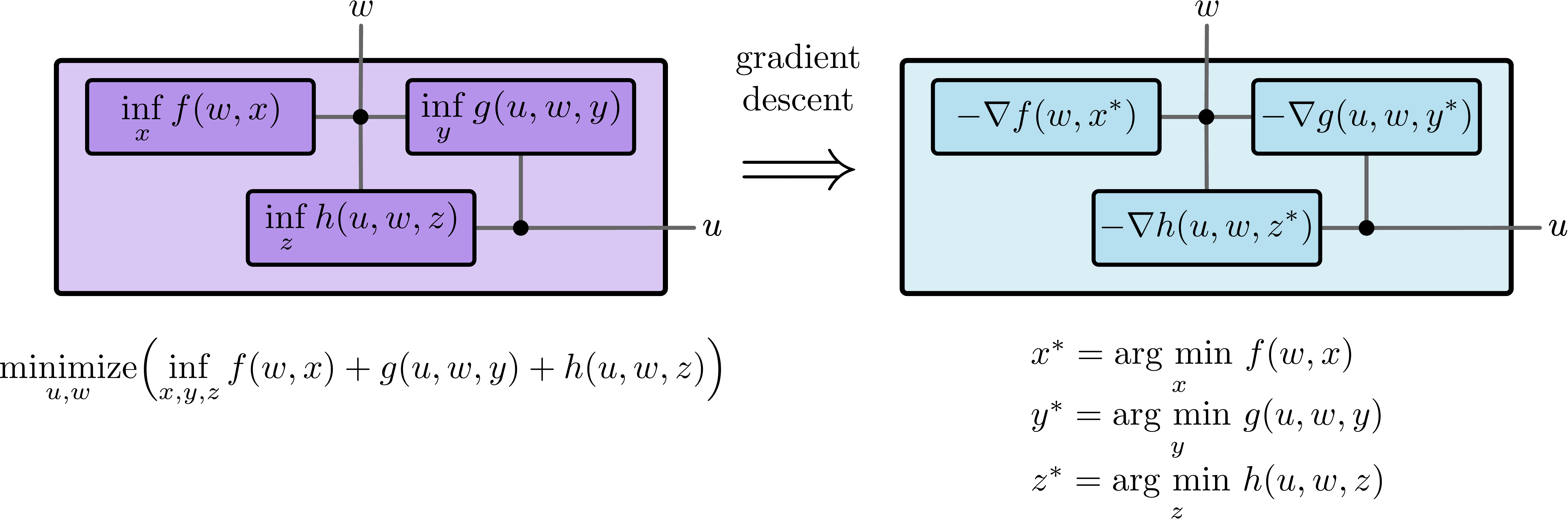}
    \caption{The $\Opt$-UWD setting up primal decomposition of Problem \eqref{eq:simple_ex}.}
    \label{fig:primal-decomp}
\end{figure}

Example \ref{ex:dgd} shows that applying the natural transformation $\gd^\gamma\maps \Opt\Rightarrow\Dynam$ gives a simple distributed gradient method that decomposes $C^1$ problems defined on arbitrary UWDs.
However, this simple distributed gradient method still requires synchronization after each gradient step.
Supposing that $f,g,$ and $h$ are $C^1$, strictly convex functions, an alternative distributed solution method for \eqref{eq:simple_ex} can be obtained by applying primal decomposition. At each iteration, primal decomposition fixes the values of shared variables to decouple the problem, finds minimizers of the subproblems in parallel, and uses these minimizers to compute gradient updates for the shared variables. We can implement primal decomposition in this framework by applying the same gradient descent natural transformation to a slightly modified UWD and component problems that fill it. Specifically, we apply gradient flow to the $\Opt$-UWD shown in Figure \ref{fig:primal-decomp}.

Because $f,g,$ and $h$ are strictly convex, their infima are attained at unique points for any values of $w$ and $u$. Thus, the gradient of $\inf_x f(w,x)$ can be computed as $\nabla_w f(w,x^*)$ where $x^*=\argmin_x f(w,x)$. A similar computation produces gradients of the other subproblems with respect to $w$ and $u$. Applying $\gd^\gamma$ to each subproblem and composing these systems results in the composite system
\begin{equation}
    \begin{bmatrix}u_{k+1}\\w_{k+1}\end{bmatrix}=\begin{bmatrix}
        u_k\\w_k
    \end{bmatrix}-\gamma Q^T\begin{bmatrix}
        \nabla f(w_k,x^*) \\ \nabla g(u_k,w_k,y^*) \\ \nabla h(u_k,w_k,z^*)
    \end{bmatrix},
\end{equation}
where $x^*=\argmin_x f(w_k,x), y^*=\argmin_y g(u_k,w_k,y), z^*=\argmin_z h(u_k,w_k,z)$, and $Q^T$ is the matrix obtained by applying the collect algebra to the UWD in Figure \ref{fig:primal-decomp}. This system yields Algorithm \ref{alg:pd}.

\SetKwComment{Comment}{/* }{ */}
\begin{algorithm}\label{alg:pd}
    \caption{Message Passing Primal Decomposition}
    \KwIn{An initial state vector $s_0\in\R^2$.}
    $s\gets s_0$\;
    \While{a stopping criterion is not reached}{
        $t \gets Q*s$ \Comment*[r]{Duplicate shared states}
        $t_1,t_2,t_3\gets \pi_1(t),\pi_2(t),\pi_3(t)$ \Comment*[r]{Project subvectors to subsystems}
        Compute the following argmins in parallel using an appropriate algorithm\;
        $x^*\gets \argmin_x f(t_1,x)$\;
        $y^*\gets \argmin_y g(t_2,y)$\;
        $z^*\gets \argmin_z h(t_3,z)$\;
        $t_1, t_2, t_3\gets \nabla f(t_1,x^*),\nabla g(t_2,y^*),\nabla h(t_3,z^*)$ \Comment*[r]{This can be done in parallel}
        $t\gets \texttt{concatenate}(t_1,t_2, t_3)$\;
        $\Delta s\gets \gamma*Q^T*t$ \Comment*[r]{Sum changes in shared states}
        $s\gets s - \Delta s$ \Comment*[r]{Update state}
    }
    \Return{s}\;
\end{algorithm}

This is primal decomposition applied to Problem \eqref{eq:simple_ex}. Naturality of gradient flow automatically proves that running this algorithm will produce the same result as running the equivalent synchronous dynamical system obtained by applying $\gd^\gamma$ to the original composite problem. Note that although Algorithm \ref{alg:dgd} and Algorithm \ref{alg:pd} are different, they both the result from applying meta-Algorithm \ref{alg:meta} to different UWDs and component objectives.
\end{example}

\subsection{The Compositional Data Condition}

Often, an instance of a specific optimization problem is determined by some data. Examples include how a neural network architecture is incorporated into a machine learning objective or how a specific graph defines a network flow objective. In general, we can view such encodings of problem data into objective functions as families of functions $p_N\maps D(N)\to \Opt(t(N))$ indexed by finite sets $N$. We think of $D(N)$ as the set of all possible problem data defined on $N$ (e.g., the set of graphs with vertex set $N$, the set of neural networks with $|N|$ weights, etc.). The function $t$ is then a transformation of $N$ into a set of decision variables for a corresponding objective. Sometimes, the problem data itself can be composed. For example, graphs can be glued together along their vertices, and neural networks can share weights across different components (weight tying).  We refer to data of these kinds as \define{compositional data}. Our \emph{compositional data condition} essentially says that if the compositional structure of a problem's data is compatible with the compositional structure of objectives in $\Opt$, then the problem can be solved via a decomposition method. Formally, this is stated as follows.


\begin{definition}[Compositional Data Condition]\label{def:cdc}
    Let $t\maps (\FinSet,+)\to (\FinSet,+)$ and $D\maps (\FinSet,+)\to (\Set,\times)$ be  lax symmetric monoidal functors. Given a family of functions $p_N\maps D(N)\to \Opt(t(N))$ for $N\in \FinSet$, we say that $p$ satisfies the \define{compositional data condition} if the maps $p_N$ form the components of a monoidal natural transformation $p\maps D\Rightarrow \Opt\circ t$.
    
\end{definition}


The compositional data condition is a sufficient condition for when a problem defined by compositional data is decomposable. Specifically, when the condition is satisfied, we obtain a monoidal natural transformation from an algebra $D$ of data to the algebra of dynamical systems by post-composition, i.e.,
\[D\xRightarrow{p} \Opt\circ t\xRightarrow{\gd^\gamma * t} \Dynam_D \circ t.\]
This provides a transformation from problem data into systems for minimizing the objectives defined by this data. The benefit of this framing is that the systems produced by this composite will respect any hierarchical decomposition of the input data encodable as an undirected wiring diagram. As an example, we will see how the minimum cost network flow problem satisfies the compositional data condition in Section \ref{sec:netflow}.


\section{Uzawa's Algorithm is an Algebra Morphism}\label{sec:uzawa}

So far, we have shown that gradient descent is natural with respect to composition for unconstrained minimization problems. To extend this framework to equality constrained problems, we show that Uzawa's algorithm \cite{uzawa_1960} is natural for equality constrained minimization problems. This result will also allow us to extend the compositional data condition to a broader class of problems. Doing so first requires a notion of composition for constrained problems. Recall that Uzawa's algorithm uses saddle functions as representations of constrained problems. Therefore, we specify how to compose constrained problems by specifying how to compose their corresponding saddle functions.

\subsection{Composing Saddle Problems}

Saddle functions should compose like unconstrained objective functions do under $\Opt$, namely by summing objectives subject to a pattern of variable sharing dictated by a UWD. The crucial difference with saddle problems is that we need to guarantee that the composite of saddle problems is again a saddle problem. One way to achieve this is to enforce that the pattern of variable sharing must respect the saddle properties of the subproblems being composed. Specifically, we only want to allow saddle problems to share a given variable if they are either both convex in that variable or both concave in it. 
If one is convex in that variable and the other is concave, then they should not be allowed to share that variable.

Therefore, we need a way to track whether a given saddle function depends in a convex or concave way upon each of its inputs. For this, we make use of a categorical construction known as the \emph{slice category}, which can be thought of as imposing a ``type system" on a category.

\begin{definition}
    Given a category $\CC$ and an object $c\in\CC$, the \define{slice category} $\CC/c$ is defined as follows.
    \begin{itemize}
        \item The objects of $\CC/c$ are morphisms $f\maps x\to c$ in $\CC$ with codomain $c$.
        \item A morphism from $f\maps x\to c$ to $g\maps y\to c$ in $\CC/c$ is a morphism $h\maps x\to y$ in $\CC$ such that the triangle
\[\begin{tikzcd}
	x && y \\
	& c
	\arrow["f"', from=1-1, to=2-2]
	\arrow["g", from=1-3, to=2-2]
	\arrow["h", from=1-1, to=1-3]
\end{tikzcd}\]
    commutes, i.e., $f=g\circ h$. 
    \end{itemize}
    A slice category $\CC/c$ has an associated \emph{forgetful functor} $U\maps \CC/c\to\CC$ defined by taking objects $x\xrightarrow{f}c$ in $\CC/c$ to their domains $x\in\CC$, and by taking commuting triangles to the morphisms which made them commute.

\end{definition}

The distinction between convex and concave is binary, 
so the slice category we use to encode the distinction is $\FinSet/[2]$. We can think of this category as essentially the same as $\FinSet$, but with each element of a set being ``labelled" with either a 1 or a 2, representing convex-labelled and concave-labelled components respectively. We note that linear functions are both convex and concave and in principle they
can be labeled with either. However, in the engineering literature, one conventionally regards a linear function as convex
when it is being minimized and concave when it is being maximized, and we adopt this convention here as well.
To produce an algebra of saddle functions, we will require that an objective function defined on such a labelled set be convex in the convex-labelled components and concave in the concave-labelled components. 

The morphisms in $\FinSet/[2]$ can then be thought of as label-preserving functions, i.e., functions which map convex-labelled elements of their domain set to convex-labelled elements of their codomain set and likewise for concave-labelled elements. Note that colimits in a slice category are constructed as colimits in the original category (\cite{riehl_context_2017}, Proposition 3.3.8). In particular, the coproduct of two objects $N\xrightarrow{\tau} [2]$ and $M\xrightarrow{\sigma} [2]$ in $\FinSet/[2]$ is the coproduct $N+M$ in $\FinSet$ labelled by the copairing $[\tau,\sigma]$. Therefore, $(\FinSet/[2],+,\emptyset\xrightarrow{!}[2])$ inherits a co-Cartesian monoidal structure from that of $(\FinSet, +, \emptyset)$.

\begin{remark}\label{rem:typing}
Given an object $\tau\maps N\to [2]$ in $\FinSet/[2]$, the sets $N_1\coloneqq \tau^{-1}(1)$ and $N_2\coloneqq \tau^{-1}(2)$ form a partition of $N$. 
We also have two unique injections $N_1\xhookrightarrow{\textnormal{cnvx}} N$ defined by $(\tau\circ \textnormal{cnvx})(n_1) = 1$ for all $n_1\in N_1$ and $N_2\xhookrightarrow{\textnormal{cncv}} N$ defined by $(\tau\circ \textnormal{cncv})(n_2) = 2$ for all $n_2\in N_2$. In other words, `cnvx' picks out the convex-labelled components of $N$ and `cncv' picks out the concave-labelled components.

Furthermore, utilizing the functor $(-)_*$ to generate $\R^{N_1}$ and $\R^{N_2}$ and taking the pushforward of `cnvx' and `cncv', we define the isomorphism $u\maps \R^{N_1}\oplus\R^{N_2}\cong \R^N$ as the unique arrow in the following coproduct diagram:
\[\begin{tikzcd}
	& {\R^N} \\
	\\
	{\R^{N_1}} & {\R^{N_1}\oplus\R^{N_2}} & {\R^{N_2}}
	\arrow["u"', dashed, hook, two heads, from=3-2, to=1-2]
	\arrow["{\iota_{\R^{N_1}}}"', from=3-1, to=3-2]
	\arrow["{\iota_{\R^{N_2}}}", from=3-3, to=3-2]
	\arrow["{\textnormal{cnvx}_*}", from=3-1, to=1-2]
	\arrow["{\textnormal{cncv}_*}"', from=3-3, to=1-2]
\end{tikzcd}\]
    This isomorphism simply takes a pair $(x,y)\in\R^{N_1}\oplus\R^{N_2}$ and arranges them into a single vector $z\in\R^N$ whose ordering of convex and concave components is given by $\tau$.
\end{remark}



\begin{restatable}[]{lemma}{saddleFunctor}
    There is a finset algebra $\Saddle\maps (\FinSet/[2],+)\to(\Set,\times)$ defined as follows.
    \begin{itemize}
        \item Given an object $\tau\maps N\to [2]$ in $\FinSet/[2]$, let $N_1,N_2$ and $u\maps \R^{N_1}\oplus\R^{N_2}\cong \R^N$ be defined as in Remark \ref{rem:typing}. Then $\Saddle$ takes $\tau$ to the set
        \begin{multline}
            \Big\{f\maps \R^N\to \R \mid f(u(\cdot,y)) \text{ is convex for any fixed } y\in\R^{N_2}\text{ } \textnormal{ and } \\ f(u(x,\cdot)) \text{ is concave for any fixed } x\in \R^{N_1}\Big\}.
            \end{multline}
        
        \item On morphisms, $\Saddle$ takes a commuting triangle 
\[\begin{tikzcd}[ampersand replacement=\&]
	N \&\& M \\
	\& {[2]}
	\arrow["\phi", from=1-1, to=1-3]
	\arrow[""{name=0, anchor=center, inner sep=0}, "{\tau_1}"', from=1-1, to=2-2]
	\arrow[""{name=1, anchor=center, inner sep=0}, "{\tau_2}", from=1-3, to=2-2]
\end{tikzcd}\]
        
    to the function $\Saddle(\tau_1)\to\Saddle(\tau_2)$ defined by precomposition: $f\mapsto f\circ \phi^*$. 
    \item Given objects $\tau_1\maps N\to [2]$ and $\tau_2\maps M\to [2]$ in $\FinSet/[2]$, the product comparison $\varphi_{\tau_1,\tau_2}\maps\Saddle(\tau_1)\times\Saddle(\tau_2)\to\Saddle(\tau_1+\tau_2)$ is defined by
    \begin{equation}
        \varphi_{\tau_1,\tau_2}(f,g)(z)\coloneqq f(\iota_N^*(z)) + g(\iota_M^*(z)),
    \end{equation}
    for all $z\in \R^{N+M}$.
    \item The unit comparison $\varphi_0\maps \1\to \Saddle(\emptyset\xrightarrow{!}[2])$ is given by the constant 0 function.
    \end{itemize}
\end{restatable}
\begin{proof}
    First, we must prove that given a saddle function $f\in\Saddle(N\xrightarrow{\tau} [2])$ and a map $\phi\maps (N\xrightarrow{\tau}[2])\to (M\xrightarrow{\sigma}[2])$, the composition~$f\circ\phi^*$ is again a saddle function whose convex and concave components respect $\sigma$. The precomposition of a saddle function with a linear map is again a saddle function (\cite{schiele_disciplined_2023}, \S 2.1) so $f\circ \phi^*$ is a saddle function because $f$ is a saddle function and $\phi^*$ is linear. 
    Furthermore, since $\phi$ makes the diagram
\[\begin{tikzcd}[ampersand replacement=\&]
	N \&\& M \\
	\& {[2]}
	\arrow["\phi", from=1-1, to=1-3]
	\arrow[""{name=0, anchor=center, inner sep=0}, "{\tau}"', from=1-1, to=2-2]
	\arrow[""{name=1, anchor=center, inner sep=0}, "{\sigma}", from=1-3, to=2-2]
\end{tikzcd}\]
    commute, $\phi^*$ must map convex labelled components of $x\in\R^M$ to convex labelled componenets in $\R^N$ and likewise for concave labelled components. Thus $f\circ \phi^*$ has convex and concave components that respect $\sigma$. With this verified, the proof of functoriality of $\Saddle$ is the same as the proof of functoriality of $\Opt$ because $\Saddle(\phi) = \Opt(\phi)$.

    Similarly, because the comparison maps of $\Saddle$ act the same as those of $\Opt$, the proofs of the necessary unitality, symmetry, and associativity diagrams carry over. We then just need to verify that the results of applying the comparison maps are again saddle functions with the correct labelling of convex and concave components. For the unit comparison, the constant 0 function is affine and thus trivially a saddle function. Furthermore, it has no arguments and thus no labelling to respect.

    For the product comparison, consider arbitrary objects $N\xrightarrow{\tau_1} [2], M\xrightarrow{\tau_2} [2]$ in $\FinSet/[2]$ and arbitrary saddle functions $f\in\Saddle(\tau_1)$ and $g\in \Saddle(\tau_2)$. Let $N_1, N_2, M_1, M_2, u_N\maps \R^{N_1}\oplus \R^{N_2}\cong \R^N$, and $u_M\maps\R^{M_1}\oplus\R^{M_2}\cong\R^M$ be defined as in Remark \ref{rem:typing}. Then we know that 
    \begin{equation}
        f(u_N(\cdot, y)) + g(u_M(\cdot, z))
    \end{equation}
    is convex for all $(y,z)\in\R^{N_2}\oplus \R^{M_2}$ because the sum of convex functions is convex. Likewise, we have that
    \begin{equation}
        f(u_N(x,\cdot)) + g(u_M(w,\cdot))
    \end{equation}
    is concave for all $(x,w)\in \R^{N_1}\oplus \R^{M_1}$ because the sum of concave functions is concave. This establishes the desired saddle property for the result of applying the product comparison maps.

\end{proof}

\subsection{Additional Classes of Optimization Problems}


The functor $\Opt$ gives a way to compose unconstrained $C^1$ non-convex minimization problems while $\Saddle$ gives a way to compose saddle problems which are not necessarily differentiable. By adding additional restrictions on the sets of functions in the image of these functors or restricting their domains, we can easily produce new functors for composing other types of problems. In this paper specifically, we will make use of the following restrictions: 
\begin{itemize}
    \item $\DiffSaddle$ restricts functions in the image of $\Saddle$ to be continuously differentiable, i.e., \[\DiffSaddle(N)\coloneqq \{L\in\Saddle(N)\mid L\in C^1\}.\]
    \item We get a subcategory of $\FinSet/[2]$ isomorphic to $\FinSet$ by only looking at objects of the form $\tau\maps N\to [2]$ for which $\tau(n)=1$ for all $n\in N$. Restricting $\Saddle$ to this subcategory gives us a finset algebra $\Conv\maps\FinSet\to\Set$ defining composition of convex optimization problems. Similarly, restricting $\Saddle$ to the subcategory for which $\tau(n)=2$ for all $n\in N$ gives us a finset algebra $\Conc\maps\FinSet\to\Set$ defining composition of concave optimization problems.
    \item Finally, we can perform the same restrictions as above to $\DiffSaddle$ to yield finset algebras 
    \[\DiffConv\maps \FinSet\to\Set \text{ and } \DiffConc\maps\FinSet\to\Set\] for composing continuously differentiable convex and concave problems, respectively.
\end{itemize}
Figure \ref{fig:hierarchy} in the Introduction shows a graphical view of the relationship between these various restrictions.

\subsection{Solving Composite Saddle Problems with Uzawa's Algorithm}

We can now show that Uzawa's algorithm, i.e., gradient descent on convex components and gradient ascent on concave components, respects the compositional structure of saddle problems. This will allow us to derive distributed versions of both Uzawa's algorithm  
and dual decomposition.

\begin{restatable}{theorem}{saddle_flow}
    Let $U\maps \FinSet/[2]\to\FinSet$ be the forgetful functor from the slice category. Then there is a monoidal natural transformation $\flow\maps\DiffSaddle\Rightarrow\Dynam\circ U$ with components
    \begin{equation}
        \flow_{N\xrightarrow{\tau}[2]}\maps \DiffSaddle(\tau)\to \Dynam(N)
    \end{equation}
    defined component-wise by
    \begin{equation}
        \flow_\tau(f)(x)_i\coloneqq \begin{cases}
            -\nabla f(x)_i & \tau(i) = 1 \\
            \nabla f(x)_i & \tau(i) = 2
        \end{cases}
    \end{equation}
    for all $x\in\R^N$ and $i\in N$.
\end{restatable}
\begin{proof}
    To verify naturality, we need to show that the diagram 
\begin{equation}\label{eq:comm_tri}\begin{tikzcd}
	{\DiffSaddle(\tau)} && {\DiffSaddle(\sigma)} \\
	{\Dynam(N)} && {\Dynam(M)}
	\arrow["{\DiffSaddle(\phi)}", from=1-1, to=1-3]
	\arrow["{\Dynam(\phi)}"', from=2-1, to=2-3]
	\arrow["{\flow_\tau}"', from=1-1, to=2-1]
	\arrow["{\flow_\sigma}", from=1-3, to=2-3]
\end{tikzcd}\end{equation}
    commutes for all $N\xrightarrow{\tau}[2]$, $M\xrightarrow{\sigma}[2]$, and commuting triangles 
\[\begin{tikzcd}[ampersand replacement=\&]
	N \&\& {M.} \\
	\& {[2]}
	\arrow["\phi", from=1-1, to=1-3]
	\arrow[""{name=0, anchor=center, inner sep=0}, "{\tau}"', from=1-1, to=2-2]
	\arrow[""{name=1, anchor=center, inner sep=0}, "{\sigma}", from=1-3, to=2-2]
\end{tikzcd}\]
    Conceptually, this proof will be similar to the proof of naturality in Theorem \ref{thm:flow}, with the crucial difference being that we also need to verify that $\Dynam(\phi)\circ\flow_\tau$ performs gradient descent on the same components as $\flow_\sigma\circ\DiffSaddle(\phi)$ and likewise for gradient ascent.

    To begin, let $f\in\DiffSaddle(\tau)$ be arbitrary. Following the bottom path yields the following dynamical system defined component-wise:
    \begin{equation}\label{eq:saddle_top}
        \Dynam(\phi)(\flow_\tau(f))(y)_j = \sum_{i\in\phi^{-1}(j)}\begin{cases}
            -[\nabla f(\phi^*(y))]_i & \tau(i)=1 \\
            [\nabla f(\phi^*(y))]_i & \tau(i)=2,
        \end{cases}
    \end{equation}
    for all $y\in\R^M$ and $j\in M$. The summation in \eqref{eq:saddle_top} comes from the component-wise definition of the pushforward. Likewise, following the bottom path yields the system
    \begin{equation}\label{eq:saddle_bot}
        \flow_\sigma(f\circ\phi^*)(y)_j = \begin{cases}
            -\nabla[f(\phi^*(y))]_j & \sigma(j) = 1 \\
            \nabla[f(\phi^*(y))]_j & \sigma(j) =2,
        \end{cases}
    \end{equation}
    for all $y\in\R^M$ and $j\in M$. By applying the chain rule and noting that the pushforward is dual to the pullback (same reasoning as in Theorem \ref{thm:flow}, we have that \eqref{eq:saddle_bot} equals
    \begin{equation}
        \sum_{i\in\phi^{-1}(j)}\begin{cases}
            -[\nabla f(\phi^*(y))]_i & \sigma(\phi(i))=1 \\
            [\nabla f(\phi^*(y))]_i & \sigma(\phi(i))=2.
        \end{cases}
    \end{equation}
    Crucially, because \eqref{eq:comm_tri} is a commuting triangle, we have that $\tau(i) = \sigma(\phi(i))$ for all $i\in N$. This shows that \eqref{eq:saddle_top} equals \eqref{eq:saddle_bot} as desired.

    To prove that $\flow$ is monoidal, it again suffices to combine the above reasoning with that of Theorem \ref{thm:flow}.
\end{proof}

\begin{corollary}\label{thm:uzawa}
    Gradient ascent-descent gives a monoidal natural transformation $\gad^\gamma\maps \DiffSaddle\Rightarrow \Dynam_D\circ U$ with components
    \begin{equation}
        \gad^\gamma_{N\xrightarrow{\tau}[2]}\maps \DiffSaddle(\tau)\to\Dynam_D(N)
    \end{equation}
    defined by
    \begin{equation}
        \gad^\gamma_\tau(f)(x)_i\coloneqq \begin{cases}
            x_i-\gamma\nabla f(x)_i & \tau(i) = 1 \\
            x_i+\gamma\nabla f(x)_i & \tau(i) = 2
        \end{cases}
    \end{equation}
    for all $x\in\R^N$ and $i\in N$.
\end{corollary}
\begin{proof}
    This follows by composing $\flow\maps \DiffSaddle\Rightarrow\Dynam\circ U$ with $\Euler^\gamma$.
\end{proof}

Unpacking the definition of this natural transformation, it is simply implementing gradient \emph{descent} on the convex components of the objective and gradient \emph{ascent} on the concave components. With this proven, we can extend the compositional data condition to the case when the objectives defined by data are in $\DiffSaddle$, i.e., when $D, p,$ and $t$ are chosen such that
\[D\xRightarrow{p}\DiffSaddle\circ t\xRightarrow{\gad^\gamma * t}\Dynam_D\circ U \circ t\]
is a monoidal natural transformation.

\begin{example}[Distributed Uzawa's Algorithm and Dual Decomposition]\label{ex:dual-decomp}
    Similar to how applying $\gd^\gamma$ to $\Opt$-UWDs gave us distributed gradient and primal decomposition algorithms in Examples \ref{ex:dgd} and \ref{ex:primal_decomp} respectively, we can apply $\gad^\gamma$ to $\DiffSaddle$-UWDs to obtain distributed primal-dual and dual decomposition algorithms. 

    \begin{figure}
        \centering
        \includegraphics[width=\textwidth]{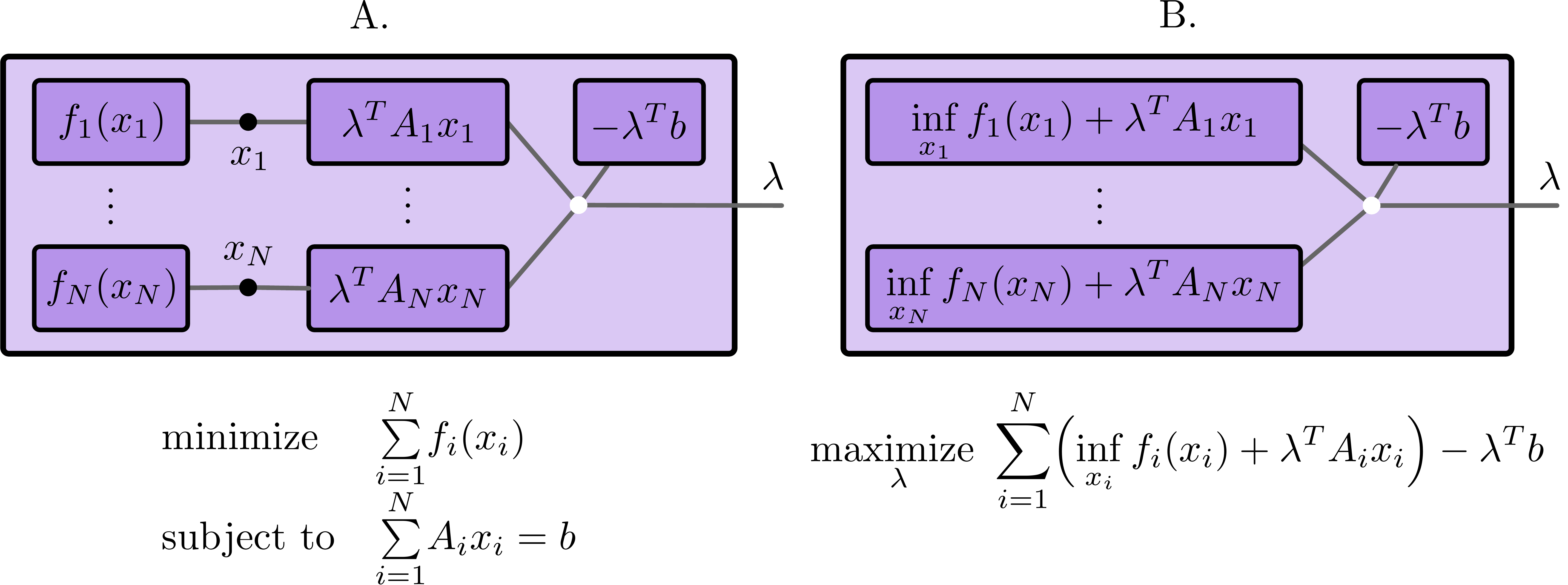}
        \caption{A. A $\DiffSaddle$-UWD encoding the structure of a separable sum of objective functions subject to a complicating equality constraint. B. A $\DiffSaddle$-UWD setting up dual decomposition of the problem encoded in (A). Since these are $\DiffSaddle$-UWDs, we denote the convex-labelled variables with black junctions and the concave labelled variables with white junctions.}
        \label{fig:dual-decomp}
    \end{figure}
    
    A common scenario, which will be revisited in our network flow example in Section~\ref{sec:netflow}, 
    is having a separable sum of objective functions subject to some complicating equality constraint. Following the discussion in Section 2.2 of \cite{boyd_distributed_2010}, a general example is the problem
    \begin{equation}\label{eq:cons-prob}
        \begin{array}{ll@{}ll}
            \textnormal{minimize } & \sum_{i=1}^N f_i(x_i) \\
            \textnormal{subject to } & Ax=b,
        \end{array}
    \end{equation}
    where the decision variable $x\in\R^n$ is partitioned into subvectors $x = [x_1,\dots,x_N]^T$ where $x_i\in\R^{n_i}$ and $\sum_i n_i=n$. Here, $A\in \R^{m\times n}, b\in\R^m$, and we assume that all the $f_i\maps\R^{n_i}\to \R$ are strictly convex. We can then partition $A$ conformably to match the partition of $x$ by using $A=[A_1,\dots,A_N]$ such that
    \begin{equation}
        Ax = \sum_{i=1}^N A_ix_i,
    \end{equation}
    where $A_i\in\R^{m\times n_i}$.
    With this partitioning, the Lagrangian is also separable in $x$ and can be written as
    \begin{equation}
        L(x,\lambda) = \sum_{i=1}^N \Bigl(f_i(x_i) + \lambda^TA_ix_i\Bigr) - \lambda^Tb.
    \end{equation}
    The separable structure of $L$ can then be encoded as the $\DiffSaddle$-UWD in Figure \ref{fig:dual-decomp}.A. Now, applying $\gad^\gamma$ to this UWD yields a distributed primal-dual solution algorithm which requires synchronization after each gradient step. Specifically, at every iteration, the gradient with respect to $\lambda$ and the $x_i$'s are computed locally for each box. These local gradients are then collected and summed along junctions to produce a gradient update for the total system.
    Like Example \ref{ex:primal_decomp} which modified a distributed gradient descent UWD to a primal decomposition UWD, we can perform a similar modification to convert our distributed primal-dual UWD to a dual decomposition UWD. For this, we take the dual function 
    \begin{equation}
        q(\lambda)\coloneqq \inf_x L(x,\lambda) = \sum_{i=1}^N\Bigl(\inf_{x_i}f_i(x_i) + \lambda^TA_ix_i\Bigr) - \lambda^Tb,
    \end{equation}
    which is again separable in $x$. The dual function can then be encoded as the $\DiffSaddle$-UWD shown in Figure \ref{fig:dual-decomp}.B. Applying $\gad^\gamma$ to this UWD yields the standard dual decomposition algorithm for solving \eqref{eq:cons-prob}.
\end{example}

    

\section{Subgradient Descent as a Functor}\label{sec:subg}
We now have algebra morphisms which provide distributed solution algorithms to compositions of unconstrained and equality constrained \emph{differentiable} problems. The goal of this section is to extend these results to their non-differentiable counterparts, a key aspect
of which is the generalization from gradients to subgradients. 
To incorporate subgradient methods into our framework, we need to show that they also provide an algebra morphism from a UWD-algebra of optimization problems to a UWD-algebra of dynamical systems. We can use the $\Conv$ UWD-algebra of convex non-differentiable problems as the domain of this algebra morphism. However, because the subgradient is a point-to-set mapping, $\Dynam$ will not suffice as the codomain. Thus, we first generalize $\Dynam$ and $\Dynam_D$ to algebras of non-deterministic continuous and discrete dynamical systems, which we call $\NDD$ and $\NDD_D$ respectively. We then show that Euler's method is still natural in this non-deterministic setting. Finally, this result allows us to prove that the subgradient method gives an algebra morphism from $\Conv$ to $\NDD_D$.

\subsection{Non-deterministic Dynamical Systems}
In general, a non-deterministic continuous dynamical system is given by a map $\upsilon\maps \R^N\to P\R^N$, where $P\R^N$ denotes the powerset of the set of vectors in $\R^N$. The map $\upsilon$ can be thought of as taking the current state to the set of possible update directions and is thus identified with the differential inclusion
\begin{equation}
    \dot{x}\in \upsilon(x).
\end{equation}

To construct an algebra for composing such non-deterministic systems, we first need to review the powerset functor, which lifts functions $X\to Y$ to functions $PX\to PY$.

\begin{lemma}\label{lem:pset_func}
    There is a functor $P\maps \Set\to\Set$ defined by the following maps:
    \begin{itemize}
        \item On objects, $P$ takes a set $X$ to its powerset $P(X)\coloneqq \{A : A\subseteq X\}$.
        \item Given a function $f\maps X\to Y$, the mapping $P(f)\maps P(X)\to P(Y)$ is the \emph{image} function
        \begin{equation}
            A\mapsto \bigcup_{a\in A} f(a)
        \end{equation}
        for $A\subseteq X$.
    \end{itemize}
\end{lemma}
\begin{proof}
    See Example 1.3.2.i in \cite{riehl_context_2017}.
\end{proof}

Henceforth, when we speak of the powerset applied to a vector space, we will always mean the powerset of the vector space's underlying set of vectors. We will also need the following lemma, which says that the powerset functor applied to a linear map is again linear with respect to the Minkowski sum. Recall that the Minkowski sum of two sets $A$ and $B$ with an addition defined on them is the set $\{a+b\mid a\in A, b\in B\}$. We denote the Minkowski sum as $A\boxplus B$ to avoid confusion with our notation for coproduct. Likewise, we denote the Minkowski difference $A\boxplus (-B)$ as $A\boxminus B$.

\begin{lemma}\label{lem:mink_dist}
    Given vector spaces $X$ and $Y$, subsets $A,B\subseteq X$, and a linear map $T\maps X\to Y$, we have $(PT)(A\boxplus B) = (PT)(A) \boxplus (PT)(B)$.
\end{lemma}
\begin{proof}
    We have the following
    \begin{align}
        (PT)(A\boxplus B) &= \{T(a+b)\mid a\in A, b\in B\} \nonumber \\
        &= \{Ta + Tb\mid a\in A,b\in B\} \nonumber \\ 
        &= (PT)(A)\boxplus(PT)(B), \nonumber
    \end{align}
    where the second equality comes from linearity of $T$ and the remaining equalities are definitions.
\end{proof}

Composition of non-deterministic systems is similar to that of deterministic systems, but utilizes the powerset functor and Minkowski summation to lift the deterministic parts of $\Dynam$ to their non-deterministic counterparts.

\begin{lemma}
    There is a finset algebra $\NDD\maps (\FinSet,+)\to (\Set,\times)$ defined as follows.
    \begin{itemize}
        \item On objects, $\NDD$ takes a finite set $N$ to the set of functions $\{\upsilon \maps \R^N\to P\R^N\}$.
        \item Given a function $\phi\maps N\to M$, the mapping $\NDD(\phi)\colon \NDD(N)\to\NDD(M)$ is defined by
        \begin{equation}
            v\mapsto P\phi_*\circ v\circ \phi^*.
        \end{equation}
        \item Given finite sets $N$ and $M$, the product comparison \begin{equation*}\varphi_{N,M}\maps\NDD(N)\times\NDD(M)\to\NDD(N+M)\end{equation*} is defined by the function
        \begin{equation}
            (\upsilon,\rho)\mapsto P\iota_{N*}\circ\upsilon\circ\iota_N^* \boxplus P\iota_{M*}\circ\rho\circ\iota_M^*.
        \end{equation}
        \item The unit comparison $\varphi_0\maps \1\to \NDD(\emptyset)$ is the function $\R^0\to P\R^0$ defined by $0\mapsto \{0\}$.
    \end{itemize}
\end{lemma}
\begin{proof}
    Functoriality follows straightforwardly from the fact that $P$ and $(-)_*$ are covariant functors while $(-)^*$ is a contravariant functor. In particular, for preservation of identities we have
    \begin{equation}
        \NDD(\id_N)(v)\coloneqq P(\id_{N*})\circ v\circ \id_N^* = \id_{P\R^N}\circ v\circ \id_{\R^N} = v.
    \end{equation}
    Similarly, for preservation of composition, given $\phi\maps X\to Y$ and $\psi\maps Y\to Z$, we have
    \begin{align}
        \NDD(\psi\circ\phi)(v) &\coloneqq P((\psi\circ\phi)_*)\circ v\circ (\psi\circ\phi)^* \\
        &= P(\psi_*\circ\phi_*)\circ v\circ \phi^*\circ \psi^* \\
        &= P\psi_*\circ P\phi_*\circ v\circ \phi^*\circ \psi^* \\
        &= \NDD(\psi)\circ \NDD(\phi)(v).
    \end{align}

    For the verification of the naturality of the product comparison, we need to show that the diagram
\[\begin{tikzcd}
	{\NDD(X)\times\NDD(Y)} &&& {\NDD(N)\times\NDD(M)} \\
	{\NDD(X+Y)} &&& {\NDD(N+M)}
	\arrow["{\NDD(\phi)\times\NDD(\psi)}", from=1-1, to=1-4]
	\arrow["{\varphi_{X,Y}}"', from=1-1, to=2-1]
	\arrow["{\varphi_{N,M}}", from=1-4, to=2-4]
	\arrow["{\NDD(\phi+\psi)}"', from=2-1, to=2-4]
\end{tikzcd}\]
    commutes for all finite sets $X,Y,N,M$ and functions $\phi\maps X\to N$ and $\psi\maps Y\to M$.
    So, fix any systems $(\upsilon,\rho)\in\NDD(X)\times\NDD(Y)$. Following the top path yields the system
    \begin{equation}\label{eq:system1}
        P\iota_{N*}\circ P\phi_*\circ\upsilon\circ\phi^*\circ\iota_N^* \boxplus P\iota_{M*}\circ P\psi_*\circ \rho \circ \psi^*\circ \iota_M^* = P(\iota_{N*}\circ\phi_*)\circ\upsilon\circ\phi^*\circ\iota_N^* \boxplus P(\iota_{M*}\circ\psi_*)\circ\rho\circ\psi^*\circ\iota_M^*,
    \end{equation}
    where the equality comes from functoriality of $P$. Following the bottom path yields the system
    \begin{equation}\label{eq:system2}
    \begin{array}{ll@{}ll}
        P(\phi + \psi)_* \circ (P\iota_{X*}\circ\upsilon\circ\iota_X^* \boxplus P\iota_{Y*}\circ\rho\circ\iota_Y^*)\circ(\phi+\psi)^*\\= P(\phi_*\oplus\psi_*)\circ (P\iota_{X*}\circ\upsilon\circ\iota_X^* \boxplus P\iota_{Y*}\circ\rho\circ\iota_Y^*)\circ(\phi^*\oplus\psi^*) \\
        = (P(\phi_*\oplus\psi_*)\circ P\iota_{X*}\circ \upsilon\circ\iota_X^* \boxplus P(\phi_*\oplus\psi_*)\circ P\iota_{Y*}\circ \rho\circ\iota_Y^*)\circ(\phi^*\oplus\psi^*) \\
        = (P((\phi_*\oplus\psi_*)\circ\iota_{X*})\circ\upsilon\circ\iota_X^* \boxplus P((\phi_*\oplus\psi_*)\circ\iota_{Y*})\circ\rho\circ\iota_Y^*)\circ(\phi^*\oplus\psi^*),
    \end{array}
    \end{equation} 
    where the first equality comes from the fact that $(-)_*$ and $(-)^*$ are strong monoidal functors, the second equality comes from Lemma \ref{lem:mink_dist}, and the third equality comes from functoriality of $P$.
    Finally, the commutativity of the following diagrams
\[\begin{tikzcd}
	{\R^X} && {\R^N} && {\R^Y} && {\R^M} \\
	{\R^X\oplus\R^Y} && {\R^N\oplus\R^M} && {\R^X\oplus\R^Y} && {\R^N\oplus\R^M}
	\arrow["{\phi_*}", from=1-1, to=1-3]
	\arrow["{\iota_{X*}}"', from=1-1, to=2-1]
	\arrow["{\iota_{N*}}", from=1-3, to=2-3]
	\arrow["{\phi_*\oplus\psi_*}"', from=2-1, to=2-3]
	\arrow["{\psi_*}", from=1-5, to=1-7]
	\arrow["{\iota_{Y*}}"', from=1-5, to=2-5]
	\arrow["{\iota_{M*}}", from=1-7, to=2-7]
	\arrow["{\phi_*\oplus\psi_*}"', from=2-5, to=2-7]
\end{tikzcd}\]
    together with the commutativity of the diagrams in \eqref{eq:oplus_comm} shows that \eqref{eq:system1} is equivalent to \eqref{eq:system2}, as desired.

    The proofs of the associativity and unit diagrams are lengthy but straightforward calculations, with the only non-trivial step being another application of Lemma \ref{lem:mink_dist}.
\end{proof}

We can perform a similar extension to define an algebra for composing non-deterministic discrete systems. First note that there is a natural transformation $\eta\maps \id_\Set\Rightarrow P$ with components $\eta_X\maps X\to PX$ defined by
\begin{equation}
    x\mapsto \{x\}
\end{equation}
for all $x\in X$.
This will be useful to lift the identity functions in the definition of $\Dynam_D$ (Example \ref{ex:dynam}) with their non-deterministic counterparts.

\begin{lemma}
    There is a finset algebra $\NDD_D\maps(\FinSet,+)\to(\Set,\times)$ defined as follows.
    \begin{itemize}
        \item On objects, $\NDD_D$ takes a finite set $N$ to the set of functions $\R^N\to P\R^N$.
        \item Given a map $\phi\maps N\to M$, the map $\NDD_D(\phi)\maps \NDD_D(N)\to\NDD_D(M)$ is defined by the function
        \begin{equation}
            v\mapsto \eta_{\R^M}\boxplus P\phi_*\circ (v\boxminus\eta_{\R^N})\circ\phi^*.
        \end{equation}
        \item The comparison maps are the same as they are in $\NDD$.
    \end{itemize}
\end{lemma}
\begin{proof}
    To see that $\NDD_D$ preserves identities, let $X\in\FinSet$ and $v\maps \R^X\to P\R^X$ be arbitrary. Then
    \begin{align}
        \NDD_D(\id_X)(v) &\coloneqq
        \eta_{\R^X}\boxplus P(\id_{X*})\circ (v\boxminus\eta_{\R^X})\circ\id_X^* \nonumber \\
        &= \eta_{\R^X} \boxplus \id_{P\R^X}\circ (v \boxminus \eta_{\R^X})\circ\id_{\R^X} \\
        &= \eta_{\R^X}\boxplus v\boxminus\eta_{\R^X} \nonumber \\
        &=v, \nonumber
    \end{align}
    where the first equality comes from preservation of identities of the powerset, pushforward, and pullback functors and the second equality comes from identities being left and right units for composition.
    To see that $\NDD_D$ preserves composition, let $\phi\maps X\to Y$ and $\psi\maps Y\to Z$ be arbitrary morphisms in $\FinSet$. Then
    \begin{align}
        \NDD_D(\psi\circ\phi)(v) &\coloneqq \eta_{\R^Z}\boxplus P((\psi\circ\phi)_*)\circ(v\boxminus\eta_{\R^X})\circ(\psi\circ\phi)^* \nonumber \\ 
        &= \eta_{\R^Z}\boxplus P(\psi_*)\circ(P(\phi_*)\circ(v\boxminus\eta_{\R^X})\circ\phi^*)\circ\psi^* \\ 
        &= \eta_{\R^Z} \boxplus P(\psi_*)\circ(\NDD_D(\phi)(v) \boxminus \eta_{\R^Y})\circ\psi^* \nonumber \\
        &= (\NDD_D(\psi)\circ\NDD_D(\phi))(v), \nonumber
    \end{align}
    where the first equality holds by preservation of composition of the powerset and free vector space functors, and the rest hold by definition.

    Since the comparison maps are the same as in $\NDD$, the proofs are similar.
\end{proof}

Similar to deterministic systems, we can again use Euler's method on non-deterministic systems to map continuous to discrete.

\begin{lemma}
    Given $\gamma >0$, there is a monoidal natural transformation $\Euler^\gamma\maps \NDD\Rightarrow\NDD_D$ with components $\Euler^\gamma_N\maps \NDD(N)\to\NDD_D(N)$ given by the function 
    \begin{equation}
        \upsilon\mapsto \eta_{\R^N} \boxplus \gamma \upsilon.
    \end{equation}
\end{lemma}
\begin{proof}
    To verify naturality, we need to prove that the diagram
\[\begin{tikzcd}
	{\NDD(N)} && {\NDD(M)} \\
	\\
	{\NDD_D(N)} && {\NDD_D(M)}
	\arrow["{\NDD(\phi)}", from=1-1, to=1-3]
	\arrow["{\Euler^\gamma_N}"', from=1-1, to=3-1]
	\arrow["{\Euler^\gamma_M}", from=1-3, to=3-3]
	\arrow["{\NDD_D(\phi)}"', from=3-1, to=3-3]
\end{tikzcd}\]
    commutes for all $N,M\in \FinSet$ and $\phi\maps N\to M$. So, let $\upsilon\maps\R^N\to P\R^N$ be arbitrary. Following the top path yields the system
    \begin{equation}\label{eq:euler_top}
        \eta_{\R^M} \boxplus \gamma(P(\phi_*)\circ\upsilon\circ\phi^*),
    \end{equation}
    while following the bottom path yields the system
    \begin{equation}
        \eta_{\R^M} \boxplus P(\phi_*)\circ (\eta_{\R^N}\boxplus\gamma \upsilon \boxminus \eta_{\R^N})\circ\phi^* = \eta_{\R^M} \boxplus P(\phi_*)\circ\gamma\upsilon\circ\phi^*,
    \end{equation}
    which equals~\eqref{eq:euler_top}, as desired.
    To verify that the transformation is monoidal, we must verify that the diagrams
\[\begin{tikzcd}
	{\NDD(N)\times\NDD(M)} &&& {\NDD_D(N)\times\NDD_D(M)} \\
	\\
	{\NDD(N+M)} &&& {\NDD_D(N+M)}
	\arrow["{\varphi^{\NDD_D}_{N,M}}", from=1-4, to=3-4]
	\arrow["{\varphi^\NDD_{N,M}}"', from=1-1, to=3-1]
	\arrow["{\Euler^\gamma_N\times\Euler^\gamma_M}", from=1-1, to=1-4]
	\arrow["{\Euler^\gamma_{N+M}}"', from=3-1, to=3-4]
\end{tikzcd}\]
    and
\[\begin{tikzcd}
	{\1} && {\NDD(\emptyset)} \\
	&& {\NDD_D(\emptyset)}
	\arrow["{\Euler^\gamma_\emptyset}", from=1-3, to=2-3]
	\arrow["{\varphi_0^\NDD}", from=1-1, to=1-3]
	\arrow["{\varphi_0^{\NDD_D}}"', from=1-1, to=2-3]
\end{tikzcd}\]
    commute for all $N,M\in\FinSet$.
    For the product comparison diagram, let $(\upsilon,\rho)\in\NDD(N)\times\NDD(M)$ be arbitrary. Following the top path yields the system
    \begin{equation}
    \begin{array}{ll@{}ll}
        P(\iota_{N*})\circ(\eta_{\R^N}\boxplus\gamma\upsilon)\circ\iota_N^* \boxplus P(\iota_{M*})\circ(\eta_{\R^M}\boxplus\gamma\rho)\circ\iota_M^*=\\
        P(\iota_{N*})\circ\eta_{\R^N}\circ\iota_N^* \boxplus P(\iota_{M*})\circ\eta_{\R^M}\circ\iota_M^* \boxplus P(\iota_{N*})\circ\gamma\upsilon\circ\iota_N^* \boxplus P(\iota_{M*})\circ\gamma\rho\circ\iota_M^*,
    \end{array}
    \end{equation}
    where the equality comes from application of Lemma \ref{lem:mink_dist} and rearrangement of terms. Following the bottom path yields the system
    \begin{equation}
        \eta_{\R^{N+M}}\boxplus P(\iota_{N*})\circ\gamma\upsilon\circ\iota_N^* \boxplus P(\iota_{M*})\circ\gamma\rho\circ\iota_M^*.
    \end{equation}
    So, to complete this part of the proof, we just need to show that $\eta_{\R^{N+M}}=P(\iota_{N*})\circ\eta_{\R^N}\circ\iota_N^* \boxplus P(\iota_{M*})\circ\eta_{\R^M}\circ\iota_M^*$. For this, note that $\R^{N+M}\cong\R^N\times\R^M$ and let $(x,y)\in\R^N\times\R^M$ be arbitrary. Then
    \begin{equation}
    \begin{array}{ll@{}ll}
        (P(\iota_{N*})\circ\eta_{\R^N}\circ\iota_N^* \boxplus P(\iota_{M*})\circ\eta_{\R^M}\circ\iota_M^*)(x,y)=\\
        \{(x,0)\} \boxplus \{(0,y)\} = \{(x,y)\}=\eta_{\R^{N+M}}(x,y).
    \end{array}
    \end{equation}

    Finally, for the unit comparison diagram, let $z\maps\R^0\to P\R^0$ denote the zero function $0\mapsto \{0\}$. We have $(\eta_{\R^0}\boxplus\gamma z)(0) = \{0\} \boxplus \gamma\{0\} = \{0\}$, as desired.
\end{proof}

\subsection{Algebra Morphisms of Subgradient and Supergradient Methods}

With UWD algebras of non-deterministic dynamical systems defined, we can turn to proving that subgradient methods decompose non-differentiable convex optimization problems defined on UWDs. This will again allow us to extend the compositional data condition in a natural way to encompass non-differentiable convex problems.


\begin{theorem}
    There is a monoidal natural transformation $\subg\maps \Conv\Rightarrow\NDD$ with components
    \begin{equation}
        \subg_N\maps \Conv(N)\to\NDD(N),\
        f \mapsto -\partial f.
    \end{equation}
\end{theorem}
\begin{proof}
    We first need to check naturality, i.e., check that the diagram
\[\begin{tikzcd}
	{\Conv(N)} && {\Conv(M)} \\
	\\
	{\NDD(N)} && {\NDD(M)}
	\arrow["{\Conv(\phi)}", from=1-1, to=1-3]
	\arrow["{\subg_N}"', from=1-1, to=3-1]
	\arrow["{\subg_M}", from=1-3, to=3-3]
	\arrow["{\NDD(\phi)}"', from=3-1, to=3-3]
\end{tikzcd}\]
    commutes for all $N,M\in \FinSet$ and $\phi\maps N\to M$. So, let $f\in\Conv(N)$ be arbitrary and let $K=\mathcal{M}(\phi^*)$. Following the bottom path yields $PK^T\circ -\partial f\circ K$ while following the top path yields $-\partial(f\circ K)$. A fact about subgradients is that given a linear transformation $K$, the identity $\partial (f\circ K) = K^T\circ \partial f\circ K$ holds (\cite{boyd_subg_notes}, \S 3.3).
    In such literature, there is an implicit application of the powerset functor such that $K^T\circ\partial f\circ K$ is implied to mean $P(K^T)\circ\partial f\circ K$. Applying this identity to the result of following the top path shows that the two paths are equivalent.

    Similar to Theorem \ref{thm:flow}, the proof of monoidality follows the same reasoning as the proof of naturality, noting that the subgradient of a sum of functions is given by the Minkowski sum of the subgradients of the summands (\cite{boyd_subg_notes}, \S 3.2).
\end{proof}

A dual result holds for concave functions and supergradients.

\begin{corollary}
    There is a monoidal natural transformation $\superg\maps \Conc\Rightarrow\NDD$ with components $\superg_N\maps \Conc(N)\to\NDD(N)$ defined by
    \begin{equation}
        \superg_N(f)\coloneqq \partial f.
    \end{equation}
\end{corollary}

\begin{corollary}\label{thm:pdsubg}
    There is a monoidal natural transformation $\pdsubg\maps \Saddle\to\NDD\circ U$ with components $\pdsubg_{N\xrightarrow{\tau}[2]}\maps \Saddle(\tau)\to\NDD(N)$ defined by
    \begin{equation}
        \pdsubg_\tau(f)(x)_i\coloneqq \begin{cases}
            -\partial_{x_i} f(x) & \tau(i) = 1 \\
            \partial_{x_i} f(x) & \tau(i) = 2,
        \end{cases}
    \end{equation}
    for all $x\in\R^N$ and $i\in N$.
\end{corollary}
Note that there are also corollaries given by composing each of the above natural transformations with the non-deterministic Euler's transformation to derive discrete systems. We omit these for brevity.

With the naturality of these subgradient methods established, we can repeat our examples of primal decomposition (Example \ref{ex:primal_decomp}) and dual decomposition (Example \ref{ex:dual-decomp}) while relaxing the assumption of strict convexity of subproblems to mere convexity. This relaxation implies that the infima in these examples could be attained at multiple points. However, any such point $x^*\in \argmin f(w,x)$ for a fixed $w$ furnishes a subgradient for $\inf_x f(w,x)$ as $\partial f(w,x^*)$. Thus, subgradient methods give a natural extension of primal and dual decomposition to cases where subproblems are not strictly convex.

\section{Extended Example: Minimum Cost Network Flow}\label{sec:netflow}

We now illustrate how our framework gives a principled way to convert problem data into solution algorithms that respect the compositional structure of the data. Specifically, we construct a finset algebra for composing flow networks by gluing together vertices and prove that the translation of a given flow network into the associated minimum cost network flow problem is a monoidal natural transformation. This satisfies the compositional data condition (Definition \ref{def:cdc}), and thus lets us derive a distributed algorithm for solving the 
minimum cost network flow (MCNF) problem, which is a widely studied class of problem in engineering applications~\cite{bazaraa11,ford15,jensen87,kennington80,evans17}. 
We will see that the result is a generalization of a standard dual decomposition algorithm for MCNF which respects arbitrary hierarchical decompositions of flow networks.

\subsection{Standard Formulation of MCNF and Dual Decomposition}

First we review the standard non-compositional formulation of the MCNF problem. A flow network is a directed graph with designated source and sink vertices (each with a respective inflow and outflow) and a cost function associated to each edge that indicates the cost to push a given flow through that edge. The minimum cost network flow problem then involves finding an allocation of flows to the edges of a flow network that minimizes the total flow cost while satisfying the constraint that the amount of flow into a vertex must equal the amount of flow out of that vertex, for all vertices. This is known as the flow conservation constraint. Applications of network flow problems range from determining if sports teams have been eliminated from championship contention to computing how electricity flows in networks of resistors \cite{netflowbook1, netflowbook2}.

Formally, a flow network on a finite set $V$ of vertices is a 5-tuple $(E,s,t,\ell,b)$ where
\begin{itemize}
    \item $E$ is a finite set of edges,
    \item $s\maps E\to V$ maps each edge to its source vertex,
    \item $t\maps E\to V$ maps each edge to its target vertex,
    \item $\ell \maps E\to K$ where $K\coloneqq\{\kappa \maps\R\to\R\mid \kappa\in C^1 \text{ and convex}\}$ defines the flow cost function for each edge,
        \item $b\maps V\to \R$ defines the incoming or outgoing flow for each vertex: a positive value for $v\in V$ represents an inflow for $v$, a negative value represents an outflow for $v$ and a value of 0 represents no flow. The total inflow of the graph needs to equal to the total outflow of the graph implying that $\sum_{v\in V}b(v)$ must equal 0. Treating $b$ as a vector in $\R^V$, we can write this condition compactly as $\boldsymbol{1}^Tb = 0$. 
\end{itemize}
All this data can be packaged up in the following diagram:
    \begin{equation}
        K\xleftarrow{\ell} E\overset{s}{\underset{t}{\rightrightarrows}} V\xrightarrow{b}\R.
    \end{equation}
    Note that there is a unique flow network on 0 vertices given by:
\begin{equation*}
    K\xleftarrow{!}\emptyset\overset{!}{\underset{!}{\rightrightarrows}}\emptyset \xrightarrow{!}\R.
\end{equation*}
Thus the set of flow networks on 0 vertices is a singleton set and thus a terminal object in $\Set$.

Suppose we have a flow network $G=(E,s,t,\ell,b)$ on $V$ vertices. We will use the variable $x\in\R^E$ to track flows on each edge. A value of $x_e>0$ indicates flow in the direction of edge $e$ while a value of $x_e<0$ indicates flow in the opposite direction of edge $e$. Thus, the objective of the minimum cost flow problem is to minimize the function
\begin{equation}
    \sum_{e\in E}\ell(e)(x_e).
\end{equation}

We also must enforce the flow conservation constraint. To encode this mathematically, let $A\in\R^{V\times E}$ denote the node incidence matrix\footnote{The first condition in the definition is required to quotient out self loops.} of $G$, i.e.,
\begin{equation}
        A_{v,e} \coloneqq \begin{cases}
            0 & \text{ if } s(e)= t(e) = v \\
            1 & \text{ if } s(e)=v \\
            -1 & \text{ if }  t(e)=v \\
            0 & \text{ otherwise.}
        \end{cases}
\end{equation}
Then flow conservation can be compactly stated as requiring that $Ax=b$.
Combining the objective function and constraints, the minimum cost network flow problem for the flow network $G$ is
\begin{equation}
    \begin{array}{ll@{}ll}
    \text{minimize} & \sum_{e\in E}\ell(e)(x_e) & \\
    \text{subject to} & Ax=b. \\
    \end{array}
\end{equation}

A common first-order solution procedure is to form the Lagrangian and use dual decomposition. The Lagrangian is
\begin{equation}
    L(x,\lambda) \coloneqq \sum_{e\in E}\ell(e)(x_e) + \lambda^T(Ax-b).
\end{equation}
Then, the dual function $q_G\maps\R^V\to\R$ is the concave function given by 
\begin{equation}
    q_G(\lambda)\coloneqq\inf_{x\in\R^E}L(x,\lambda)= \inf_{x\in\R^E}\sum_{e\in E}\ell(e)(x_e) + \lambda^T(Ax-b).
\end{equation}
Applying supergradient ascent to $q_G$ gives the iterative algorithm
\begin{equation}
\begin{array}{ll@{}ll}
    x^*&\coloneqq \argmin L(x,\lambda_k)\\
    \lambda_{k+1}&\coloneqq \lambda_k + \gamma (Ax^* - b).
    \end{array}
\end{equation}
This works because $Ax^*-b$ is a supergradient of $q_G$ at $\lambda_k$. Since $L$ is separable in $x$, the computation of $x^*$ decomposes into $|E|$ scalar minimizations.

\subsection{Compositional Formulation of MCNF}

To generalize this dual decomposition algorithm to arbitrary hierarchical decompositions of flow networks, we must first define a finset algebra for composing flow networks.

\begin{lemma}
    There is a finset algebra\footnote{This is technically only a lax symmetric monoidal \emph{pseudofunctor} because some of the required diagrams only hold up to canonical isomorphism. See \cite{baez_structured_2022} for details on this technicality.} $\FG\maps (\FinSet,+)\to(\Set,\times)$ defined by the following maps:
    \begin{itemize}
        \item On objects, $\FG$ takes a finite set $N$ to the set of flow networks with vertex set $N$:
        \begin{equation}
        \{K\xleftarrow{\ell} E\overset{s}{\underset{t}{\rightrightarrows}} N\xrightarrow{b}\R\mid \textbf{1}^Tb=0\}.
    \end{equation}
        \item Given a morphism $\phi\maps N\to M$ in $\FinSet$, the map $\FG(\phi)\maps\FG(N)\to\FG(M)$ is defined by the function
        \begin{equation}
            (E,s,t,\ell,b)\mapsto (E,\phi\circ s,\phi\circ t,\ell, \phi_*(b)).
        \end{equation}
        \item The unit comparison $\varphi_0\maps \1\to\FG(\emptyset)$ is the unique arrow since $\FG(\emptyset)$ is a terminal object in $\Set$.
        \item Given objects $N,M\in\FinSet$, the product comparison
        \begin{equation}
            \varphi_{N,M}\maps \FG(N)\times\FG(M)\to\FG(N+M)
        \end{equation}
        takes a pair of flow networks $(E_1,s_1,t_1,\ell_1,b_1)$ and $(E_2,s_2,t_2,\ell_2,b_2)$ to their disjoint union, i.e., $(E_1+E_2,s_1+s_2,t_1+t_2,[\ell_1,\ell_2],[b_1,b_2])$.
    \end{itemize}
\end{lemma}
\begin{proof}
    $\FG$ is an extension of the labelled graph functor introduced in \cite{baez_structured_2022} with the addition of a scalar on each vertex defined by the $b$ vector, so most of their proof carries over. We only need to check that the morphism map and product comparison preserve the property that 
    $\boldsymbol{1}^Tb=0$. The pairing $[b_1,b_2]$ in the product comparison preserves this property because both $\boldsymbol{1}^T b_1 = 0$ and $\boldsymbol{1}^T b_2 = 0$. For the morphism map, note that for vectors $v = \phi_*(b)$ in the image of the pushforward map, a component $v_i$ is either 0, a single element of $b$, or the sum of one or more elements of $b$. Since the components of $b$ sum to $0$ and every component of $b$ gets mapped under $\phi_*$ either to itself or a sum with other components of $b$, the components of $v$ in the image of $\phi$ must sum to zero. Furthermore, since the components of $v$ not in the image of $\phi$ are 0, the sum of all components of $v$ is 0. 
\end{proof}
Given a UWD and subgraphs with the appropriate number of nodes for each inner box, the $\FG$ algebra constructs a composite graph by merging vertices which share the same junctions. This is illustrated in Figure \ref{fig:flow-graphs}. 

\begin{figure}
    \centering
    \includegraphics[width=\textwidth]{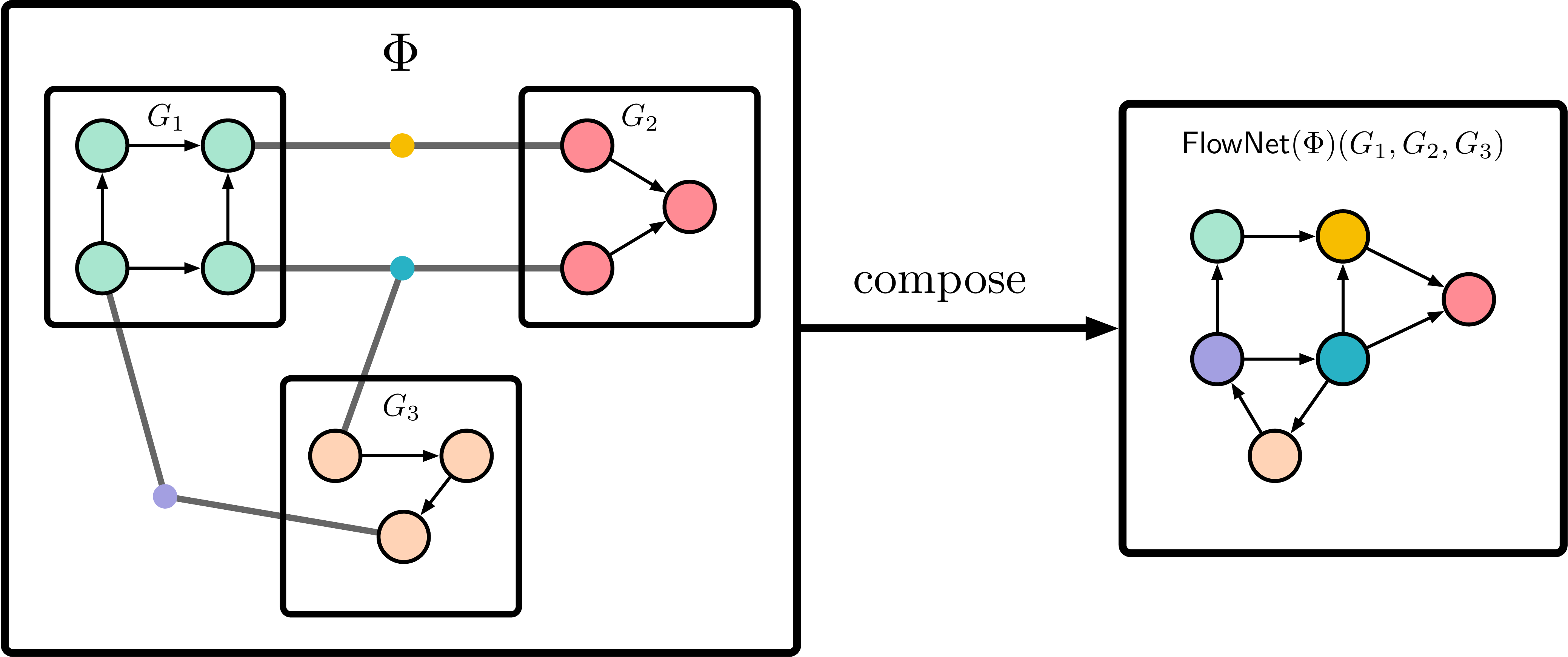}
    \caption{An illustration of composing flow networks using the $\FG$ UWD-algebra. $\FG$ is applied to the UWD $\Phi$ and the component subgraphs $G_1,G_2,$ and $G_3$ to yield the resultant flow network on the right. Here, we only illustrate the action of $\FG$ on the underlying directed graph structure of flow networks. Colors were added to emphasize the distinction between vertices which are merged and vertices which are not merged by composition.}
    \label{fig:flow-graphs}
\end{figure}

Now, our goal is to show that the translation of flow networks into their associated MCNF objective function constitutes a monoidal natural transformation and thus satisfies the compositional data condition.
First note that the translation of a flow network $G$ on $V$ vertices into the dual function $q_G$ of the associated MCNF problem gives us a function $\netflow_V\maps\FG(V)\to\Conc(V)$ defined by $G\mapsto q_G$. Since this translation works for any $V$, we have a \emph{family} of maps indexed by finite sets which will form the components of our natural transformation. 
\begin{theorem}
    There is a monoidal natural transformation $\netflow\maps \FG\Rightarrow \Conc$ whose components $\netflow_N\maps \FG(N)\to\Conc(N)$ take a flow network $G\coloneqq (E,s,t,\ell,b)$ on $N$ vertices to the concave function
    \begin{equation}
        \lambda\mapsto \inf_{x\in\R^E}\sum_{e\in E}\ell(e)(x_e) + \lambda^T(Ax-b),
    \end{equation}
    where $A$ is the node incidence matrix of $G$.
\end{theorem}
\begin{proof}
    To verify naturality, we need to show that the diagram
\[\begin{tikzcd}
	{\FG(N)} && {\FG(M)} \\
	\\
	{\Conc(N)} && {\Conc(M)}
	\arrow["{\FG(\phi)}", from=1-1, to=1-3]
	\arrow["{\Conc(\phi)}"', from=3-1, to=3-3]
	\arrow["{\netflow_N}"{description}, from=1-1, to=3-1]
	\arrow["{\netflow_M}"{description}, from=1-3, to=3-3]
\end{tikzcd}\]
    commutes for all $N,M\in\FinSet$ and $\phi\maps N\to M$. Let $(E,s,t,\ell,b)\in\FG(N)$ be arbitrary, and let $K=\mathcal{M}(\phi^*)$. Because the pushforward is dual to the pullback, we know that $K^T=\mathcal{M}(\phi_*)$.  Now, following the top path results in the function
    \begin{equation}
        \lambda\mapsto \inf_{x\in\R^E}\sum_{e\in E}\ell(e)(x_e)+\lambda^T(Ax-K^Tb),
    \end{equation}
    where $A\in\R^{M\times E}$ is defined by
    \begin{equation}
        A_{e,m}\coloneqq \begin{cases}
            0 & \text{ if } \phi(s(e))= \phi(t(e)) = m \\
            1 & \text{ if } \phi(s(e))=m \\
            -1 & \text{ if }  \phi(t(e))=m \\
            0 & \text{ otherwise.}
        \end{cases}
    \end{equation}
    Similarly, following the bottom path yields the function
    \begin{equation}
        \lambda\mapsto \inf_{x\in\R^E}\sum_{e\in E}\ell(e)(x_e)+(K\lambda)^T(Cx - b) = \inf_{x\in\R^E}\sum_{e\in E}\ell(e)(x_e)+\lambda^T(K^TCx - K^Tb),
    \end{equation}
    where $C\in\R^{N\times E}$ is defined by
    \begin{equation}
        C_{e,n}\coloneqq \begin{cases}
            0 & \text{ if } s(e)= t(e) = n \\
            1 & \text{ if } s(e)=n \\
            -1 & \text{ if }  t(e)=n \\
            0 & \text{ otherwise.}
        \end{cases}
    \end{equation}
    Thus, to complete the proof, it suffices to show that $A = K^TC$. To see why this is the case, first note that $A$ and $C$ both have $|E|$ columns. So consider an arbitrary $e\in E$. We just need to show that $a_e = K^Tc_e \coloneqq \phi_*(c_e)$, where $a_e$ and $c_e$ are the $e^{th}$ columns of $A$ and $C$, respectively. Letting $m\in M$ be arbitrary and using the definition of pushforward, we know that
    \begin{equation}\label{eq:c_e}
        \phi_*(c_e)(m) \coloneqq \sum_{n\in\phi^{-1}(m)}c_e(n) = \sum_{n\in\phi^{-1}(m)}\begin{cases}
            0 & \text{ if } s(e)= t(e) = n \\
            1 & \text{ if } s(e)=n \\
            -1 & \text{ if }  t(e)=n \\
            0 & \text{ otherwise.}
        \end{cases}
    \end{equation}
    We now consider the three possibilities for the value of $a_e(m)$ and show that in every case they must equal the value of $\phi_*(c_e)(m)$.

    First, suppose that $a_e(m) = 1$. Then by definition of $A$, we know that $\phi(s(e))=m$ and $\phi(t(e))\neq m$. These equations imply that $s(e)\in\phi^{-1}(m)$ and $t(e)\notin\phi^{-1}(m)$, and thus \eqref{eq:c_e} says that $\phi_*(c_e)(m)$ must equal $1$. A similar argument shows that if $a_e(m) = -1$, then $\phi_*(c_e)(m) = -1$.

    Finally, we must consider the case when $a_e(m) = 1$. There are two possibilities: either $s(e)\notin\phi^{-1}(m)$ \emph{and} $t(e)\notin\phi^{-1}(m)$, in which case \eqref{eq:c_e} says that $\phi_*(c_e)(m) = 0$, or $s(e)\in\phi^{-1}(m)$ \emph{and} $t(e)\in\phi^{-1}(m)$. In the latter case, by \eqref{eq:c_e}, $\phi_*(c_e)(m)\coloneqq 1 - 1 = 0$, as desired.

    To verify that the transformation is monoidal, we need to verify that the diagrams
\[\begin{tikzcd}
	{\FG(N)\times\FG(M)} && {} & {\Conc(N)\times\Conc(M)} \\
	{\FG(N+M)} &&& {\Conc(N+M)}
	\arrow["{\varphi_{N,M}^\FG}"', from=1-1, to=2-1]
	\arrow["{\netflow_N\times\netflow_M}", from=1-1, to=1-4]
	\arrow["{\varphi_{N,M}^\Conc}", from=1-4, to=2-4]
	\arrow["{\netflow_{N+M}}"', from=2-1, to=2-4]
\end{tikzcd}\]
    and
\[\begin{tikzcd}
	{\1} && {\FG(\emptyset)} \\
	&& {\Conc(\emptyset)}
	\arrow["{\netflow_\emptyset}", from=1-3, to=2-3]
	\arrow["{\varphi_0^\FG}", from=1-1, to=1-3]
	\arrow["{\varphi_0^\Conc}"', from=1-1, to=2-3]
\end{tikzcd}\]
    commute for all $N,M\in\FinSet$.

    For the first diagram, let $G_1\coloneqq (E_1,s_1,t_1,\ell_1,b_1)\in\FG(N)$ and $G_2\coloneqq (E_2,s_2,t_2,\ell_2,b_2)\in\FG(M)$ be arbitrary, and let $A_{G_1}$ and $A_{G_2}$ denote the node incidence matrices of $G_1$ and $G_2$, respectively. Then, following the top path results in the function
    \begin{equation}\label{eq:nf_mon_1}
        \lambda\mapsto \inf_{x\in\R^{E_1}}\sum_{e\in E_1}\ell_1(e)(x_e) + \pi_N(\lambda)^T(A_{G_1}x-b_1) + \inf_{x'\in\R^{E_2}}\sum_{e'\in E_2}\ell_2(e')(x'_{e'}) + \pi_M(\lambda)^T(A_{G_2}x'-b_2).
    \end{equation}
    Likewise, following the bottom path yields the function 
    \begin{equation}\label{eq:nf_mon_2}
        \lambda\mapsto \inf_{y\in\R^{E_1+E_2}}\sum_{e\in E_1+E_2}[\ell_1,\ell_2](e)(y_e) + \lambda^T\left(\begin{bmatrix}
            A_{G_1} & 0 \\ 0 & A_{G_2}
        \end{bmatrix}y - \begin{bmatrix}
            b_1 \\ b_2
        \end{bmatrix}\right).
    \end{equation}
    Noting that \eqref{eq:nf_mon_2} is separable along the components of $y$ in $\R^{E_1}$ and the components of~$y$ in $\R^{E_2}$ shows that it is equivalent to \eqref{eq:nf_mon_1}.

    For the second diagram, note that the node incidence matrix of an empty flow network is necessarily the unique linear map $\R^0\to\R^0$. Similarly, for an empty flow network $b$ must be $0\in\R^0$. With this case in mind, following the top and bottom paths both yield the constant zero function, as desired.
\end{proof}

    With this algebra morphism established, we can finally turn to the composite given by the following diagram:
\begin{equation}\label{eq:big_daddy}\begin{tikzcd}
	\FG && \Conc && \NDD && {\NDD_D}
	\arrow["{\netflow}", Rightarrow, from=1-1, to=1-3]
	\arrow["{\superg}", Rightarrow, from=1-3, to=1-5]
	\arrow["{\Euler^\gamma}", Rightarrow, from=1-5, to=1-7]
\end{tikzcd}\end{equation}
    This transformation takes a flow network $G$ to the discrete non-deterministic dynamical system that solves the minimum cost network flow problem on $G$ using supergradient ascent on the dual problem $q_G$. The results proven throughout this paper show that this composite respects any decomposition of $G$. In particular, if we completely decompose $G$ with a UWD where each inner box contains only a single edge, then applying the composite in \eqref{eq:big_daddy} yields the standard dual decomposition algorithm as in Section 2.A of \cite{jadbabaie_distributed_2009}.

    This decomposition algorithm is flat, in that it decomposes into a single master problem and scalar subproblems for every eadge. More interestingly, we can apply \eqref{eq:big_daddy} to other decomposition structures of $G$. For example, we could decompose $G$ by densely connected components. Each connected component can then be fully decomposed by edges to yield a hierarchical decomposition. The subsequent section illustrates how exploiting such higher-level compositional structure can be beneficial.

\section{Implementation and Numerical Results}\label{sec:impl}

This section discusses our prototype implementation of this framework in a Julia package called AlgebraicOptimization.jl\footnote{\url{https://github.com/AlgebraicJulia/AlgebraicOptimization.jl}}. We used this package to perform numerical experiments demonstrating that the exploitation of higher-level compositional structure of flow networks can result in a faster solver than standard dual decomposition for MCNF.

\subsection{The AlgebraicOptimization Software Architecture}

The framework developed in this paper consists of a collection of gradient-based algebra morphisms
\begin{equation}
    g\maps O\Rightarrow D,
\end{equation}
where $O$ is one of the optimization problem algebras and $D$ is one of the dynamical system algebras (visualized in Figure \ref{fig:hierarchy}). These algebra morphisms take optimization problems as inputs and produce iterative solution algorithms as outputs. Thus for each $g\maps O\Rightarrow D$, we require the following:
\begin{itemize}
    \item an implementation of $O$ as a UWD-algebra,
    \item an implementation of $D$ as a UWD-algebra,
    \item an implementation of the components of $g$.
\end{itemize}

The Catlab.jl software package for categorical computing \cite{halter2020compositional} provides a generic interface for specifying UWD-algebras, while AlgebraicDynamics.jl implements the $\Dynam$ and $\Dynam_D$ UWD-algebras as well as the $\Euler^\gamma$ algebra morphism \cite{libkind_operadic_2022}. So, for example, to specify $\gd^\gamma\maps \Opt\Rightarrow \Dynam_D$, we just needed to specify $\Opt$ as a UWD-algebra using Catlab and the components of $\flow\maps\Opt\Rightarrow\Dynam$, which can then be composed with $\Euler^\gamma$ to recover $\gd^\gamma$. The gradient flow dynamical system for a given objective function is straightforward to implement thanks to Julia's extensive support for automatic differentiation with packages such as ForwardDiff.jl \cite{RevelsLubinPapamarkou2016}. 

The ease with which we can implement this framework illustrates one of the benefits of formalizing these optimization problems with category theory: with a package such as Catlab for categorical computing, the translation of mathematical definitions to usable software is straightforward. Moreover, thanks to the intuitive graphical synatx of UWDs, an application-oriented user of AlgebraicOptimization need not be concerned with the underlying category theory
to benefit from these results. They can instead simply specify a UWD and subproblems to fill each inner box, and AlgebraicOptimization will automatically generate a 
decomposed solution method for their problem. The user can also specify whether the solution method is serial (distribute to subsystems, update subsystems sequentially, then collect from subsystems), or parallel (distribute to subsystems, update subsystems in parallel, then collect from subsystems). 





\subsection{Compositional Network Flow Experiments}
\begin{figure}[htbp]
    \centering
    \includegraphics[width=\textwidth]{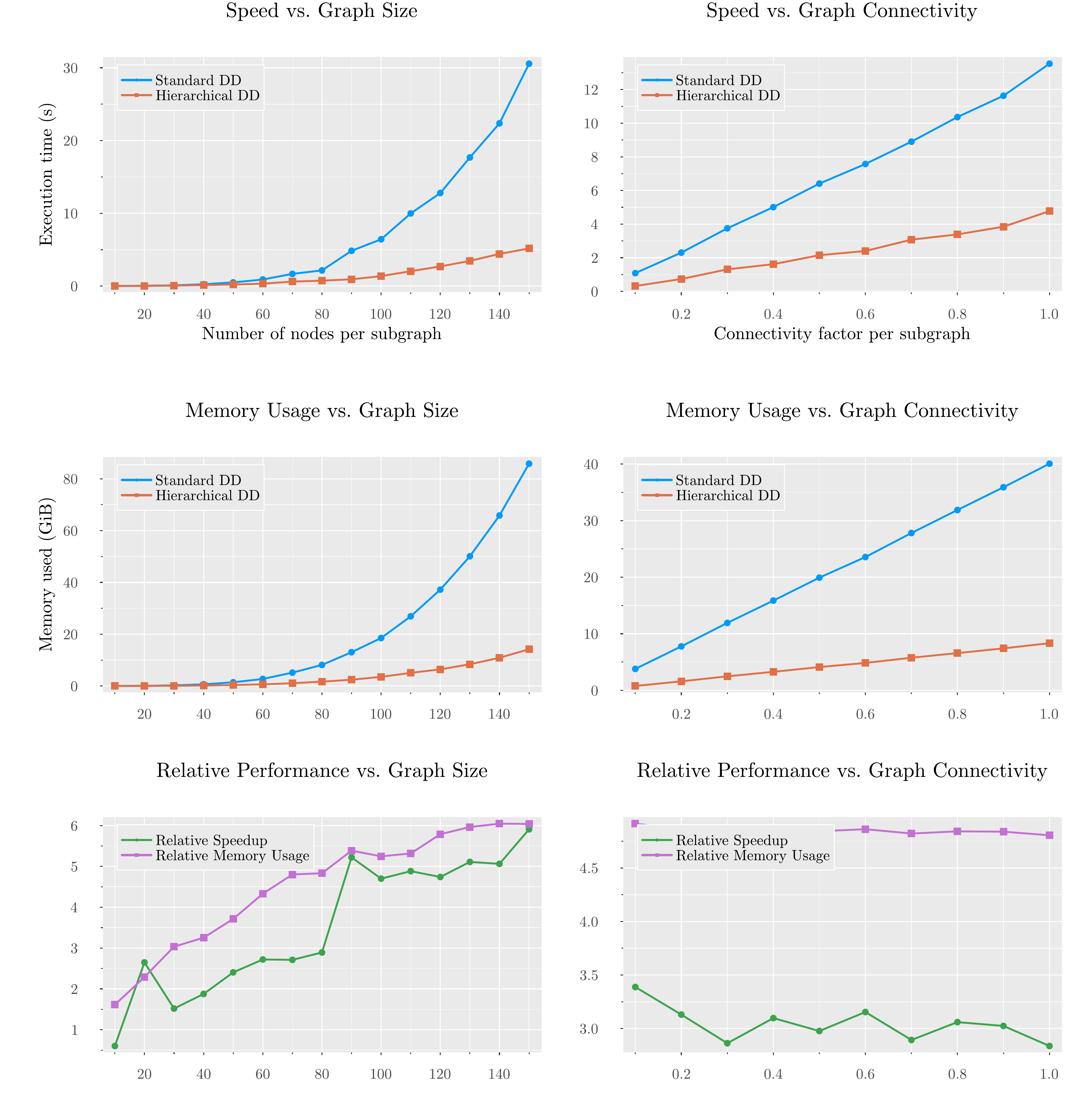}
    \caption{Benchmark results of standard dual decomposition and hierarchical dual decomposition on random composite flow networks. The left column reports results for varying the number of nodes in each components subgraph from 10 to 150 while holding connectivity fixed at 0.2. The right column reports results for varying the graph connectivity parameter of each subgraph from 0.1 to 1.0 while holding the number of nodes per subgraph fixed at 80. The top row compares execution times, the second row compares total memory used, and the bottom row charts the ratios between standard and hierarchical dual decomposition for each of these quantities. 
    }
    \label{fig:benchmarks}
\end{figure}

Intuitively, we expect that compositional modeling of MCNF problems will be practically useful because many real world applications of network flow involve flow networks with several densely connected components and sparse connectivity among these components. For example, a traffic flow network may model several densely connected urban areas coupled by relatively few interstates and highways. This is an example of hierarchical compositional structure where the first level decomposes each urban area from the others and the second level decomposes the edges within each urban area.

Motivated by such applications, we conducted a brief experimental study to measure the performance of standard dual decomposition against the hierarchical dual decomposition enabled by our framework when solving the minimum cost network flow problem. We worked with flow networks whose compositional structure was given by the UWD $\Phi\maps [3]+[2]+[2]\to\emptyset$ in Figure \ref{fig:flow-graphs}. The graphs were randomly generated to avoid implicitly favoring any particular solution method. More specifically, to generate test data for this algorithm, we built three component flow networks $(G_1,G_2,G_3)$ by starting with three connected Erd\H{o}s-R\'enyi random graphs \cite{Erdos1959OnRG}, assigning to each edge a random convex quadratic cost function, and to the vertices a random vector of inflows and outflows which summed to 0. We then randomly designated 3 boundary vertices for $G_1$ and 2 boundary vertices for each of $G_2$ and $G_3$. Thus constructed, $G_1$ fills the box of $\Phi$ with three ports, and $G_2$ and $G_3$ fill the boxes with two ports. Given $(G_1, G_2, G_3)$, we examined two ways to proceed towards solving the corresponding MCNF problem:

\begin{enumerate}
    \item (Standard Dual Decomposition) This algorithm is given by first composing component subgraphs with $\FG$ and then taking the dual decomposition optimizer of the composite graph. 
    \item (Hierarchical Dual Decomposition) This algorithm is given by taking the dual decomposition optimizer of each component subgraph and composing these systems with $\Dynam_D$.
\end{enumerate}

Standard dual decomposition corresponds to following the top path of the following diagram while hierarchical dual decomposition corresponds to following the bottom path.
\[\begin{tikzcd}
	{\FG([3])\times\FG([2])\times\FG([2])} &&& {\FG(\emptyset)} \\
	{\DiffConc([3])\times\DiffConc([2])\times\DiffConc([2])} &&& {\DiffConc(\emptyset)} \\
	{\Dynam([3])\times\Dynam([2])\times\Dynam([2])} &&& {\Dynam(\emptyset)} \\
	{\Dynam_D([3])\times\Dynam_D([2])\times\Dynam_D([2])} &&& {\Dynam_D(\emptyset)}
	\arrow["{\FG(\Phi)}", from=1-1, to=1-4]
	\arrow["{\netflow_\emptyset}", from=1-4, to=2-4]
	\arrow["\flow_\emptyset", from=2-4, to=3-4]
	\arrow["{\Euler^\gamma_\emptyset}", from=3-4, to=4-4]
	\arrow["{\netflow_{[3]}\times\netflow_{[2]}\times\netflow_{[2]}}"', from=1-1, to=2-1]
	\arrow["{\flow_{[3]}\times\flow_{[2]}\times\flow_{[2]}}"', from=2-1, to=3-1]
	\arrow["{\Euler^\gamma_{[3]}\times\Euler^\gamma_{[2]}\times\Euler^\gamma_{[2]}}"', from=3-1, to=4-1]
	\arrow["{\Dynam_D(\Phi)}"', from=4-1, to=4-4]
\end{tikzcd}\]
Note that for all experiments performed, we used strictly convex cost functions for each edge, implying that the dual function of a flow network generated by the $\netflow$ transformation is differentiable. We were therefore able to use standard gradient flow instead of supergradient flow. 
Also recall that the algebras in this diagram represent the UWD-algebras of the corresponding finset algebras defined in this paper. So for example, $\FG(\emptyset)$ represents the set of all open flow networks with empty boundary, rather than the empty flow network.
Since the above diagram is the naturality square of a composite of natural transformations, the dynamical systems generated by following the top and bottom paths are equivalent in that they follow the same trajectory.

Given these two methods of building solvers from problem data, a fair comparison between them must also account for relevant properties of the random graphs. In particular, one might expect that the size, i.e. number of nodes and edges, in the component graphs would affect execution time and memory usage of each algorithm. As such we sought to compare both standard and hierarchical dual decomposition using graphs of various sizes. More specifically, we conducted two experiments to compare standard dual decomposition and hierarchical dual decomposition. First note that an Erd\H{o}s-R\'enyi random graph $G(N,p)$ is parameterized by two values: $N\geq 0$ is the number of vertices in the graph and $0\leq p\leq1$ represents the probability that a given edge is included in the graph. Therefore, increasing $p$ increases the connectivity of $G(N,p)$. For the first experiment, we held $p$ at the fixed value of 0.2 and varied the number of nodes for each component subgraph. Conversely, for the second experiment, we held the number of nodes of each component subgraph fixed at 80 and varied $p$. For each selection of parameters, we measured the elapsed time and total memory consumption to run 10 outer iterations of dual decomposition. We used the Optim.jl package \cite{mogensen2018optim} to solve the inner problems of computing optimal flows for edges given fixed values of the dual variables.

The results of these experiments are summarized in Figure \ref{fig:benchmarks}. These results clearly demonstrate the benefits of hierarchical dual decomposition over standard dual decomposition when hierarchical structure exists. Note that we do not compare against state of the art MCNF algorithms, but rather aim to show some practical utility of hierarchical structure.  In particular, the advantages of hierarchical decomposition are more pronounced on dense graphs with many nodes. This is expected because although both algorithms must solve the same number of scalar minimization problems (one for each edge), these minimizations can be computed more efficiently in the hierarchical case because they are only parameterized by the dual variables of the subgraph they belong to, rather than all the dual variables in the case of standard dual decomposition.


\section{Conclusion}\label{sec:conclusion}

\subsection*{Summary}
We have presented a novel algebraic/structural perspective on first-order optimization decomposition algorithms using the abstractions of undirected wiring diagrams (UWDs), operad algebras, and algebra morphisms. 
Specifically, we demonstrated that whenever a first-order method defines an algebra morphism from a UWD-algebra of problems to a UWD-algebra of dynamical systems, that method decomposes problems defined on arbitrary UWDs and provides distributed message-passing solution algorithms. We leveraged this paradigm to 
unify and generalize many known algorithms within a common framework, including gradient descent (Theorem \ref{thm:flow}), Uzawa's algorithm (Theorem \ref{thm:uzawa}), primal and dual decomposition (Examples \ref{ex:primal_decomp} and \ref{ex:dual-decomp}), and subgradient variants of each (Corollary \ref{thm:pdsubg}). Working within this framework allowed the statement of the compositional data condition (Definition \ref{def:cdc}), which is a novel sufficient condition for when optimization problems parameterized by compositional data are decomposable. We subsequently showed that the minimum cost network flow problem satisfied the compositional data condition by defining an algebra morphism from a UWD-algebra of flow networks to the UWD-algebra of concave maximization problems. Experimental results using our implementation of this compositional framework highlighted the computational benefits afforded by the ability to exploit hierarchical and compositional problem structure. 

\subsection*{Impacts}

We believe this framework will have many positive impacts for both optimization theory and practice. Most immediately, our implementation of this framework provides a convenient tool for engineers to compositionally build complex optimization problems in a hierarchical fashion using the intuitive graphical syntax of undirected wiring diagrams. Such specifications are then automatically transformed into performant distributed solution algorithms by leveraging the algebra morphisms described in this paper. We also hope that optimization researchers can utilize this general unified language for optimization decomposition methods to more easily build novel algorithms for specialized applications (e.g. by finding an application for a hierarchical combination of both primal and dual decomposition). Furthermore, the compositional data condition provides a principled way of proving when problems defined by compositional data are decomposable and generating appropriate distributed solution algorithms. We hope that this condition can be used to derive hierarchical decomposition methods for many more applications beyond the network flow example presented in the paper.

In terms of theoretical impacts, this work has situated optimization decomposition methods in the broad mathematical landscape of category theory. Seen in this light, the underlying mathematical structure of optimization decomposition methods is the same as or similar to the structures underlying entropy maximization of thermostatic systems \cite{Baez_thermo_2023} and equilibrium behavior of Markov processes \cite{Baez_2016_markov} in that both also tell stories of operadic dynamical systems. Also of interest is that backpropagation in neural networks has been shown to exhibit semantics for \emph{directed} wiring diagrams \cite{fong2019backprop}. In making such connections we hope to facilitate the transfer of ideas amongst optimization and other disciplines. Additionally, we believe that the category theoretic point of view will lead to the expansion of problems considered by optimization researchers. Indeed constructs already exist that allow one to generalize gradient based optimization to settings beyond differentiable functions over the real numbers \cite{cruttwell2022categorical, shiebler_generalized_2022}, for example to polynomials over arbitrary finite fields or boolean circuits. Our work has the potential to add composition mechanisms to instances of such non-traditional optimization problems.

\subsection*{Limitations and Future Work}

There are many directions for future work. The first natural question is whether other, more sophisticated decomposition methods such as the alternating direction method of multipliers (ADMM) or second-order methods such as Newton's method can be expressed within our framework. We are hopeful that an affirmative answer can be given for both of these methods. For example, \cite{hallac_snapvx_nodate} defines a way to apply ADMM to decompose problems defined on factor graphs, which are similar to UWDs, and \cite{jadbabaie_distributed_2009} defines a distributed Newton's method for solving minimum cost network flow.

There are also some fundamental limitations of our approach which we aim to address in future work. For example, Euler's method (and by extension gradient descent) is only an algebra morphism for a fixed step-size. However, many optimization algorithms use more sophisticated methods such as backtracking line search for computing appropriate step-sizes at each iteration. As such, we would like to develop a theory of \emph{approximate} algebra morphisms and see if we can bound the approximation error based on the compositional structure of a given problem.

Another limitation of our framework is that the distributed solution algorithms generated by applying our algebra morphisms are \emph{synchronous}, meaning they require a blocking synchronization step after every parallel computation step. In reality, many algorithms such as distributed gradient descent and Uzawa's can be run asynchronously while still guaranteeing convergence \cite{hendrickson2022totally}. We thus would like to extend the existing algebraic theory of composite dynamical systems to encompass such asynchrony.

\subsection*{Acknowledgements}

Tyler Hanks was supported by the National Science Foundation Graduate
Research Fellowship Program under Grant No. DGE-1842473. Any opinions, findings,
and conclusions or recommendations expressed in this material are those of the author(s)
and do not necessarily reflect the views of the NSF. 
Matthew Hale was supported by the Air Force Office of
Scientific Research (AFOSR) under Grant No. FA9550-23-1-0120. 
 Matthew Klawonn was funded by a SMART SEED grant. James Fairbanks was supported by DARPA under award no. HR00112220038.






\bibliographystyle{unsrt}
\bibliography{main}

\appendix
\section{Category Theory Definitions}\label{sec:app}

\subsection{Categories, Functors, and Natural Transformations}\label{sec:ct-basics}

\begin{definition}
\label{def:cat}
    A \define{category} $\mathcal{C}$ consists of a collection of \emph{objects} $X,Y,Z,\dots$, and a collection of \emph{morphisms} (also called \emph{arrows}) $f,g,h, \dots$, together with the following data: 
    \begin{itemize}
        \item Each morphism has a specified \emph{domain} and \emph{codomain} object; the notation $f:X\rightarrow Y$ signifies that $f$ is a morphism with domain $X$ and codomain $Y$.
        \item Each object $X\in \mathcal{C}$ has a distinguished \emph{identity} morphism $\text{id}_X:X\rightarrow X$.
        \item For any triple of objects $X,Y,Z\in \mathcal{C}$ and any pair of morphisms $f:X\rightarrow Y$ and $g:Y\rightarrow Z$, there is a specified \emph{composite} morphism $g\circ f:X\rightarrow Z$.
    \end{itemize}
    This data is subject to the following two axioms:
    \begin{itemize}
        \item (unitality) For any morphism $f:X\rightarrow Y$, we have $f\circ \text{id}_X = f = \text{id}_Y\circ f$,
        \item (associativity) For any objects $W,X,Y,Z\in\mathcal{C}$ and morphisms $f:W\rightarrow X$, $g:X\rightarrow Y$, $h:Y\rightarrow Z$, we have $h\circ(g\circ f) = (h\circ g)\circ f$.
    \end{itemize}
\end{definition}

\begin{definition}
    Given a category $\mathcal{C}$, the opposite category, denoted $\mathcal{C}^{op}$, has 
    \begin{itemize}
        \item the same objects as $\mathcal{C}$,
        \item a morphism $f^{op}\maps Y \to X$ for every morphism $f\maps X \to Y$ in $\mathcal{C}$.
    \end{itemize}
    Identities are the same as those in $\mathcal{C}$ and composition is defined by composition in $\mathcal{C}$:
    \begin{equation}
        g^{op}\circ f^{op} = (f\circ g)^{op}.
    \end{equation}
\end{definition}

\begin{definition}\label{def:functor}
    Given categories $\mathcal{C}$ and $\mathcal{D}$, a (covariant) \define{functor} $F:\mathcal{C}\rightarrow \mathcal{D}$ consists of the following maps:
    \begin{itemize}
        \item an \define{object map} $c\mapsto F(c)$ for all $c\in \textnormal{Ob}(\mathcal{C})$,
        \item a \define{morphism map} respecting domains and codomains, i.e., $(f: c \to c') \mapsto (F(f): F(c) \to F(c'))$ for every morphism $f\maps c\to c'$ in $\mathcal{C}$.
    \end{itemize}
    These maps must satisfy the following equations:
    \begin{itemize}
        \item $F(\text{id}_c) = \text{id}_{F(c)}$ (identities are preserved),
        \item $F(g\circ f) = F(g)\circ F(f)$ (composition is preserved),
    \end{itemize}
    for all objects $c$ and composable morphisms $f,g$ in $\mathcal{C}$. A \define{contravariant functor} is a functor $F\maps \CC^{op}\to \DD$ whose domain is an opposite category.
\end{definition}

\begin{definition}
    Given a pair of functors $F,G\maps \CC\to \DD$ with common domain and codomain, a natural transformation $\theta$ from $F$ to $G$, denoted $\theta\maps F\Rightarrow G$ consists of a family of morphisms in $\DD$ indexed by objects in $\CC$, namely
    \begin{equation}
        \theta_X\maps F(X)\to G(X)
    \end{equation}
    for all $X\in \CC$. These are known as the \define{components} of $\theta$. These components are required to make the following \define{naturality squares} commute for all morphisms objects $X,Y$ and morphisms $f\maps X\to Y$ in $\CC$.
\[\begin{tikzcd}
	{F(X)} && {F(Y)} \\
	{G(X)} && {G(Y)}
	\arrow["{F(f)}", from=1-1, to=1-3]
	\arrow["{G(f)}"', from=2-1, to=2-3]
	\arrow["{\theta_X}"', from=1-1, to=2-1]
	\arrow["{\theta_Y}", from=1-3, to=2-3]
\end{tikzcd}\]
    When each component of $\theta$ is an isomorphism, we call $\theta$ a \define{natural isomorphism}.
\end{definition}

\subsection{Universal Constructions}\label{sec:limits}

One of the core benefits of category theory is the ability to study an object in terms of its relationships to other objects in a category. This line of reasoning leads to very general constructions that can be formulated in any category and prove useful when specialized to familiar categories.

\begin{definition}
    An \define{initial object} in a category $\mathcal{C}$ is an object $i\in\mathcal{C}$ such that for all objects $c\in\mathcal{C}$, there exists a unique morphism $i\to c$ in $\mathcal{C}$.
    Dually, a \define{terminal object} in $\mathcal{C}$ is an initial object in $\CC^{op}$, i.e., an object $t\in\mathcal{C}$ such that for all objects $c\in\mathcal{C}$, there exists a unique morphism $c\to t$ in $\mathcal{C}$.

\end{definition}
\begin{itemize}
    \item The initial object in $\Set$ is the empty set while the terminal object\footnote{It is easy to verify that initial and terminal objects (and more generally any objects defined by universal properties) are unique up to isomorphism, so we are justified in using the terminology of ``\emph{the} initial/terminal object"} in $\Set$ is the singleton set $\1$.
    \item The vector space $\R^0$ is both the initial and the terminal object in $\Vect$.
\end{itemize}

In the following definition, a \define{span} in a category simply refers to a pair of morphisms in that category with common domain. Dually, a \define{cospan} in a category is a pair of morphisms with shared codomain.

\begin{definition}
    The \define{product} of objects $X$ and $Y$ in a category $\mathcal{C}$ is a span of the form $X\xleftarrow{\pi_X} X\times Y\xrightarrow{\pi_Y} Y$ such that for any other span of the form $X\xleftarrow{f} Z\xrightarrow{g} Y$, there exists a unique morphism $Z\xrightarrow{!} X\times Y$ such that $\pi_X\circ ! = f$ and $\pi_Y\circ ! = g$. The unique morphism is called the \define{pairing} of $f$ and $g$ and denoted by $\langle f,g\rangle$.
Dually, the \define{coproduct} of objects $X$ and $Y$ in a category $\mathcal{C}$ is the product of $X$ and $Y$ in $\mathcal{C}^{op}$. This dual nature is clearly visualized by the following pair of diagrams, where the left diagram represents a product in $\mathcal{C}$ and the right diagram represents a coproduct in $\mathcal{C}$.
\[\begin{tikzcd}
	X &&&& X \\
	& {X\times Y} && Z && {X+Y} && Z \\
	Y &&&& Y
	\arrow["{\pi_X}", from=2-2, to=1-1]
	\arrow["{\pi_Y}"', from=2-2, to=3-1]
	\arrow[""{name=0, anchor=center, inner sep=0}, "f"', curve={height=6pt}, from=2-4, to=1-1]
	\arrow[""{name=1, anchor=center, inner sep=0}, "g", curve={height=-6pt}, from=2-4, to=3-1]
	\arrow["{\exists !}"', dashed, from=2-4, to=2-2]
	\arrow["{\iota_X}"', from=1-5, to=2-6]
	\arrow["{\iota_Y}", from=3-5, to=2-6]
	\arrow["{\exists !}", dashed, from=2-6, to=2-8]
	\arrow[""{name=2, anchor=center, inner sep=0}, "f", curve={height=-6pt}, from=1-5, to=2-8]
	\arrow[""{name=3, anchor=center, inner sep=0}, "g"', curve={height=6pt}, from=3-5, to=2-8]
\end{tikzcd}\]
The unique arrow $X+Y\to Z$ in the coproduct diagram is called the \define{copairing} of $f$ and $g$, and denoted $[f,g]$.
\end{definition}
\begin{itemize}
    \item The product in $\Set$ is the Cartesian product of sets together with the natural projection maps. The coproduct in $\Set$ is the disjoint union of sets together with the natural inclusion maps.
    \item The product in $\Vect$ is the Cartesian product of vector spaces while the coproduct is the direct sum of vector spaces. In finite dimensions, the product and direct sum are isomorphic, meaning $\Vect$ is a special type of category called a biproduct category.
\end{itemize}

\begin{definition}\label{def:push_pull}
    The \define{pullback} of a cospan $X\xrightarrow{f}Z\xleftarrow{g} Y$ in a category $\CC$ is a span $X\xleftarrow{p_X} P\xrightarrow{p_Y}Y$ such that $f\circ p_X = g\circ p_Y$ and for any other span $X\xleftarrow{q_X} Q\xrightarrow{q_Y}Y$, there exists a unique morphism $Q\to P$ making the following diagram commute.
\[\begin{tikzcd}
	Q \\
	& P & Y \\
	& X & Z
	\arrow["f"', from=3-2, to=3-3]
	\arrow["g", from=2-3, to=3-3]
	\arrow["{p_Y}", from=2-2, to=2-3]
	\arrow["{p_X}"', from=2-2, to=3-2]
	\arrow["{\exists!}"', dashed, from=1-1, to=2-2]
	\arrow["{q_Y}", curve={height=-12pt}, from=1-1, to=2-3]
	\arrow["{q_X}"', curve={height=12pt}, from=1-1, to=3-2]
\end{tikzcd}\]
    The pullback object $P$ is typically denoted $X\times_Z Y$.
    
    Dually, the \define{pushout} of a span $X\xleftarrow{f}Z\xrightarrow{g}Y$ is a cospan $X\xrightarrow{i_X} P\xleftarrow{i_Y}Y$ such that $i_X\circ f=i_Y\circ g$ and for any other cospan $X\xrightarrow{q_X} Q\xleftarrow{q_Y} Y$, there exists a unique morphism $P\to Q$ making the following diagram commute.
\[\begin{tikzcd}
	Q \\
	& P & Y \\
	& X & Z
	\arrow["f"', tail reversed, no head, from=3-2, to=3-3]
	\arrow["g", tail reversed, no head, from=2-3, to=3-3]
	\arrow["{i_Y}", tail reversed, no head, from=2-2, to=2-3]
	\arrow["{i_X}"', tail reversed, no head, from=2-2, to=3-2]
	\arrow["{\exists!}"', dashed, tail reversed, no head, from=1-1, to=2-2]
	\arrow["{q_Y}", curve={height=-12pt}, tail reversed, no head, from=1-1, to=2-3]
	\arrow["{q_X}"', curve={height=12pt}, tail reversed, no head, from=1-1, to=3-2]
\end{tikzcd}\]
    The pushout object $P$ is typically denoted $X+_Z Y$.
\end{definition}
\begin{itemize}
    \item The pullback of a cospan $X\xrightarrow{f} Z\xleftarrow{g} Y$ in $\Set$ is the set $X\times_Z Y\coloneqq \{(x,y)\in X\times Y \mid f(x)=g(y)\}$ together with the projections from $X\times Y$ restricted to $X\times_Z Y$.
    \item The pushout of a span $X\xleftarrow{f} Z\xrightarrow{g} Y$ in $\Set$ is the set $X+_Z Y\coloneqq X+Y/\sim$ where $\sim$ is the equivalence relation $f(z)\sim g(z)$ for all $z\in Z$. The pushout also includes the maps $i_X\maps X\to X+_Z Y$ and $i_Y\maps Y\to X+_Z Y$ making $i_X\circ f = i_Y\circ g$.
\end{itemize}

\subsection{Symmetric Monoidal Categories}\label{sec:smcs}

\begin{definition}\label{def:smc}
    A \define{symmetric monoidal category} consists of a category $\CC$ equipped with a functor $\otimes\maps\CC\times \CC\to \CC$ called the \define{monoidal product} and an object $1\in\CC$ called the \define{monoidal unit}, together with the following natural isomorphisms:
    \begin{itemize}
        \item (associator) $\alpha_{X,Y,Z}\maps (X\otimes Y)\otimes Z \cong X\otimes (Y\otimes Z)$,
        \item (left unitor) $\lambda_X\maps 1\otimes X\cong X$,
        \item (right unitor) $\rho_X\maps X\otimes 1\cong X$.
    \end{itemize}
    This data must satisfy two families of commutative diagrams called the triangle and pentagon identities which we omit for brevity.
\end{definition}

\begin{remark}
    When a category $\CC$ has finite products and a terminal object, it automatically forms a symmetric monoidal category by taking the monoidal product of morphisms $f\maps X\to Y$ and $g\maps W\to Z$ to be $\langle f\circ \pi_X, g\circ\pi_W\rangle\maps X\times W\to Y\times Z$ as in the following diagram:
\[\begin{tikzcd}
	& X && Y \\
	{X\times W} && {Y\times Z} \\
	& W && {Z.}
	\arrow["{\pi_X}", from=2-1, to=1-2]
	\arrow["{\pi_W}"', from=2-1, to=3-2]
	\arrow["{\pi_Y}"', from=2-3, to=1-4]
	\arrow["{\pi_Z}", from=2-3, to=3-4]
	\arrow["f", from=1-2, to=1-4]
	\arrow["g"', from=3-2, to=3-4]
	\arrow["{\langle f\circ\pi_X,g\circ\pi_W\rangle}", dashed, from=2-1, to=2-3]
\end{tikzcd}\]
    This is known as a \define{Cartesian monoidal category}. For example, $(\Set,\times, 1)$ is cartesian monoidal.

    Dually, when a category has finite coproducts and an initial object, it automatically forms a symmetric monoidal category by taking the monoidal product of morphisms $f\maps X\to Y$ and $g\maps W\to Z$ to be $[\iota_Y\circ f,\iota_Z\circ g]$ as in the following diagram:
\[\begin{tikzcd}
	X && Y \\
	& {X+W} && {Y+Z.} \\
	W && {Z}
	\arrow["{\iota_X}"', from=1-1, to=2-2]
	\arrow["{\iota_W}", from=3-1, to=2-2]
	\arrow["{[\iota_Y\circ f,\iota_Z\circ g]}", dashed, from=2-2, to=2-4]
	\arrow["f", from=1-1, to=1-3]
	\arrow["g"', from=3-1, to=3-3]
	\arrow["{\iota_Y}", from=1-3, to=2-4]
	\arrow["{\iota_Z}"', from=3-3, to=2-4]
\end{tikzcd}\]
    This is known as a \define{co-Cartesian monoidal category}. For example, $(\FinSet, +, \emptyset)$ is co-Cartesian monoidal.
\end{remark}

\begin{definition}\label{smf}
    Given symmetric monoidal categories $(\CC, \otimes, 1)$ and $(\DD, \boxtimes, e)$, a \define{lax symmetric monoidal functor} $(\CC, \otimes, 1)\to (\DD,\boxtimes, e)$ consists of a functor $F\maps \CC\to \DD$ together with the following \define{comparison maps}:
    \begin{itemize}
        \item a morphism $\varphi_0\maps e\to F(1)$ called the \define{unit comparison},
        \item a natural transformation called the \define{product comparison} with components
        \begin{equation}
            \varphi_{X,Y}\maps F(X)\boxtimes F(Y)\to F(X\otimes Y)
        \end{equation}
        for all $X,Y\in \CC$.
    \end{itemize}
    This data must satisfy the following axioms.
    \begin{enumerate}
        \item (associativity) For all objects $X,Y,Z\in\CC$, the following diagram commutes.
\[\begin{tikzcd}
	{(F(X)\boxtimes F(Y))\boxtimes F(Z)} && {F(X)\boxtimes (F(Y)\boxtimes F(Z))} \\
	{F(X\otimes Y)\boxtimes F(Z)} && {F(X)\boxtimes F(Y\otimes Z)} \\
	{F((X\otimes Y)\otimes Z)} && {F(X\otimes (Y\otimes Z))}
	\arrow["{\alpha^\DD_{FX,FY,FZ}}", from=1-1, to=1-3]
	\arrow["{\varphi_{X,Y}\boxtimes\id}"', from=1-1, to=2-1]
	\arrow["{\varphi_{X\otimes Y, Z}}"', from=2-1, to=3-1]
	\arrow["{F(\alpha^\CC_{X,Y,Z})}"', from=3-1, to=3-3]
	\arrow["{\id\otimes\varphi_{Y,Z}}", from=1-3, to=2-3]
	\arrow["{\varphi_{X,Y\otimes Z}}", from=2-3, to=3-3]
\end{tikzcd}\]
        \item (unitality) For all $X\in \CC$, the following pair of diagrams commutes.
\[\begin{tikzcd}
	{e\boxtimes F(X)} && {F(1)\boxtimes F(X)} && {F(X)\boxtimes e} && {F(X)\boxtimes F(1)} \\
	{F(X)} && {F(1\otimes X)} && {F(X)} && {F(X\otimes 1)}
	\arrow["{\lambda_{F(X)}^\DD}"', from=1-1, to=2-1]
	\arrow["{\varphi_0\boxtimes\id}", from=1-1, to=1-3]
	\arrow["{\varphi_{1,X}}", from=1-3, to=2-3]
	\arrow["{F(\lambda_X^\CC)}", from=2-3, to=2-1]
	\arrow["{\id\boxtimes\varphi_0}", from=1-5, to=1-7]
	\arrow["{\rho_{F(X)}^\DD}"', from=1-5, to=2-5]
	\arrow["{\varphi_{X,1}}", from=1-7, to=2-7]
	\arrow["{F(\rho^\CC_X)}", from=2-7, to=2-5]
\end{tikzcd}\]
        \item (symmetry) For all $X,Y\in \CC$, the following diagram commutes.
\[\begin{tikzcd}
	{F(X)\boxtimes F(Y)} && {F(Y)\boxtimes F(X)} \\
	{F(X\otimes Y)} && {F(Y\otimes X)}
	\arrow["{\sigma_{FX,FY}^\DD}", from=1-1, to=1-3]
	\arrow["{\varphi_{X,Y}}"', from=1-1, to=2-1]
	\arrow["{\varphi_{Y,X}}", from=1-3, to=2-3]
	\arrow["{F(\sigma_{X,Y}^\CC)}"', from=2-1, to=2-3]
\end{tikzcd}\]
    \end{enumerate}
    We write such a functor as a pair $(F,\varphi)$. When the comparison maps are isomorphisms, $(F,\varphi)$ is called a \define{strong} symmetric monoidal functor.
\end{definition}

\begin{definition}\label{def:mon-nat-t}
    Given a pair of lax symmetric monoidal functors $(F,\varphi^F),(G,\varphi^G)\maps (\CC,\otimes, 1)\to (\DD,\boxtimes, e)$, a \define{monoidal natural transformation} from $(F,\varphi^F)$ to $(G,\varphi^G)$ is a natural transformation $\theta\maps F\Rightarrow G$ which also makes the following diagrams commute for all $X,Y\in \CC$.
\[\begin{tikzcd}
	{F(X)\boxtimes F(Y)} && {G(X)\boxtimes G(Y)} && e && {F(1)} \\
	{F(X\otimes Y)} && {G(X\otimes Y)} &&&& {G(1)}
	\arrow["{\varphi^F_{X,Y}}"', from=1-1, to=2-1]
	\arrow["{\varphi^G_{X,Y}}", from=1-3, to=2-3]
	\arrow["{\theta_X\boxtimes\theta_Y}", from=1-1, to=1-3]
	\arrow["{\theta_{X\otimes Y}}"', from=2-1, to=2-3]
	\arrow["{\varphi^F_0}", from=1-5, to=1-7]
	\arrow["{\theta_1}", from=1-7, to=2-7]
	\arrow["{\varphi^G_0}"', from=1-5, to=2-7]
\end{tikzcd}\]
\end{definition}

\end{document}